\documentclass{amsart}

\usepackage{graphicx,epsfig}

\usepackage{amsmath,amssymb,latexsym, amsfonts, amscd, amsthm, xy}

\usepackage{url}

\usepackage{hyperref}

\input{xy}
\xyoption{all}

\usepackage[usenames]{color}

\makeindex 

=630pt \headheight = 12pt \headsep = 20pt

\theoremstyle{plain}
\newtheorem{theorem}{Theorem}[section]
\newtheorem{conjecture}[theorem]{Conjecture}
\newtheorem{proposition}[theorem]{Proposition}
\newtheorem{corollary}[theorem]{Corollary}

\newtheorem{lemma}[theorem]{Lemma}

\newtheorem{problem}[theorem]{Problem}

\theoremstyle{definition}
\newtheorem{definition}[theorem]{Definition}
\newtheorem{remark}[theorem]{Remark}

\def\bA{\mathbb{A}}

\def\bF{\mathbb{F}}

\def\bP{\mathbb{P}}
\def\bQ{\mathbb{Q}}
\def\bZ{\mathbb{Z}}
\def\bR{\mathbb{R}}

\def\A{\mathbf{A}}
\def\B{\mathbf{B}}
\def\C{\mathbf{C}}
\def\D{\mathbf{D}}
\def\E{\mathbf{E}}
\def\F{\mathbf{F}}
\def\G{\mathbf{G}}
\def\H{\mathbf{H}}
\def\I{\mathbf{I}}
\def\J{\mathbf{J}}

\def\S{\mathbf{S}}

\def\cA{\mathcal{A}}
\def\cB{\mathcal{B}}
\def\cC{\mathcal{C}}
\def\cD{\mathcal{D}}
\def\cE{\mathcal{E}}
\def\cF{\mathcal{F}}

\def\cP{\mathcal{P}}

\def\cS{\mathcal{S}}

\def\cX{\mathcal{X}}

\def\cV{\mathcal{V}}
\def\cT{\mathcal{T}}

\def\fC{\mathfrak{C}}

\def\Br{\mathrm{Br}}

\begin{document}

\title{Generalized Mordell curves, generalized Fermat curves, and the Hasse principle}

\author{Nguyen Ngoc Dong Quan}

\date{December 12, 2012}

\address{Department of Mathematics \\
         University of British Columbia \\
         Vancouver, British Columbia \\
         V6T 1Z2, Canada}

\email{\href{mailto:dongquan@math.ubc.ca}{\tt dongquan@math.ubc.ca}, \href{mailto:dongquan.nndq@gmail.com}{\tt dongquan.nndq@gmail.com}}

\subjclass[2010]{Primary 14G05, 11G35, 11G30}

\keywords{Azumaya algebras, Brauer groups, Brauer-Manin obstruction, Hasse principle, generalized Mordell curves, generalized Fermat curves}

\maketitle

\tableofcontents

\begin{abstract}

A generalized Mordell curve of degree $n \ge 3$ over $\bQ$ is the smooth projective model of the affine curve of the form $Az^2 = Bx^n + C$, where $A, B, C$ are nonzero integers. A generalized Fermat curve of signature $(n, n, n)$ with $n \ge 3$ over $\bQ$ is the smooth projective curve of the form $Ax^n + By^n + Cz^n = 0$ for some nonzero integers $A, B, C$. In this paper, we show that for each prime $p$ with $p \equiv 1 \pmod{8}$ and $p \equiv 2 \pmod{3}$, there exists a threefold $\cX_p \subseteq \bP^6$ such that certain rational points on $\cX_p$ produce infinite families of non-isomorphic generalized Mordell curves of degree $12n$ and infinite families of generalized Fermat curves of signature $(12n, 12n, 12n)$ for each $n \ge 2$ that are counterexamples to the Hasse principle explained by the Brauer-Manin obstruction. We also show that the set of special rational points on $\cX_p$ producing generalized Mordell curves and generalized Fermat curves that are counterexamples to the Hasse principle is infinite, and can be constructed explicitly.

\end{abstract}

\section{Introduction}
\label{introduction-section}

Faltings' theorem, n\'ee the Mordell conjecture states that a smooth geometrically irreducible curve over a number field has only finitely many rational points. Despite of this celebrated result, the following open problem remains widely open: For a family of smooth geometrically irreducible curves $\cC$ over a number field $\bQ$, determine the sets $\cC(\bQ)$ of all rational points on the curves $\cC$. This problem remains open even in the case where we assume further that the family of curves $\cC$ under consideration only consists of \textit{generalized Mordell curves} and \textit{generalized Fermat curves}.

For a positive integer $n \ge 3$, a \textit{generalized Mordell curve of degree $n$} over $\bQ$ is the smooth projective model of the affine curve defined by
\begin{align*}
Az^2 = Bx^n + C,
\end{align*}
where $A, B, C$ are nonzero integers. Although the defining equations of generalized Mordell curves are simple-looking, the problem of finding all rational points on the family of all generalized Mordell curve of arbitrary degree remains widely open. There are several authors studying the arithmetic of generalized Mordell curves. In 2004, Bennett and Skinner \cite{bennett-skinner} developed techniques for finding all rational points on certain generalized Mordell curves of the form $Az^2 = Bx^n + C$ for certain choices of integers $A, B, C$. Their techniques rely on the theory of elliptic curves, Galois representations, and modular forms, and was previously used by Wiles in his celebrated proof of Fermat's last theorem (see \cite{wiles}). In \cite{ivorra-kraus}, Ivorra and Kraus studied the arithmetic of certain generalized Mordell curves of the shape $Az^2 = Bx^p + C$, where $p$ is an odd prime.

For a positive integer $n \ge 3$, a \textit{generalized Fermat curve of signature $(n, n, n)$ over $\bQ$} is the projective curve defined by
\begin{align*}
Ax^n + By^n + Cz^n = 0,
\end{align*}
where $A, B, C$ are nonzero integers such that $\gcd(A, B, C) = 1$. There are several authors studying the arithmetic of generalized Fermat curves; for example, Selmer \cite{selmer} constructed the generalized Fermat curve of signature $(3, 3, 3)$ defined by
\begin{align*}
3x^3 + 4y^3 + 5z^3 = 0,
\end{align*}
which is the first cubic curve of genus one violating the Hasse principle. In \cite{wiles}, Wiles proved Fermat's last theorem, which says that the only rational points on the Fermat curve $x^n + y^n = z^n$ with $n \ge 3$ are the trivial ones.

In this paper, we are concerned with studying nonexistence of rational points on certain families of generalized Mordell curves of degree $12n$ with $n \ge 1$ and of generalized Fermat curves of signature $(12n, 12n, 12n)$ with $n \ge 1$. The basic technique used in this paper is the Brauer-Manin obstruction. It is also crucial that we need to obtain information about rational points on certain higher dimensional varieties to determine the arithmetic of certain families of curves that we consider throughout this work. To add interest to the arithmetic of these curves, we require that they have points over each local field $\bQ_p$ for each prime $p$ including $p = \infty$, and hence these families of curves are counterexamples to the \textit{Hasse principle}.

Let us digress for a moment to review some basic notions in the Brauer-Manin obstruction. Recall that the \textit{Hasse reciprocity law} (see \cite{skorobogatov}) states that the sequence of abelian groups
\begin{align*}
0 \rightarrow \Br(\bQ) \rightarrow \oplus_p\Br(\bQ_p) \rightarrow \bQ/\bZ \rightarrow 0
\end{align*}
is exact, where for each scheme $\cX$, we denote by $\Br(\cX)$ the Brauer group of $\cX$ and for a commutative ring $A$, define
\begin{align*}
\Br(A) := \Br(\text{Spec}(A)).
\end{align*}
In $1970$, Manin \cite{manin}, based on the Hasse reciprocity law, introduced the notion of the Brauer-Manin obstruction. Roughly speaking, the Brauer-Manin obstruction measures how badly the Hasse principle for varieties fails.

Take a smooth geometrically irreducible curve $\cC$ defined over $\bQ$, and let $\bA_{\bQ}$ be the ring of rational adeles. Let $\cC(\bA_{\bQ})$ denote the set of adelic points on $\cC$. It is well-known that
\begin{align*}
\cC(\bA_{\bQ})= \prod_{p}\cC(\bQ_p).
\end{align*}
Manin \cite{manin} introduced a subset of $\cC(\bA_{\bQ})$, say $\cC(\bA_{\bQ})^{\Br}$, such that
\begin{align*}
\cC(\bQ) \subseteq \cC(\bA_{\bQ})^{\Br} \subseteq \cC(\bA_{\bQ}).
\end{align*}
Here $\cC(\bA_{\bQ})^{\Br}$ is defined to be the right kernel of the \textit{adelic Brauer-Manin pairing} (see \cite{skorobogatov})
\begin{align}
\label{Pairing-the-adelic-Brauer-Manin-pairing}
\cE : \Br(\cC) \times \cC(\bA_{\bQ}) &\longrightarrow \bQ/\bZ \\
                   (\cA, (P_p)_{p})    &\mapsto \sum_{p}\text{inv}_{p}\left(\cA(P_p)\right), \nonumber
\end{align}
where for each prime $p$, $\text{inv}_p : \text{Br}(\bQ_p) \longrightarrow \bQ/\bZ$ is the invariant map of local class field theory. We say that $\cC$ \textit{satisfies the Hasse principle} if the following holds: $\cC(\bQ) \ne \emptyset$ if and only if $\cC(\bA_{\bQ}) \ne \emptyset$. We say that $\cC$ is a \textit{counterexample to the Hasse principle} if $\cC(\bQ) = \emptyset$ but $\cC(\bA_{\bQ}) \ne \emptyset$. Furthermore, we say that $\cC$ is a \textit{counterexample to the Hasse principle explained by the Brauer-Manin obstruction} if $\cC$ is a counterexample to the Hasse principle and satisfies $\cC(\bA_{\bQ})^{\Br} = \emptyset$.

Let $p$ be a prime such that $p \equiv 1 \pmod{8}$. Here, as throughout this paper, we denote by $\cX_p$ the threefold in $\bP^6_{\bQ}$ defined by
\begin{align}
\label{threefold-Xp-equations}
\cX_p :
\begin{cases}
b^2 - c^2 + 2pef &= 0 \\
2ab - 2cd + pf^2 &= 0 \\
a^2 - d^2 + pg^2 & = 0
\end{cases}
.
\end{align}

We now introduce certain rational points on $\cX_p$ which are of great interest throughout this paper.

\begin{definition}
\label{septuple-satisfies-hypothesis-FM-definition}

Let $p$ be a prime such that $p \equiv 1 \pmod{8}$. Let $n$ be a positive integer such that $n \ge 1$. Let $(A, B, C, D, E, F, G) \in \bZ^7$ be a septuple of integers such that at least one of them is non-zero. We say that $(A, B, C, D, E, F, G)$ \textit{satisfies Hypothesis FM with respect to the couple $(p, n)$} if the following are true:
\begin{itemize}

\item [(A1)] the point $\cP := (a : b : c : d : e : f : g) = (A : B : C : D : E : F : G)$ belongs to $\cX_p(\bQ)$.

\item [(A2)] let $l$ be any odd prime such that $\gcd(l, 3) = \gcd(l, p) = 1$ and $l$ divides $E$. Then $p$ is a square in $\bQ_l^{\times}$ or $v_l(E) - v_l(G) < 6n$.

\item [(A3)] $\gcd(A, D, G) = 1$, $E \not\equiv 0 \pmod{p}$ and $G  \not\equiv 0 \pmod{p}$.

\item [(A4)] let $l$ be any odd prime such that $\gcd(l, 3) = \gcd(l, p) = 1$ and
\begin{align*}
\gcd\left(AC - BD, DE - CF, AE - BF\right) \equiv 0 \pmod{l}.
\end{align*}
Then $p$ is a square in $\bQ_l^{\times}$.

\item [(A5)] there exists an integer $H$ such that $G - EH^{6} \equiv 0 \pmod{p}$ and $A + \zeta BH^{4}$ is a quadratic non-residue in $\bF_p^{\times}$ for any cube root of unity $\zeta$ in $\bF_p^{\times}$.

\end{itemize}
Moreover, if $3$ is a quadratic non-residue in $\bF_p^{\times}$, we further assume that the following are true:
\begin{itemize}

\item [(A6)] $v_3(E) - v_3(G) < 6n$

\item [(A7)] $A + B \not\equiv 0 \pmod{3}$ and $G \equiv 0 \pmod{3}$.

\end{itemize}

\end{definition}

\begin{remark}
\label{3-is-not-a-square-in-Fp-remark}

Since $p \equiv 1 \pmod{8}$, it follows from the quadratic reciprocity law that $-3$ is not a square in $\bF_p^{\times}$ if and only if $p$ is not a square in $\bF_3^{\times}$, or equivalently $p \equiv 2 \pmod{3}$. Hence if $p \equiv 2 \pmod{3}$, then the group of all cube roots of unity in $\bF_p^{\times}$ is trivial. Thus $(A5)$ is tantamount to the condition that there exists an integer $H$ such that $G - EH^{6} \equiv 0 \pmod{p}$ and $A + BH^{4}$ is a quadratic non-residue in $\bF_p^{\times}$.

\end{remark}

\begin{remark}
\label{non-zero-of-A-D-E-G-remark}

Note that Hypothesis FM implies that $A, D, E$ and $G$ are nonzero. Indeed, by $(A3)$, we know that $E, G \not\equiv 0 \pmod{p}$, and hence $E$, $G$ are nonzero. Assume that $A = 0$. Then, by $(A1)$ and the third equation of $(\ref{threefold-Xp-equations})$, we see that $D^2 = pG^2$. Hence $p = \left(\tfrac{D}{G}\right)^2$, which is a contradiction since $p$ is a prime. Hence $A \ne 0$. Assume that $D = 0$. Then, by $(A1)$ and the third equation of $(\ref{threefold-Xp-equations})$, we deduce that $A^2 + pG^2 = 0$, which is a contradiction since $A^2 + pG^2 > 0$. Hence $D \ne 0$.

\end{remark}

For a prime $p \equiv 1 \pmod{8}$, a positive integer $n$, and each septuple $(A, B, C, D, E, F, G)$ satisfying Hypothesis FM with respect to $(p, n)$, we will produce an Azumaya algebra $\cA$ on the generalized Mordell curve $\cC$ of the shape $pz^2 = E^2x^{12n} - G^2$. The construction of the Azumaya algebra $\cA$ relies mainly on the equations of the threefold $\cX_p$. Using conditions $(A1)-(A7)$, we will show that $\cA$ satisfies
\begin{align*}
\text{inv}_l(\cA(P_l)) =
\begin{cases}
0  \; &\text{if $l \ne p$,}
\\
1/2 \; &\text{if $l = p$}
\end{cases}
\end{align*}
for any $P_l \in \cC(\bQ_l)$. From these invariants of the Azumaya algebra $\cA$, it follows that $\cC(\bA_{\bQ})^{\Br} = \emptyset$, and thus the generalized Mordell curve $\cC$ of degree $12n$ has no rational points. More precisely, we will prove the following.

\begin{theorem}
\label{Theorem-1-in-introduction-section}
(see Theorem \ref{non-existence-of-rational-points-for-generalized-mordell-curves-using-rational-points-on-threefold-Xp-theorem})

Let $p$ be a prime such that $p \equiv 1 \pmod{8}$, and let $n$ be a positive integer. Let $(A, B, C, D, E, F, G)$ be a septuple of integers satisfying Hypothesis FM with respect to $(p, n)$. Then the generalized Mordell curve $\cC$ defined by
\begin{align*}
\cC : pz^2 = E^2x^{12n} - G^2
\end{align*}
satisfies $\cC(\bA_{\bQ})^{\Br} = \emptyset$.

\end{theorem}

The basic technique that we exploit in the proof of Theorem \ref{Theorem-1-in-introduction-section} is based upon the Brauer-Manin obstruction, and thus our method of studying the arithmetic of generalized Mordell curves is completely different from that appeared in \cite{bennett-skinner} and \cite{ivorra-kraus} in which the authors mainly used the modularity method. Moreover, in order to use the Brauer-Manin obstruction to show that certain generalized Mordell curves possess no rational points, it is crucial that we need to obtain information about the set of rational points on $\cX_p$ satisfying Hypothesis FM. The existence of rational points on the threefolds $\cX_p$ satisfying Hypothesis FM that proves nonexistence of rational points on certain generalized Mordell curves suggests that there exist certain higher dimensional varieties which play an important role in studying the arithmetic of certain curves. In this work, we introduce threefolds $\cX_p$ for each prime $p \equiv 1 \pmod{8}$, and show that the existence of a subset of $\cX_p(\bQ)$ consisting of rational points on $\cX_p$ that satisfy Hypothesis FM is equivalent to the existence of certain families of generalized Mordell curves and of generalized Fermat curves having no rational points. Thus the arithmetic of threefolds $\cX_p$ determine the arithmetic of a special class of curves including certain generalized Mordell curves and certain generalized Fermat curves.

Let $p$ be a prime such that $p \equiv 1 \pmod{8}$ and $p \equiv 2 \pmod{3}$, and let $n$ be a positive integer such that $n \ge 2$. In Section \ref{a-subset-of-the-set-of-rational-points-on-X-p-section}, we will describe an infinite subset of $\cX_p(\bQ)$ that consists of rational points on $\cX_p$ satisfying Hypothesis FM with respect to $(p, n)$. This subset is parameterized by three parameters $\alpha, \beta, \kappa$. For some choice of $\kappa$, and by Theorem \ref{Theorem-1-in-introduction-section}, we prove that there exist infinitely many non-isomorphic generalized Mordell curves of the shape $pz^2 = 3^6\kappa^6x^{12n} - 1$ that are counterexamples to the Hasse principle explained by the Brauer-Manin obstruction. This result is the combination of Theorem \ref{the-first-infinite-family-of-generalized-Mordell-curves-violates-the-Hasse-principle-theorem} and Corollary \ref{infinitude-of-non-isomorphic-generalized-Mordell-curves-violating-the-Hasse-principle-corollary}.

For any positive integer $m$ such that $2m < n$, we will show that there are infinitely many integers $\kappa$ such that the generalized Mordell curve $\cD_{(p, m)}^{(\kappa)}$ has at least two rational points in its affine locus whereas the generalized Mordell curve $\cD_{(p, n)}^{(\kappa)}$ is a counterexample to the Hasse principle explained by the Brauer-Manin obstruction, where $\cD_{(p, m)}^{(\kappa)}$ and $\cD_{(p, m)}^{(\kappa)}$ are defined by
\begin{align*}
\cD_{(p, m)}^{(\kappa)} : pz^2 = 3^6\kappa^6x^{12m} - 1
\end{align*}
and
\begin{align*}
\cD_{(p, n)}^{(\kappa)} : pz^2 = 3^6\kappa^6x^{12n} - 1,
\end{align*}
respectively. This is Lemma \ref{the-existence-of-the-set-C-n-m-such-that-kappa-in-C-n-m-generates-the-sequence-satisfying-the-DCC-lemma} in Section \ref{the-descending-chain-condition-on-sequences-of-curves-section}. This result also shows that the degree $12n$ of the generalized Mordell curves in Theorem \ref{Theorem-1-in-introduction-section} is \textit{optimal} in the sense that one can not replace $n$ by any positive divisor $s$ of $n$ with $s \ne n$. This will be explained in more detail in Remark \ref{Remark-the-degree-12n-is-optimal-for-generalized-Mordell-curves-D}.

In Section \ref{the-descending-chain-condition-on-sequences-of-curves-section}, motivated by Lemma \ref{the-existence-of-the-set-C-n-m-such-that-kappa-in-C-n-m-generates-the-sequence-satisfying-the-DCC-lemma}, we introduce a notion of the \textit{descending chain condition} (DCC) on sequences of curves, which is an analogue of the notion of the descending chain condition on partially ordered sets. Let $(\cX_n, \psi_n)_{n \ge 1}$ be a sequence of curves defined over a global field $k$, where for each $n \ge 1$, $\cX_n$ is a curve of genus $g_n$ defined over $k$, and for each $n \ge 1$, $\psi_n : \cX_{n + 1} \rightarrow \cX_n$ is a $k$-morphism of curves. The sequence $(\cX_n, \psi_n)_{n \ge 1}$ is said to satisfy the DCC of length $h$ if $\cX_n(k) \ne \emptyset$ for each $1 \le n \le h - 1$, and $\cX_h$ is a counterexample to the Hasse principle explained by the Brauer-Manin obstruction. To rule out certain trivial cases, we assume that $g_n > g_m$ for any positive integers $n > m$. It is not difficult to see that there are infinitely many sequences of curves that do not satisfy the DCC of any length. However, constructing sequences of curves satisfying the DCC of arbitrary length seems a nontrivial problem. In Section \ref{the-descending-chain-condition-on-sequences-of-curves-section}, we will show that there exist infinitely many sequences of generalized Mordell curves that satisfy the DCC of arbitrary length, and hence these sequences provide a nontrivial class of sequences of curves satisfying the DCC. This will be proved in Corollary \ref{the-existence-of-infinitude-of-sequences-of-curves-satisfying-DCC-corollary}.

Sections \ref{certain-generalized-Fermat-curves-violate-the-Hasse-principle-section} and \ref{infinitude-of-the-quadruples-p-n-kappa-chi-section} are motivated by the following problem, which was first studied by Halberstadt and Kraus \cite{halberstadt-kraus}.

\begin{problem}
\label{Problem-Halberstadt-Kraus}
(see \cite[Probl\`eme 3 and Probl\`eme 4]{halberstadt-kraus})

For each prime $p \ge 3$, does there exist an explicit generalized Fermat curve $\cC$ of signature $(p, p, p)$ defined over $\bQ$ of the shape $Ax^p + By^p + Cz^p = 0$ that is a counterexample to the Hasse principle?

\end{problem}

As was mentioned before, Selmer \cite{selmer} constructed the generalized Mordell curve defined by $3x^3 + 4y^3 + 5z^3 = 0$ that violates the Hasse principle, and hence Problem \ref{Problem-Halberstadt-Kraus} holds for $p = 3$. For each prime $p$ with $5 \le p \le 100$, Halberstadt and Kraus \cite{halberstadt-kraus} constructed an explicit generalized Fermat curve of signature $(p, p, p)$ that violates the Hasse principle. Hence Problem \ref{Problem-Halberstadt-Kraus} holds for each prime $p$ between $5$ and $100$. Furthermore, assuming the abc conjecture, Halberstadt and Kraus \cite{halberstadt-kraus} proved that Problem \ref{Problem-Halberstadt-Kraus} holds for every prime $p \ge 5$.

In Problem \ref{Problem-Halberstadt-Kraus}, the reason why we only consider certain generalized Fermat curves of signature $(p, p, p)$ with $p$ prime is that we want to rule out certain trivial cases, and that we do not want $p$ to be replaced by any positive divisor $s$ of $p$ with $s \ne p$. Motivated by this observation, we generalize Problem \ref{Problem-Halberstadt-Kraus}, and study the following problem.

\begin{problem}
\label{Problem-DongQuan}

For each positive integer $n \ge 3$, does there exist an explicit generalized Fermat curve $\cC_n$ of signature $(n, n, n)$ defined over $\bQ$ of the shape $Ax^n + By^n + Cz^n = 0$ with $\gcd(A, B, C) = 1$ that is a counterexample to the Hasse principle explained by the Brauer-Manin obstruction such that for any positive divisor $s$ of $n$ with $s \ne n$, the generalized Fermat curve $\cC_s$ of signature $(s, s, s)$ defined by $Ax^s + By^s + Cz^s = 0$ has at least one rational point?

\end{problem}

Upon assuming that $n$ is a prime, we see that Problem \ref{Problem-Halberstadt-Kraus} is a special case of Problem \ref{Problem-DongQuan}. In Sections \ref{certain-generalized-Fermat-curves-violate-the-Hasse-principle-section} and \ref{infinitude-of-the-quadruples-p-n-kappa-chi-section}, we answer Problem \ref{Problem-DongQuan} in the affirmative for any positive integer $n \equiv 0 \pmod{12}$ with $n \ge 24$. In fact, we prove that for each positive integer $n \equiv 0 \pmod{12}$ with $n \ne 12$, there are infinitely many generalized Fermat curves $\cC_n$ of signature $(n, n, n)$ satisfying Problem \ref{Problem-DongQuan}. This result is the combination of Theorem \ref{the-first-family-of-certain-generalized-Fermat-curves-violating-the-Hasse-principle-theorem} and Lemma \ref{infinitude-of-the-quadruples-p-n-kappa-chi-lemma}. As a consequence, in Section \ref{the-DCC-on-sequences-of-generalized-Fermat-curves-section}, we show that there exist infinitely many sequences of generalized Fermat curves that satisfy the DCC of any length.

In the last section, we describe another subset of $\cX_p(\bQ)$, and prove that upon assuming Schinzel's Hypothesis H (which will be reviewed in the same section), there should exist another infinite subset of $\cX_p(\bQ)$ consisting of rational points on $\cX_p$ that satisfy Hypothesis FM with respect to $(p, n)$, where $p$ is a prime such that $p \equiv 1 \pmod{8}$ and $p \equiv 2 \pmod{3}$ and $n$ is a positive integer. Using the same arguments as in Sections \ref{certain-generalized-Mordell-curves-violate-the-Hasse-principle-section} and \ref{certain-generalized-Fermat-curves-violate-the-Hasse-principle-section}, there should exist certain families of generalized Mordell curves and of generalized Fermat curves that are counterexamples to the Hasse principle explained by the Brauer-Manin obstruction. Since we will construct infinitely many generalized Mordell curves and infinitely many Fermat curves violating the Hasse principle in Sections \ref{certain-generalized-Mordell-curves-violate-the-Hasse-principle-section} and \ref{certain-generalized-Fermat-curves-violate-the-Hasse-principle-section}, we will restrict ourselves to only describing another infinite subset of $\cX_p(\bQ)$ in the last section that should produce infinitely many septuples satisfying Hypothesis FM, but not proceeding to construct certain generalized Mordell curves and certain generalized Fermat curves violating the Hasse principle that arise from this subset.

\section{Non-existence of rational points on certain generalized Mordell curves}
\label{nonexistence-of-rational-points-on-certain-generalized-mordell-curves-section}

In this section, we describe a relationship between rational points on $\cX_p$ that satisfy Hypothesis FM and nonexistence of rational points on certain generalized Mordell curves. More precisely, we show that for each prime $p \equiv 1 \pmod{8}$ and a positive integer $n$, the existence of a septuple $(A, B, C, D, E, F, G)$ satisfying Hypothesis FM with respect to the couple $(p, n)$ implies nonexistence of rational points on the generalized Mordell curve of the shape $pz^2 = E^2x^{12n} - G^2$. We begin by proving the main lemma in this section.

\begin{lemma}
\label{azumaya-algebra-A-for-generalized-mordell-curve-C-lemma}

Let $p$ be a prime such that $p \equiv 1 \pmod{8}$, and let $n$ be a positive integer. Assume that $(A, B, C, D, E, F, G) \in \bZ^7$ is a septuple of integers satisfying Hypothesis FM with respect to the couple $(p, n)$. Let $\cC$ be the generalized Mordell curve defined by
\begin{equation}
\label{generalized-mordell-curve-C-equation}
\cC : pz^2 = E^2x^{12n} - G^2.
\end{equation}
Let $\bQ(\cC)$ be the function field of $\cC$, and let $\cA$ be the class of the quaternion algebra $\left(p, A + Bx^{4n} + pz\right)$ in $\Br(\bQ(\cC))$. Then $\cA$ is an Azumaya algebra of $\cC$. Furthermore, $\cB := \left(p, A + Bx^{4n} - pz\right)$ and $\cE := \left(p, \dfrac{A + Bx^{4n} + pz}{x^{6n}}\right)$ represent the same class as $\cA$ in $\Br(\bQ(\cC))$.

\end{lemma}

\begin{proof}

We will prove that there is a Zariski open covering $(U_i)_i$ of $\cC$ such that $\cA$ extends to an element of $\text{Br}(U_i)$ for each $i$.

By $(A1)$, we see that $(\ref{generalized-mordell-curve-C-equation})$ can be written in the form
\begin{align}
\label{factorization-of-equation-of-generalized-mordell-curve-C}
(A + Bx^{4n} + pz)(A + Bx^{4n} - pz) &= (Cx^{4n} + D)^2 - px^{4n}(Ex^{4n} + F)^2 \\
                                     &= \text{Norm}_{\bQ(\sqrt{p})/\bQ}\left((Cx^{4n} + D) - \sqrt{p}x^{2n}(Ex^{4n} + F) \right).   \nonumber
\end{align}
It follows from the identity above that $\cA + \cB = 0$. Furthermore, we see that $\cA - \cE = (p, x^{6n}) = 0$. Since $\cA, \cB$ and $\cE$ belong to the $2$-torsion part of $\Br(\bQ(\cC))$, we deduce that $\cA = \cB = \cE$.

Let $U_1$ be the \textit{largest} open subvariety of $\cC$ in which the rational function $R_1 := A + Bx^{4n} + pz$ has neither a zero nor pole. Let $U_2$ be the \textit{largest} open subvariety of $\cC$ in which $R_2 := A + Bx^{4n} - pz$ has neither a zero nor pole. Since $\cA = \cB$, $\cA$ is an Azumaya algebra on $U_1$ and also on $U_2$. We prove that in the affine part of $\cC$, the locus where both of $R_1$ and $R_2$ have a zero is empty. Assume the contrary, and let $(X, Z)$ be a common zero of $R_1$ and $R_2$. We see that
\begin{align*}
A + BX^{4n} = R_1 + R_2 = 0
\end{align*}
and
\begin{align*}
Z = \dfrac{R_1 - R_2}{2p} = 0.
\end{align*}
This implies that $B \ne 0$; otherwise, we deduce that $A = 0$, which is a contradiction to Remark \ref{non-zero-of-A-D-E-G-remark}. Hence $B \ne 0$, and thus it follows that $X^{4n} = -\tfrac{A}{B}$. By $(\ref{generalized-mordell-curve-C-equation})$, we deduce that
\begin{align*}
X^{12n} = \dfrac{G^2}{E^2} = \left(-\dfrac{A}{B}\right)^3.
\end{align*}
Hence
\begin{align}
\label{relation-among-A-B-E-G-in-azumaya-algebra-equation}
(-A)^3 = \dfrac{B^3G^2}{E^2}.
\end{align}
Let $H$ be an integer satisfying $(A5)$ in Definition \ref{septuple-satisfies-hypothesis-FM-definition}. Since $E \not\equiv 0 \pmod{p}$, it follows from $(A5)$ and $(\ref{relation-among-A-B-E-G-in-azumaya-algebra-equation})$ that
\begin{align*}
(-A)^3 = \dfrac{B^3G^2}{E^2} \equiv (BH^4)^3 \pmod{p}.
\end{align*}
Hence $-A \equiv \zeta BH^4 \pmod{p}$ for some cube root of unity $\zeta$ in $\bF_p^{\times}$. Thus $A + \zeta BH^4 \equiv 0 \pmod{p}$, which is a contradiction to $(A5)$. Therefore, in the affine part of $\cC$, the locus where both of $R_1$ and $R_2$ have a zero is empty.

Let $R_3 := \dfrac{A + Bx^{4n} + pz}{x^{6n}}$, and let $\infty = (X_{\infty}: Y_{\infty}: Z_{\infty})$ be a point at infinity on $\cC$. We know that $Y_{\infty} = 0$ and
\begin{align*}
\dfrac{Z_{\infty}}{X_{\infty}^{6n}} = \pm \dfrac{E}{\sqrt{p}}.
\end{align*}
It follows that
\begin{align*}
R_3(\infty) = \dfrac{AY_{\infty}^{6n} + BX_{\infty}^{4n}Y_{\infty}^{2n} + pZ_{\infty}}{X_{\infty}^{6n}} = \dfrac{pZ_{\infty}}{X_{\infty}^{6n}} = \pm \sqrt{p}E.
\end{align*}
By Remark \ref{non-zero-of-A-D-E-G-remark}, we know that $E$ is nonzero. Hence $R_3 \ne 0$, and thus $R_3$ is regular and non-vanishing at the points at infinity on $\cC$.

Let $U_3$ be the \textit{largest} open subvariety of $\cC$ in which the rational function $R_3$ has neither a zero nor pole. Then, since $\cA = \cE$, we deduce that $\cA$ is an Azumaya algebra on $U_3$. By what we have shown, it follows that $\cC = U_1 \cup U_2 \cup U_3$. Since $\cA$ is an Azumaya algebra on each $U_i$ for $1 \le i \le 3$, we deduce that $\cA$ belongs to $\Br(\cC)$, proving our contention.

\end{proof}

We now prove the main theorem in this section, which relates rational points on $\cX_p$ satisfying Hypothesis FM to nonexistence of rational points on certain generalized Mordell curves.

\begin{theorem}
\label{non-existence-of-rational-points-for-generalized-mordell-curves-using-rational-points-on-threefold-Xp-theorem}

Let $p$ be a prime such that $p \equiv 1 \pmod{8}$, and let $n$ be a positive integer. Let $(A, B, C, D, E, F, G) \in \bZ^7$ be a septuple of integers satisfying Hypothesis FM with respect to $(p, n)$. Let $\cC$ be the generalized Mordell curve defined by $(\ref{generalized-mordell-curve-C-equation})$ in Lemma \ref{azumaya-algebra-A-for-generalized-mordell-curve-C-lemma}. Then $\cC(\bA_{\bQ})^{\Br} = \emptyset$.

\end{theorem}

\begin{proof}

We maintain the notation of Lemma \ref{azumaya-algebra-A-for-generalized-mordell-curve-C-lemma}. We will prove that for any $P_l \in \cC(\bQ_l)$,
\begin{align}
\label{invariant-value-of-A-equations}
 \text{inv}_l(\cA(P_l)) =
\begin{cases}
0  \; &\text{if $l \ne p$,}
\\
1/2 \; &\text{if $l = p$.}
\end{cases}
\end{align}
Since $\cC$ is smooth, we know that $\cC^{\ast}(\bQ_l)$ is $l$-adically dense in $\cC(\bQ_l)$, where $\cC^{\ast}$ is the affine curve given by $pz^2 = E^2x^{12n} - G^2$. Since $\text{inv}_l(\cA(P_l))$ is a continuous function on $\cC(\bQ_l)$ with respect to the $l$-adic topology, it suffices to prove $(\ref{invariant-value-of-A-equations})$ for any $P_l \in \cC^{\ast}(\bQ_l)$.

Suppose that $l = 2, \infty$ or $l$ is an odd prime such that $l \ne p$ and $p$ is a square in $\bQ_l^{\times}$. Then, for any $t \in \bQ_l^{\times}$, the local Hilbert symbol $(p, t)_l$ is $1$. Thus $\text{inv}_l(\cA(P_l))$ is $0$.

Suppose that $l = 3$. If $3$ is a square in $\bF_p^{\times}$, then by the quadratic reciprocity law, we know that $p$ is a square in $\bF_3^{\times}$. Hence, using the same arguments as above, we deduce that $\text{inv}_3(\cA(P_3)) = 0$. If $3$ is a quadratic non-residue in $\bF_p^{\times}$, then we contend that $v_3(x) \ge 0$. Assume the contrary, that is, $v_3(x) = \epsilon < 0$. Since $\epsilon = v_3(x) \in \bZ$, we see that $\epsilon \le -1$. Hence, by $(A6)$, we deduce that
\begin{align*}
v_3(E^2x^{12n}) = 2v_3(E) + 12n\epsilon \le 2v_3(E) - 12n < 2v_3(G) = v_3(G^2).
\end{align*}
It then follows that
\begin{align*}
2v_3(z) = v_3(pz^2) = \min\left(v_3(E^2x^{12n}), v_3(G^2)\right) = v_3(E^2x^{12n}) = 2v_3(E) + 12n\epsilon,
\end{align*}
and hence $v_3(z) = v_3(E) + 6n\epsilon$. Therefore there exist elements $x_0, z_0, E_0 \in \bZ_3^{\times}$ such that
\begin{align*}
x &= 3^{\epsilon}x_0, \\
z & = 3^{v_3(E) + 6n\epsilon}z_0,\\
E &= 3^{v_3(E)}E_0.
\end{align*}
By $(\ref{generalized-mordell-curve-C-equation})$, we deduce that
\begin{align*}
p3^{2v_3(E) + 12n\epsilon}z_0^2 = 3^{2v_3(E) + 12n\epsilon}E_0^2x_0^{12n} - G^2.
\end{align*}
Multiplying both sides of the above equation by $3^{-2v_3(E) - 12n\epsilon}$, we see that
\begin{align}
\label{x0-z0-equation-for-prime-3}
pz_0^2 = E_0^2x_0^{12n} - 3^{-2v_3(E) - 12n\epsilon}G^2.
\end{align}
We see that
\begin{align*}
v_3(3^{-2v_3(E) - 12n\epsilon}G^2) = 2v_3(G) - 2v_3(E) - 12n\epsilon &> -12n - 12n\epsilon \; \; (\text{by (A6)}) \\
                                                                   &= 12n(-\epsilon - 1) \ge 0.
\end{align*}
Hence $v_3(3^{-2v_3(E) - 12n\epsilon}G^2) > 0$. Since $v_3(3^{-2v_3(E) - 12n\epsilon}G^2)$ is an integer, it follows that
\begin{align*}
v_3(3^{-2v_3(E) - 12n\epsilon}G^2) \ge 1,
\end{align*}
and thus $3^{-2v_3(E) - 12n\epsilon}G^2 \in 3\bZ_3$. Reducing equation $(\ref{x0-z0-equation-for-prime-3})$ modulo $3$, we deduce that
\begin{align*}
p \equiv \left(\dfrac{E_0x_0^{6n}}{z_0}\right)^2 \pmod{3},
\end{align*}
which is a contradiction since $p$ is not a square in $\bF_3^{\times}$. This contradiction implies that $v_3(x) \ge 0$. By $(\ref{generalized-mordell-curve-C-equation})$, we see that $v_3(z) \ge 0$.

By $(A7)$ and $(\ref{generalized-mordell-curve-C-equation})$, we deduce that
\begin{equation*}
pz^2 = E^2x^{12n} - G^2 \equiv E^2x^{12n} \pmod{3}.
\end{equation*}
Since $p$ is not a square in $\bF_3^{\times}$, it follows that $z \equiv Ex \equiv 0 \pmod{3}$. Assume that $A + Bx^{4n} + pz \equiv 0 \pmod{3}$. Since $z \equiv 0 \pmod{3}$, it follows that $A + Bx^{4n} \equiv 0 \pmod{3}$. We see from $(\ref{factorization-of-equation-of-generalized-mordell-curve-C})$ that
\begin{equation*}
(Cx^{4n}  + D)^2 - px^{4n}(Ex^{4n} + F)^2 \equiv 0 \pmod{3}.
\end{equation*}
Since $p$ is not a square in $\bF_3^{\times}$, it follows that
\begin{align*}
Cx^{4n}  + D \equiv x(Ex^{4n} + F) \equiv 0 \pmod{3}.
\end{align*}
Thus we deduce that $A + Bx^{4n} \equiv Cx^{4n}  + D \equiv 0 \pmod{3}$. We contend that $x \not\equiv 0 \pmod{3}$; otherwise, $A \equiv D \equiv 0 \pmod{3}$, and hence it follows from $(A7)$ that $3$ divides $\gcd(A, D, G)$, which is a contradiction to $(A3)$. Thus $x \not\equiv 0 \pmod{3}$, and hence we deduce that $x^2 \equiv 1 \pmod{3}$. By $(A7)$, we see that
\begin{align*}
0 \equiv A + Bx^{4n} \equiv A + B \not\equiv 0 \pmod{3},
\end{align*}
which is a contradiction. Hence $A + Bx^{4n} + pz \not\equiv 0 \pmod{3}$, and thus we deduce that the local Hilbert symbol $(p, A + Bx^{4n} + pz)_3$ is $1$. Therefore $\text{inv}_3(\cA(P_3))$ is $0$.

Suppose that $l$ is an odd prime such that $\gcd(l, 3) = \gcd(l, p) = 1$ and $p$ is not a square in $\bQ_l^{\times}$. We consider the following two cases.

$\star$ \textit{Case 1. $v_l(x) \ge 0$.}

Assume that
\begin{align}
\label{x-z-equation-1-for-l-not-square}
\begin{cases}
A + Bx^{4n} + pz &\equiv 0 \pmod{l}, \\
A + Bx^{4n} - pz &\equiv 0 \pmod{l}.
\end{cases}
\end{align}
By $(\ref{factorization-of-equation-of-generalized-mordell-curve-C})$, we deduce that
\begin{equation*}
(Cx^{4n}  + D)^2 - px^{4n}(Ex^{4n} + F)^2 \equiv 0 \pmod{l}.
\end{equation*}
Since $p$ is not a square in $\bQ_l^{\times}$, it follows that
\begin{align}
\label{x-z-equation-2-for-l-not-square}
\begin{cases}
Cx^{4n} + D &\equiv 0 \pmod{l}, \\
x(Ex^{4n} + F) &\equiv 0 \pmod{l}.
\end{cases}
\end{align}
Adding both of the equations of $(\ref{x-z-equation-1-for-l-not-square})$, we deduce that\begin{align}
\label{x-equation-1-for-l-not-square}
A + Bx^{4n} \equiv 0 \pmod{l}.
\end{align}
If $x \equiv 0 \pmod{l}$, then it follows from $(\ref{x-equation-1-for-l-not-square})$ and the first equation of $(\ref{x-z-equation-2-for-l-not-square})$ that $A \equiv D \equiv 0 \pmod{l}$. By $(A1)$, we deduce that $pG^2 = D^2 - A^2 \equiv 0 \pmod{l}$. Since $p \ne l$, we deduce that $l$ divides $G$, and hence $l$ divides $\gcd(A, D, G)$, which is a contradiction to $(A3)$. If $x \not\equiv 0 \pmod{l}$, then it follows from the second equation of $(\ref{x-z-equation-2-for-l-not-square})$ that
\begin{align}
\label{x-equation-2-for-l-not-square}
Ex^{4n} + F \equiv 0 \pmod{l}.
\end{align}
Hence it follows from $(\ref{x-equation-1-for-l-not-square})$, $(\ref{x-equation-2-for-l-not-square})$ and the first equation of $(\ref{x-z-equation-2-for-l-not-square})$ that
\begin{align*}
\begin{cases}
BCx^{4n} &\equiv -AC \equiv -BD \pmod{l}, \\
BEx^{4n} &\equiv -AE \equiv -BF \pmod{l}, \\
CEx^{4n} &\equiv -DE \equiv -CF \pmod{l}.
\end{cases}
\end{align*}
Thus
\begin{align*}
\begin{cases}
AC - BD &\equiv 0 \pmod{l}, \\
AE - BF &\equiv 0 \pmod{l}, \\
DE - CF &\equiv 0 \pmod{l},
\end{cases}
\end{align*}
and hence $l$ divides $\gcd\left(AC - BD, DE - CF, AE - BF\right)$. Thus it follows from $(A4)$ that $p$ is a square in $\bQ_l^{\times}$, which is a contradiction. Therefore at least one of $A + Bx^{4n} + pz$ and $A + Bx^{4n} - pz$ is nonzero modulo $l$, say $U$. Since $\cA$ and $\cB$ represent the same class in $\Br(\bQ(\cC))$, we deduce that the local Hilbert symbol $(p, U)_l$ is $1$. Hence $\text{inv}_l(\cA(P_l))$ is $0$.

$\star$ \textit{Case 2. $v_l(x) = \epsilon < 0$.}

Since $\epsilon = v_l(x) \in \bZ$, we deduce that $\epsilon \le - 1$. If $v_l(E) = 0$, then we see that
\begin{align*}
v_l(E^2x^{12n}) = 12n\epsilon \le -12n < v_l(G^2).
\end{align*}
If $v_l(E) > 0$, then $l$ divides $E$. Since $p$ is not a square in $\bQ_l^{\times}$, it follows from $(A2)$ that
\begin{align*}
v_l(E) - v_l(G) < 6n.
\end{align*}
Hence
\begin{align*}
v_l(E^2x^{12n}) = 2v_l(E) + 12n\epsilon \le 2v_l(E) - 12n < 2v_l(G) = v_l(G^2).
\end{align*}
Thus, in any event, we see that $v_l(E^2x^{6n}) < v_l(G^2)$. Hence it follows from $(\ref{generalized-mordell-curve-C-equation})$ that
\begin{align*}
v_l(z) = \dfrac{v_l(pz^2)}{2} = \dfrac{\min(v_l(E^2x^{12n}), v_l(G^2))}{2} = v_l(E) + 6n\epsilon.
\end{align*}
Hence there are elements $x_0, z_0, E_0 \in \bZ_l^{\times}$ such that
\begin{align*}
x &= l^{\epsilon}x_0, \\
z &= l^{v_l(E) + 6n\epsilon}z_0, \\
E &= l^{v_l(E)}E_0.
\end{align*}
Hence it follows from $(\ref{generalized-mordell-curve-C-equation})$ that
\begin{align*}
pl^{2v_l(E) + 12n\epsilon}z_0^2 = l^{2v_l(E) + 12n\epsilon}E_0^2x_0^{12n} - G^2.
\end{align*}
Multiplying both sides of the above equation by $l^{-2v_l(E) - 12n\epsilon}$, we deduce that
\begin{align}
\label{x0-z0-equation-for-l-not-square-and-v-l(x)-less-than-0}
pz_0^2 = E_0^2x_0^{12n} - G^2l^{-2v_l(E) - 12n\epsilon}.
\end{align}

If $v_l(E) = 0$, then we see that
\begin{align*}
v_l\left(G^2l^{-2v_l(E) - 12n\epsilon}\right) = 2v_l(G) - 2v_l(E) - 12n\epsilon = 2v_l(G) - 12n\epsilon \ge 2v_l(G) + 12n \ge 12n \ge 12.
\end{align*}
If $v_l(E) > 0$, then we know that $l$ divides $E$. Since $p$ is not a square in $\bQ_l^{\times}$, we deduce from $(A2)$ that
\begin{align*}
v_l(E) - v_l(G) < 6n,
\end{align*}
and hence
\begin{align*}
v_l\left(G^2l^{-2v_l(E) - 12n\epsilon}\right) &= 2v_l(G) - 2v_l(E) - 12n\epsilon \\
&> -12n - 12n\epsilon = 12n(-\epsilon - 1) \ge 0.
\end{align*}
Thus, in any event, we see that
\begin{align*}
v_l\left(G^2l^{-2v_l(E) - 12n\epsilon}\right) > 0.
\end{align*}
Since $v_l(G^2l^{-2v_l(E) - 12n\epsilon})$ is an integer, it follows that $v_l(G^2l^{-2v_l(E) - 12n\epsilon}) \ge 1$, and hence
\begin{align*}
G^2l^{-2v_l(E) - 12n\epsilon} \in l\bZ_l.
\end{align*}
Reducing equation $(\ref{x0-z0-equation-for-l-not-square-and-v-l(x)-less-than-0})$ modulo $l$, one obtains that
\begin{align*}
p \equiv \left(\dfrac{E_0x_0^{6n}}{z_0}\right)^2 \pmod{l},
\end{align*}
which is a contradiction since $p$ is not a square in $\bQ_l^{\times}$.

Thus, in any event, if $l$ is an odd prime such that $\gcd(l, 3) = \gcd(l, p) = 1$ and $p$ is not a square in $\bQ_l^{\times}$, then $\text{inv}_l(\cA(P_l))$ is $0$.

Suppose that $l = p$. If $v_p(x) = \epsilon < 0$, then we know from $(A3)$ that $E \not\equiv 0 \pmod{p}$. Hence we see that
\begin{align*}
v_p\left(E^2x^{12n}\right) = 12n\epsilon < 0 \le v_p(G^2).
\end{align*}
It then follows from the above inequality and $(\ref{generalized-mordell-curve-C-equation})$ that
\begin{align*}
1 + 2v_p(z) = v_p(pz^2) = v_p(E^2x^{12n}) = 12n\epsilon,
\end{align*}
which is a contradiction since the left-hand side is an odd integer whereas the right-hand side is an even integer.

If $v_p(x) \ge 0$, then it follows from $(\ref{generalized-mordell-curve-C-equation})$ that
\begin{align*}
1 + 2v_p(z) = v_p(pz^2) = v_p\left(E^2x^{12n} - G^2\right) \ge 0.
\end{align*}
Hence it follows that $v_p(z) \ge -\dfrac{1}{2}$. Since $v_p(z) \in \bZ$, we deduce that $v_p(z) \ge 0$, and hence $z \in \bZ_p$. We contend that $x \in \bZ_p^{\times}$. Assume the contrary, that is, $x \equiv 0 \pmod{p}$. By $(\ref{generalized-mordell-curve-C-equation})$, we deduce that
\begin{align*}
G^2 = E^2x^{12n} - pz^2 \equiv 0 \pmod{p},
\end{align*}
which is a contradiction to $(A3)$. Hence we deduce that $x \in \bZ_p^{\times}$. Reducing equation $(\ref{generalized-mordell-curve-C-equation})$ modulo $p$, we deduce that
\begin{align}
\label{equation-of-x-in-terms-of-E-and-G-in-the-case-of-l-equal-p-equation}
E^2x^{12n} - G^2 \equiv 0 \pmod{p}.
\end{align}
Let $H$ be an integer satisfying $(A5)$ in Definition \ref{septuple-satisfies-hypothesis-FM-definition}. By $(\ref{equation-of-x-in-terms-of-E-and-G-in-the-case-of-l-equal-p-equation})$, $(A3)$ and $(A5)$, we see that
\begin{align*}
x^{12n} &\equiv \left(\dfrac{G}{E}\right)^2 \equiv H^{12} \pmod{p},
\end{align*}
and hence $x^{4n} \equiv \zeta H^4 \pmod{p}$ for some cube root of unity $\zeta$ in $\bF_p^{\times}$. Thus it follows from $(A5)$ that
\begin{align*}
A + Bx^{4n} + pz \equiv A + \zeta BH^4 \not\equiv 0 \pmod{p}.
\end{align*}
Using Theorem $5.2.7$ in \cite{cohen}, we deduce from $(A5)$ that the local Hilbert symbol $(p, A + Bx^{4n} + pz)_p$ satisfies
\begin{align*}
(p, A + Bx^{4n} + pz)_p = \left(\dfrac{A + \zeta BH^4}{p} \right) = -1,
\end{align*}
which proves that $\text{inv}_p(\cA(P_p)) = 1/2$.

Therefore, in any event, we have
\begin{align*}
\sum_{l}\text{inv}_l\cA(P_l) = 1/2
\end{align*}
for any $(P_l)_{l} \in \cC(\bA_{\bQ})$, and hence $\cC(\bA_{\bQ})^{\text{Br}} = \emptyset$. Hence our contention follows.

\end{proof}

\section{A subset of the set of all rational points on $\cX_p$}
\label{a-subset-of-the-set-of-rational-points-on-X-p-section}

Let $p$ be a prime such that $p \equiv 1 \pmod{8}$ and $p \equiv 2 \pmod{3}$. Let $n$ be an integer such that $n \ge 2$. In this section, we will construct a subset of the set of all rational points on $\cX_p$ that produces infinitely many septuples $(A, B, C, D, E, F, G) \in \bZ^7$ satisfying Hypothesis FM with respect to $(p, n)$ in the sense of Definition \ref{septuple-satisfies-hypothesis-FM-definition}. This subset of $\cX_p(\bQ)$ plays a key role in constructing certain families of generalized Mordell curves and certain families of generalized Fermat curves that are counterexamples to the Hasse principle explained by the Brauer-Manin obstruction.

Let $(\alpha, \beta, \kappa) \in \bZ^3$ be a triple of nonzero integers satisfying the following conditions:
\begin{itemize}

\item [(B1)] $\alpha$ and $\beta$ are odd and $\gcd(\alpha, 3) = \gcd(\alpha, p) = \gcd(\alpha, \beta) = \gcd(\beta, p) = 1$.

\item [(B2)] let $l$ be any odd prime such that $\gcd(l, 3) = 1$ and $l$ divides $\alpha\beta$. Then $p$ is a square in $\bQ_l^{\times}$.

\item [(B3)] $v_3(\kappa) < 2n - 1$ and $v_3(\kappa)$ is an odd positive integer.

\item [(B4)] $\kappa \not\equiv 0 \pmod{p}$.

\item [(B5)] let $l$ be any odd prime such that $\gcd(l, 3) = 1$ and $l$ divides $\kappa$. Then $p$ is a square in $\bQ_l^{\times}$.

\end{itemize}
Define
\begin{align}
\label{the-first-parameterization-of-septuples-(A-B-C-D-E-F-G)-satisfying-Hypothesis-FM-equations}
\begin{cases}
A &:= \dfrac{p\alpha^2 - 9\beta^2}{2}  \\
B &:= 9(p\alpha^2 - 9\beta^2)\kappa^2    \\
C &:= 9(p\alpha^2 + 9\beta^2)\kappa^2   \\
D &:= \dfrac{p\alpha^2 + 9\beta^2}{2}  \\
E &:= 81\alpha\beta\kappa^3 \\
F &:= 18\alpha\beta\kappa \\
G &:= 3\alpha\beta.
\end{cases}
\end{align}
By $(\ref{the-first-parameterization-of-septuples-(A-B-C-D-E-F-G)-satisfying-Hypothesis-FM-equations})$, we deduce that
\begin{align}
\label{the-expressions-of-B-C-E-F-in-terms-of-A-D-G-equations}
\begin{cases}
B &=  18\kappa^2A \\
C &=  18\kappa^2D \\
E &= 27\kappa^3G \\
F &= 6\kappa G.
\end{cases}
\end{align}
We now prove that $(A, B, C, D, E, F, G)$ satisfies Hypothesis FM with respect to $(p, n)$.

\begin{lemma}
\label{the-first-parameterization-of-septuples-(A-B-C-D-E-F-G)-satisfying-Hypothesis-FM-lemma}

Let $p$ be a prime such that $p \equiv 1 \pmod{8}$ and $p \equiv 2 \pmod{3}$. Let $n$ be an integer such that $n \ge 2$. Let $(\alpha, \beta, \kappa) \in \bZ^3$ be a triple of nonzero integers satisfying $(B1)-(B5)$. Let $\cT := (A, B, C, D, E, F, G) \in \bZ^7$ be a septuple of integers defined by $(\ref{the-first-parameterization-of-septuples-(A-B-C-D-E-F-G)-satisfying-Hypothesis-FM-equations})$. Then $\cT$ satisfies Hypothesis FM with respect to $(p, n)$.

\end{lemma}

\begin{proof}

By $(B1)$, $(B4)$ and $(\ref{the-first-parameterization-of-septuples-(A-B-C-D-E-F-G)-satisfying-Hypothesis-FM-equations})$, we see that $\cT$ satisfies $(A3)$ and $(A7)$. Furthermore, it is not difficult to verify that the point $\cP := (a : b : c : d : e : f : g) = (A : B : C : D : E : F : G)$ belongs to $\cX_p(\bQ)$, and hence $\cT$ satisfies $(A1)$.

We now prove that $\cT$ satisfies $(A2)$. Indeed, let $l$ be an odd prime such that $\gcd(l, 3) = \gcd(l, p) = 1$ and $l$ divides $E$. We see that either $l$ divides $\kappa$ or $\gcd(l, \kappa) = 1$. If $l$ divides $\kappa$, then it follows from $(B5)$ that $p$ is a square in $\bQ_l^{\times}$. Assume that $\gcd(l, \kappa) = 1$. By $(\ref{the-expressions-of-B-C-E-F-in-terms-of-A-D-G-equations})$, we see that
\begin{align*}
v_l(E) - v_l(G) = v_l(27\kappa^3G) - v_l(G) = v_l(27\kappa^3) = 0 < 6n,
\end{align*}
and hence $\cT$ satisfies $(A2)$.

We now prove that $\cT$ satisfies $(A4)$. Suppose that $l$ is an odd prime such that $\gcd(l, 3) = \gcd(l, p) = 1$ and
\begin{align*}
\gcd(AC - BD, DE - CF, AE - BF) \equiv 0 \pmod{l}.
\end{align*}
By $(B1)$ and $(\ref{the-first-parameterization-of-septuples-(A-B-C-D-E-F-G)-satisfying-Hypothesis-FM-equations})$, one can show that $\gcd(A, D) = 1$. By $(\ref{the-expressions-of-B-C-E-F-in-terms-of-A-D-G-equations})$, we see that
\begin{align*}
AC - BD &= 0, \\
DE - CF &= -3^4\kappa^3DG,  \\
AE - BF &= -3^4\kappa^3AG.
\end{align*}
Hence it follows that
\begin{align*}
\gcd(AC - BD, DE - CF, AE - BF) = \pm 3^4\kappa^3G,
\end{align*}
and hence we deduce that $l$ divides $3^4\kappa^3G$. Since $\gcd(l, 3) = 1$, it follows from $(\ref{the-first-parameterization-of-septuples-(A-B-C-D-E-F-G)-satisfying-Hypothesis-FM-equations})$ that $l$ divides $\kappa$ or $\alpha\beta$. By $(B2)$ and $(B5)$, we deduce that $p$ is a square in $\bQ_l^{\times}$, which proves that $(A4)$ is true.

We now prove that $(A5)$ holds. Let $\kappa^{\ast}$ be the integer such that $\kappa = 3^{v_3(\kappa)}\kappa^{\ast}$ and $\gcd(\kappa^{\ast}, 3) = 1$, where $v_3$ denotes the $3$-adic valuation. We contend that $3\kappa$ is a square in $\bF_p^{\times}$. Indeed, we see that $3\kappa = 3^{v_3(\kappa) + 1}\kappa^{\ast}$. By $(B3)$, we deduce that $v_3(\kappa) + 1$ is an even positive integer, and hence it follows that $3^{v_3(\kappa) + 1}$ is a square in $\bF_p^{\times}$. Thus it suffices to show that $\kappa^{\ast}$ is a square in $\bF_p^{\times}$. Write $\kappa^{\ast} = \pm 2^{v_2(\kappa^{\ast})}\prod_{l \mid \kappa^{\ast}}l^{v_l(\kappa^{\ast})}$, where $l$ ranges over the set of all odd primes dividing $\kappa^{\ast}$. By $(B5)$ and the quadratic reciprocity law, we know that $l$ is a square in $\bF_p^{\times}$ for any odd prime $l$ dividing $\kappa^{\ast}$. Since $p \equiv 1 \pmod{8}$, we deduce that the Jacobi symbol $\left(\dfrac{\kappa^{\ast}}{p}\right)$ satisfies
\begin{align*}
\left(\dfrac{\kappa^{\ast}}{p}\right) = \left(\dfrac{\pm 1}{p}\right)\left(\dfrac{2^{v_2(\kappa^{\ast})}}{p}\right)\left(\prod_{l \mid \kappa^{\ast}}\left(\dfrac{l}{p}\right)\right) = 1.
\end{align*}
Thus $\kappa^{\ast}$ is a square in $\bF_p^{\times}$, and hence $3\kappa$ is a square in $\bF_p^{\times}$. Thus there exists an integer $H$ such that
\begin{align}
\label{the-first-paramterization-of-H-satisfies-Hypothesis-FM-equation}
3\kappa H^2 \equiv 1 \pmod{p}.
\end{align}
We contend that $H$ satisfies $(A5)$. Indeed, by $(\ref{the-expressions-of-B-C-E-F-in-terms-of-A-D-G-equations})$ and $(\ref{the-first-paramterization-of-H-satisfies-Hypothesis-FM-equation})$, we see that
\begin{align*}
G - EH^6 = G - 27\kappa^3GH^6 = G\left(1 - 27\kappa^3H^6\right) = G\left(1 - (3\kappa H^2)^3\right) \equiv 0 \pmod{p}.
\end{align*}

By Remark \ref{3-is-not-a-square-in-Fp-remark} and since $p \equiv 2 \pmod{3}$, it remains to show that $A + BH^4$ is a quadratic non-residue in $\bF_p^{\times}$. By $(\ref{the-expressions-of-B-C-E-F-in-terms-of-A-D-G-equations})$ and $(\ref{the-first-paramterization-of-H-satisfies-Hypothesis-FM-equation})$, we know that
\begin{align*}
A + BH^4 = A + 18\kappa^2AH^4 = A\left(1 + 18\kappa^2H^4\right) = A\left(1 + 2(3\kappa H^2)^2\right) \equiv 3A \pmod{p}.
\end{align*}
By $(\ref{the-first-parameterization-of-septuples-(A-B-C-D-E-F-G)-satisfying-Hypothesis-FM-equations})$, we see that
\begin{align*}
3A = \dfrac{3(p\alpha^2 - 9\beta^2)}{2} \equiv -\dfrac{3(3\beta)^2}{2} \pmod{p}.
\end{align*}
Since $p \equiv 1 \pmod{8}$, we know that $-1$ and $1/2$ are squares in $\bF_p^{\times}$. Since $p \equiv 2 \pmod{3}$, we see that $p$ is a quadratic non-residue modulo $3$, and it follows from the quadratic reciprocity law that $3$ is not a square in $\bF_p^{\times}$. Thus we deduce that the Jacobi symbol $\left(\dfrac{A + BH^4}{p}\right)$ satisfies
\begin{align*}
\left(\dfrac{A + BH^4}{p}\right) = \left(\dfrac{3A}{p}\right) = \left(\dfrac{-1}{p}\right)\left(\dfrac{1/2}{p}\right)\left(\dfrac{3}{p}\right)\left(\dfrac{(3\beta)^2}{p}\right) = -1,
\end{align*}
and thus $A + BH^4$ is a quadratic non-residue in $\bF_p^{\times}$. Therefore $\cT$ satisfies $(A5)$.

By $(B3)$ and $(\ref{the-expressions-of-B-C-E-F-in-terms-of-A-D-G-equations})$, we know that
\begin{align*}
v_3(E) - v_3(G) = v_3\left(27\kappa^3G\right) - v_3(G) = v_3(27) + v_3(\kappa^3) = 3 + 3v_3(\kappa) < 6n,
\end{align*}
and hence $\cT$ satisfies $(A6)$. Thus our contention follows.

\end{proof}

\begin{corollary}
\label{the-generalized-Mordell-curve-D-have-no-rational-points-corollary}

Let $p$ be a prime such that $p \equiv 1 \pmod{8}$ and $p \equiv 2 \pmod{3}$. Let $n$ be an integer such that $n \ge 2$. Let $\kappa$ be a nonzero integer satisfying $(B3)$, $(B4)$ and $(B5)$. Let $\cD$ be the smooth projective model of the affine curve defined by
\begin{align}
\label{the-generalized-Mordell-curve-D-equation}
\cD : pz^2 = 3^6\kappa^6x^{12n} - 1.
\end{align}
Then $\cD(\bA_{\bQ})^{\Br} = \emptyset$.

\end{corollary}

\begin{proof}

It is easy to show that there are infinitely many couples $(\alpha, \beta) \in \bZ^2$ of nonzero integers that satisfy $(B1)$ and $(B2)$. Take such a couple $(\alpha, \beta) \in \bZ^2$, and let $\cT := (A, B, C, D, E, F, G) \in \bZ^7$ be the septuple of integers defined by $(\ref{the-first-parameterization-of-septuples-(A-B-C-D-E-F-G)-satisfying-Hypothesis-FM-equations})$. It follows from Lemma \ref{the-first-parameterization-of-septuples-(A-B-C-D-E-F-G)-satisfying-Hypothesis-FM-lemma} that $\cT$ satisfies Hypothesis FM with respect to $(p, n)$. Let $\cC$ be the smooth projective model of the affine curve defined by $(\ref{generalized-mordell-curve-C-equation})$ in Theorem \ref{non-existence-of-rational-points-for-generalized-mordell-curves-using-rational-points-on-threefold-Xp-theorem}, that is, $\cC$ is of the form
\begin{align*}
\cC : pz^2 = E^2x^{12n} - G^2.
\end{align*}
By Theorem \ref{non-existence-of-rational-points-for-generalized-mordell-curves-using-rational-points-on-threefold-Xp-theorem}, we know that $\cC(\bA_{\bQ})^{\mathrm{Br}} = \emptyset$. By $(\ref{the-expressions-of-B-C-E-F-in-terms-of-A-D-G-equations})$, we know that $E = 3^3\kappa^3G$. Hence, under the morphism
\begin{align*}
(x, z) &\mapsto (x, Gz),
\end{align*}
we see that the curve $\cC$ is isomorphic to $\cD$ over $\bQ$. Thus we deduce that $\cD(\bA_{\bQ})^{\mathrm{Br}} = \emptyset$.

\end{proof}

\section{Certain generalized Mordell curves violating the Hasse principle}
\label{certain-generalized-Mordell-curves-violate-the-Hasse-principle-section}

Let $p$ be an odd prime such that $p \equiv 1 \pmod{8}$ and $p \equiv 2 \pmod{3}$. Let $n$ be an integer such that $n \ge 2$. Let $\kappa$ be a nonzero integer satisfying $(B3)$, $(B4)$ and $(B5)$. In this section, we will prove a sufficient condition depending on certain congruences of $\kappa$ modulo finitely many primes under which the curve $\cD$ defined by $(\ref{the-generalized-Mordell-curve-D-equation})$ is a counterexample to the Hasse principle explained by the Brauer-Manin obstruction. This result will be proved in Subsection \ref{a-sufficient-condition-for-which-certain-generalized-Mordell-curves-violate-the-Hasse-principle-subsection}. As a consequence, in Subsection \ref{infinitude-of-the-triples-p-n-kappa-subsection}, we will show that there are infinitely many integers $\kappa$ satisfying the sufficient condition, and hence it follows that there are infinitely many non-isomorphic generalized Mordell curves violating the Hasse principle explained by the Brauer-Manin obstruction.

\subsection{A sufficient condition}
\label{a-sufficient-condition-for-which-certain-generalized-Mordell-curves-violate-the-Hasse-principle-subsection}

We recall the celebrated \textit{Hasse-Weil bound}.

\begin{lemma}
\label{Hasse-Weil-bound-lemma}
$(\text{see \cite[Corollary 7.2.1, p.130]{poonen3}})$

Let $\cX$ be a smooth geometrically irreducible projective curve of genus $n$ over the finite field $\bF_q$. Then
\begin{align*}
\left|\#\cX(\bF_q) - (q + 1)\right| \le 2n\sqrt{q}.
\end{align*}

\end{lemma}

We now prove the main theorem in this section.

\begin{theorem}
\label{the-first-infinite-family-of-generalized-Mordell-curves-violates-the-Hasse-principle-theorem}

Let $p$ be a prime such that $p \equiv 1 \pmod{8}$ and $p \equiv 2 \pmod{3}$. Let $n$ be an integer such that $n \ge 2$. Let $\kappa$ be a nonzero integer satisfying $(B3)$, $(B4)$ and $(B5)$. Assume further that the following are true:
\begin{itemize}

\item [(C1)] $\kappa \equiv \dfrac{1}{3} \pmod{p^{2v_p(n) + 1}}$.

\item [(C2)] let $\A$ be the set of odd primes $l$ satisfying the following three conditions:

        \begin{itemize}

        \item [(i)] $\gcd(l, 3) = \gcd(l, p) = \gcd(l, \kappa) = 1$.

        \item [(ii)] $\left(\dfrac{p}{l}\right) = \left(\dfrac{-p}{l}\right) = -1$.

        \item [(iii)] $l$ divides $n$.

        \end{itemize}
We assume that $3^6\kappa^6 - 1$ is a quadratic non-residue in $\bF_l^{\times}$ or $\kappa \equiv \dfrac{1}{3} \pmod{l^{2v_l(n) + 1}}$ for each prime $l \in \A$.

\item [(C3)] let $\B$ be the set of odd primes $l$ satisfying the following three conditions:

        \begin{itemize}

        \item [(i)] $\gcd(l, 3) = \gcd(l, p) = \gcd(l, \kappa) = \gcd(l, n) = 1$.

        \item [(ii)] $\left(\dfrac{p}{l}\right) = \left(\dfrac{-p}{l}\right) = -1$.

        \item [(iii)] $l \le 4(6n - 1)^2$.

        \end{itemize}
We assume that $3^6\kappa^6 - 1$ is a quadratic non-residue in $\bF_l^{\times}$ or $\kappa \equiv \dfrac{1}{3} \pmod{l}$ for each prime $l \in \B$.

\end{itemize}
Let $\cD$ be the smooth projective model of the affine curve defined by $(\ref{the-generalized-Mordell-curve-D-equation})$ in Corollary \ref{the-generalized-Mordell-curve-D-have-no-rational-points-corollary}. Then $\cD$ is a counterexample to the Hasse principle explained by the Brauer-Manin obstruction.

\end{theorem}

\begin{proof}

By Corollary \ref{the-generalized-Mordell-curve-D-have-no-rational-points-corollary}, we know that $\cD(\bA_{\bQ})^{\mathrm{Br}} = \emptyset$. Hence it remains to prove that $\cD$ is everywhere locally solvable.

Let $\S_1$ be the set of odd primes $l$ with $\gcd(l, p) = 1$ such that $\left(\dfrac{p}{l}\right) = 1$ or $\left(\dfrac{-p}{l}\right) = 1$. Let $\S_2$ be the set of odd primes $l$ with $\gcd(l, p) = 1$ such that $\left(\dfrac{p}{l}\right) = \left(\dfrac{-p}{l}\right) = -1$. Since $-p \equiv -2 \equiv 1 \pmod{3}$, we see that $3$ belongs to $\S_1$. By $(B5)$, we also know that if $l$ is any odd prime such that $\gcd(l, 3) = 1$ and $l$ divides $\kappa$, then $l$ belongs to $\S_1$. We know that
\begin{align*}
\left\{\text{the set of all primes}\right\} = \{2\} \cup \{p\} \cup \S_1 \cup \S_2.
\end{align*}
It suffices to consider the following cases.

$\star$ \textit{Case 1. $l = p$.}

We consider the system of equations
\begin{align}
\label{Hensel-system-of-D-at-p-equations}
\begin{cases}
F(x, z) := 3^6\kappa^6x^{12n} - 1 - pz^2 &\equiv 0 \pmod{p^{2v_p(n) + 1}} \\
\tfrac{\partial F}{\partial x}(x, z) = 2^23^7\kappa^6nx^{12n - 1} &\equiv 0 \pmod{p^{v_p(n)}} \\
\tfrac{\partial F}{\partial x}(x, z) &\not\equiv 0 \pmod{p^{v_p(n) + 1}}.
\end{cases}
\end{align}
By $(C1)$, we see that
\begin{align*}
F(1, 0) = (3\kappa)^6 - 1 \equiv 0 \pmod{p^{2v_p(n) + 1}}.
\end{align*}
Since $p \ne 2, 3$ and $\gcd(\kappa, p) = 1$, we deduce that
\begin{align*}
\dfrac{\partial F}{\partial x}(1, 0) &= 2^23^7\kappa^6n \equiv 0 \pmod{p^{v_p(n)}}, \\
\dfrac{\partial F}{\partial x}(1, 0) &= 2^23^7\kappa^6n \not\equiv 0 \pmod{p^{v_p(n) + 1}}. \\
\end{align*}
Hence $(1, 0)$ is a solution to the system $(\ref{Hensel-system-of-D-at-p-equations})$, and it thus follows from Hensel's lemma that $\cD$ is locally solvable at $p$.

$\star$ \textit{Case 2. $l = \infty$, $l = 2$ or $l \in \S_1$.}

We see that the curve $\cD^{\ast}$ defined by
\begin{align*}
\cD^{\ast} : pz^2 = 3^6\kappa^6x^{12n} - y^{12n}
\end{align*}
is an open subscheme of $\cD$. Assume first that $l = \infty$, $l = 2$ or $l$ is an odd prime such that $\left(\dfrac{p}{l}\right) = 1$. Define
\begin{align*}
P_1 := (x : y : z) = \left(p : 0 : 3^3\kappa^3p^{6n - 1}\sqrt{p}\right).
\end{align*}
Since $\sqrt{p} \in \bQ_l^{\times}$, we deduce that $P_1 \in \cD^{\ast}(\bQ_l) \subseteq \cD(\bQ_l)$.

Suppose now that $l$ is an odd prime such that $\left(\dfrac{-p}{l}\right) = 1$. We define
\begin{align*}
P_2 := (x, z) = \left(0, \dfrac{\sqrt{-p}}{p}\right).
\end{align*}
Since $\sqrt{-p} \in \bQ_l^{\times}$, we deduce that $P_2$ belongs to $\cD(\bQ_l)$. Hence, in any event, $\cD$ is locally solvable at $l$.

$\star$ \textit{Case 3. $l \in \S_2$.}

By the discussion preceding \textit{Case 1}, one can assume that $\gcd(l, 3) = \gcd(l, \kappa) = 1$. We know that the discriminant of the curve $\cD$ is
\begin{align*}
\text{Discriminant}(\cD) = -3^{6(12n - 1)}\kappa^{6(12n - 1)}p^{2(12n - 1)}(12n)^{12n}.
\end{align*}
Hence $\cD$ is only singular at the primes $q = 2, 3, p$, and the primes $q$ dividing $\kappa n$; so the Hasse-Weil bound (see Lemma \ref{Hasse-Weil-bound-lemma}) assures that $\cD$ is locally solvable at primes $q > 4(6n - 1)^2$ such that $q \ne 2, 3, p$ and $\gcd(q, \kappa) = \gcd(q, n) = 1$. Hence, by \textit{Cases 1 and 2}, we only need to prove that $\cD$ is locally solvable at the primes $l$, where $l \in \A$ or $l \in \B$.

$\bullet$ \textit{Subcase 1. $l \in \A$.}

Since $l \in \A$, it follows from $(C2)$ that $3^6\kappa^6 - 1$ is a quadratic non-residue in $\bF_l^{\times}$ or $\kappa \equiv \dfrac{1}{3} \pmod{l^{2v_l(n) +1}}$. Assume first that $3^6\kappa^6 - 1$ is a quadratic non-residue in $\bF_l^{\times}$. We consider the system of equations
\begin{align}
\label{Hensel-system-of-D-at-l-dividing-n-equations}
\begin{cases}
F(x, z) = 3^6\kappa^6x^{12n} - 1 - pz^2 &\equiv 0 \pmod{l} \\
\tfrac{\partial F}{\partial z}(x, z) = -2pz &\not\equiv 0 \pmod{l}. \\
\end{cases}
\end{align}
Since $l$ belongs to $\A$, we know that $p$ is a quadratic non-residue in $\bF_l^{\times}$. It then follows that $\tfrac{3^6\kappa^6 - 1}{p}$ is a square in $\bF_l^{\times}$, and thus there exists an integer $z_0$ such that $z_0 \not\equiv 0 \pmod{l}$ and $\tfrac{3^6\kappa^6 - 1}{p} \equiv z_0^2 \pmod{l}$. Hence we deduce that
$F(1, z_0)$ is zero modulo $l$. Since $l$ does not divide $z_0$, it follows that $\tfrac{\partial F}{\partial z}(1, z_0) = -2pz_0 \not\equiv 0 \pmod{l}$. Thus $(1, z_0)$ is a solution to the system $(\ref{Hensel-system-of-D-at-l-dividing-n-equations})$, and therefore it follows from Hensel's lemma that $\cD$ is locally solvable at $l$.

Suppose now that $\kappa \equiv \dfrac{1}{3} \pmod{l^{2v_l(n) + 1}}$. Since $l \ne 2, 3$ and $\gcd(l, p) = \gcd(l, \kappa) = 1$, one can show that the point $(1, 0)$ is a solution to the system of equations
\begin{align*}
\begin{cases}
F(x, z) = 3^6\kappa^6x^{12n} - 1 - pz^2 &\equiv 0 \pmod{l^{2v_l(n) + 1}} \\
\dfrac{\partial F}{\partial x}(x, z) = 2^23^7\kappa^6nx^{12n - 1} &\equiv 0 \pmod{l^{v_l(n)}} \\
\dfrac{\partial F}{\partial x}(x, z) &\not\equiv 0 \pmod{l^{v_l(n) + 1}}.
\end{cases}
\end{align*}
By Hensel's lemma, we deduce that $\cD$ is locally solvable at $l$.

$\bullet$ \textit{Subcase 2. $l \in \B$.}

Since $l \in \B$, it follows from $(C3)$ that $3^6\kappa^6 - 1$ is a quadratic non-residue in $\bF_l^{\times}$ or $\kappa \equiv \dfrac{1}{3} \pmod{l}$. Assume first that $3^6\kappa^6 - 1$ is a quadratic non-residue in $\bF_l^{\times}$. Then, using the same arguments as in \textit{Subcase 1}, we deduce that $\dfrac{3^6\kappa^6 - 1}{p}$ is a square in $\bF_l^{\times}$, and thus there exists an integer $z_0$ such that $z_0 \not\equiv 0 \pmod{l}$ and $\dfrac{3^6\kappa^6 - 1}{p} \equiv z_0^2 \pmod{l}$. Repeating in the same manner as in \textit{Subcase 1}, we see that $(1, z_0)$ is a solution to the system of equations
\begin{align*}
\begin{cases}
F(x, z) = 3^6\kappa^6x^{12n} - 1 - pz^2 &\equiv 0 \pmod{l} \\
\dfrac{\partial F}{\partial z}(x, z) = -2pz &\not\equiv 0 \pmod{l}. \\
\end{cases}
\end{align*}
By Hensel's lemma, we deduce that $\cD$ is locally solvable at $l$.

Suppose now that $\kappa \equiv \dfrac{1}{3} \pmod{l}$. Since $l \ne 2, 3$ and $\gcd(l, p) = \gcd(l, \kappa) = \gcd(l, n) = 1$, one can show that the point $(1, 0)$ is a solution to the system of equations
\begin{align*}
\begin{cases}
F(x, z) = 3^6\kappa^6x^{12n} - 1 - pz^2 &\equiv 0 \pmod{l} \\
\dfrac{\partial F}{\partial x}(x, z) = 2^23^7\kappa^6nx^{12n - 1} &\not\equiv 0 \pmod{l}. \\
\end{cases}
\end{align*}
By Hensel's lemma, we deduce that $\cD$ is locally solvable at $l$.

Therefore, in any event, $\cD$ is everywhere locally solvable, and hence our contention follows.

\end{proof}

\subsection{Infinitude of the triples $(p, n, \kappa)$}
\label{infinitude-of-the-triples-p-n-kappa-subsection}

In this subsection, we will prove that for a prime $p \equiv 1 \pmod{8}$ and a positive integer $n \ge 2$, there are infinitely many nonzero integers $\kappa$ satisfying the conditions in Theorem \ref{the-first-infinite-family-of-generalized-Mordell-curves-violates-the-Hasse-principle-theorem}. The main result in this subsection is the following lemma.

\begin{lemma}
\label{infinitude-of-the-triples-p-n-kappa-lemma}

Let $p$ be a prime such that $p \equiv 1 \pmod{8}$ and $p \equiv 2 \pmod{3}$. Let $n$ and $m$ be two integers such that $n \ge 2$ and $ 1 \le m < n$. Then, for any positive integer $r$ dividing $m$, there are infinitely many nonzero integers $\kappa$ of the form
\begin{align*}
\kappa = 3^{2m - 1}\kappa_{\ast}^r,
\end{align*}
where $\kappa_{\ast}$ is an odd prime with $\gcd(\kappa_{\ast}, 3) = 1$ such that $\kappa$ satisfies $(B3)-(B5)$ and $(C1)-(C3)$.

\end{lemma}

\begin{proof}

Let $r$ be any positive integer such that $r$ divides $m$, and define
\begin{align*}
s := \dfrac{m}{r}.
\end{align*}
Let $\A^{\ast}$ be the set of odd primes $l$ satisfying the following three conditions:

\begin{itemize}

\item [(1)] $\gcd(l, 3) = \gcd(l, p) = 1$,

\item [(2)] $\left(\dfrac{p}{l}\right) = \left(\dfrac{-p}{l}\right) = -1$, and

\item [(3)] $l$ divides $n$.

\end{itemize}
Let $\B^{\ast}$ be the set of odd primes $l$ satisfying the following three conditions:

\begin{itemize}

\item [(1)] $\gcd(l, 3) = \gcd(l, p) = \gcd(l, n) = 1$,

\item [(2)] $\left(\dfrac{p}{l}\right) = \left(\dfrac{-p}{l}\right) = -1$, and

\item [(3)] $l \le 4(6n - 1)^2$.

\end{itemize}
It is easy to see that $p$ does not belong to $\A^{\ast} \cup \B^{\ast}$ and $\A^{\ast} \cap \B^{\ast} = \emptyset$. We also note that $\A^{\ast}$ and $\B^{\ast}$ are of finite cardinality. Hence, by the Chinese Remainder Theorem, there exists an integer $\kappa_{{\ast}, 0}$ satisfying the following conditions:

\begin{itemize}

\item [(i)] $\kappa_{{\ast}, 0} \equiv \dfrac{1}{3^{2s}} \pmod{p^{2v_p(n) + 1}}$,

\item [(ii)] $\kappa_{{\ast}, 0} \equiv \dfrac{1}{3^{2s}} \pmod{l^{2v_l(n) + 1}}$ for each prime $l \in \A^{\ast}$, and

\item [(iii)] $\kappa_{{\ast}, 0} \equiv \dfrac{1}{3^{2s}} \pmod{l}$ for each prime $l \in \B^{\ast}$.

\end{itemize}
Let $P(X) \in \bZ[X]$ be the linear polynomial defined by
\begin{align}
\label{the-equation-of-the-generating-prime-polynomial-P(x)-in-the-lemma-about-infinitude-of-the-triples-p-n-kappa-equation}
P(X) := p^{2v_p(n) + 1}\left(\prod_{l \in \A^{\ast}}l^{2v_l(n) + 1}\right)\left(\prod_{l \in \B^{\ast}}l\right)X + \kappa_{{\ast}, 0}.
\end{align}
By $(i), (ii)$ and $(iii)$ above, we know that $\gcd(\kappa_{{\ast}, 0}, p) = \gcd(\kappa_{{\ast}, 0}, l) = 1$ for any prime $l \in \A^{\ast} \cup \B^{\ast}$. Using Dirichlet's theorem on arithmetic progressions, we deduce that there are infinitely many integers $X$ such that $P(X) \ne 3$ and $P(X)$ is an odd prime. Take such an integer $X$, and let $\kappa_{\ast} := P(X)$. We define
\begin{align*}
\kappa := 3^{2m - 1}\kappa_{\ast}^r.
\end{align*}
We contend that $\kappa$ satisfies our requirements. Indeed, since $\gcd(\kappa_{\ast}, 3) = 1$, we see that
\begin{align*}
v_3(\kappa) = v_3\left(3^{2m - 1}\kappa_{\ast}^r\right) = 2m - 1 < 2n - 1.
\end{align*}
Hence $\kappa$ satisfies $(B3)$ in Section \ref{a-subset-of-the-set-of-rational-points-on-X-p-section}.

By $(i)$ and $(\ref{the-equation-of-the-generating-prime-polynomial-P(x)-in-the-lemma-about-infinitude-of-the-triples-p-n-kappa-equation})$ above, one knows that
\begin{align}
\label{the-congruence-mod-p-of-kappa-ast-equation}
\kappa_{\ast} = P(X) \equiv \kappa_{{\ast}, 0} \equiv \dfrac{1}{3^{2s}} \pmod{p^{2v_p(n) + 1}},
\end{align}
and hence
\begin{align*}
\kappa_{\ast} \equiv \left(\dfrac{1}{3^{s}}\right)^2 \not\equiv 0 \pmod{p}.
\end{align*}
Thus $\kappa_{\ast}$ is a quadratic residue in $\bF_p^{\times}$, and therefore it follows from the quadratic reciprocity law that $p$ is a square in $\bF_{\kappa_{\ast}}^{\times}$. Therefore $\kappa$ satisfies $(B4)$ and $(B5)$.

By $(\ref{the-congruence-mod-p-of-kappa-ast-equation})$ and the definition of $\kappa$, we see that
\begin{align*}
\kappa = 3^{2m - 1}\kappa_{\ast}^r \equiv \dfrac{3^{2m - 1}}{3^{2rs}} \equiv \dfrac{3^{2m - 1}}{3^{2m}} \equiv \dfrac{1}{3} \pmod{p^{2v_p(n) + 1}},
\end{align*}
which proves that $(C1)$ in Theorem \ref{the-first-infinite-family-of-generalized-Mordell-curves-violates-the-Hasse-principle-theorem} is true.

Now let $\A$ and $\B$ be the sets of odd primes defined as in $(C2)$ and $(C3)$ in Theorem \ref{the-first-infinite-family-of-generalized-Mordell-curves-violates-the-Hasse-principle-theorem}. One sees that
\begin{align*}
\A &= \left\{l \in \A^{\ast} | \gcd(l, \kappa) = 1 \right\} \subseteq \A^{\ast}, \\
\B &= \left\{l \in \B^{\ast} | \gcd(l, \kappa) = 1 \right\} \subseteq \B^{\ast}.
\end{align*}
By $(ii)$ and $(\ref{the-equation-of-the-generating-prime-polynomial-P(x)-in-the-lemma-about-infinitude-of-the-triples-p-n-kappa-equation})$, we know that
\begin{align*}
\kappa = 3^{2m - 1}\kappa_{\ast}^r \equiv 3^{2m - 1}\kappa_{{\ast}, 0}^r \equiv \dfrac{3^{2m - 1}}{3^{2rs}} \equiv \dfrac{3^{2m - 1}}{3^{2m}} \equiv \dfrac{1}{3} \pmod{l^{2v_l(n) + 1}}
\end{align*}
for any prime $l \in \A^{\ast}$. Since $\A \subseteq \A^{\ast}$, we deduce that $\kappa \equiv \dfrac{1}{3} \pmod{l^{2v_l(n) + 1}}$ for any prime $l \in \A$, and thus $(C2)$ holds. Finally, using the same arguments as above, it follows from  $(iii)$ and $(\ref{the-equation-of-the-generating-prime-polynomial-P(x)-in-the-lemma-about-infinitude-of-the-triples-p-n-kappa-equation})$ that $\kappa \equiv \dfrac{1}{3} \pmod{l}$ for any prime $l \in \B^{\ast}$. Since $\B \subseteq \B^{\ast}$, we deduce that $\kappa \equiv \dfrac{1}{3} \pmod{l}$ for any prime $l \in \B$, and thus $(C3)$ is true. Therefore our contention follows.

\end{proof}

\begin{corollary}
\label{infinitude-of-non-isomorphic-generalized-Mordell-curves-violating-the-Hasse-principle-corollary}

Let $p$ be a prime such that $p \equiv 1 \pmod{8}$ and $p \equiv 2 \pmod{3}$. Let $n$ be an integer such that $n \ge 2$. Then there are infinitely many non-isomorphic generalized Mordell curves of degree $12n$ that are counterexamples to the Hasse principle explained by the Brauer-Manin obstruction.

\end{corollary}

\begin{proof}

Let $m$ be any integer such that $1 \le m < n$. Let $r$ be any positive integer such that $r$ divides $m$. Let $\C$ be the set of nonzero integers $\kappa$ satisfying the following conditions:

\begin{itemize}

\item [(1)] $\kappa = 3^{2m - 1}\kappa_{\ast}^r$ for some odd prime $\kappa_{\ast} \ne 3$, and

\item [(2)] $\kappa$ satisfies $(B3)-(B5)$ and $(C1)-(C3)$.

\end{itemize}
By Lemma \ref{infinitude-of-the-triples-p-n-kappa-lemma}, we know that $\C$ is of infinite cardinality. For each integer $\kappa \in \C$, let $\cD_{\kappa}$ be the smooth projective model of the affine curve defined by
\begin{align*}
\cD_{\kappa} : pz^2 = 3^6\kappa^6x^{12n} - 1.
\end{align*}
By Theorem \ref{the-first-infinite-family-of-generalized-Mordell-curves-violates-the-Hasse-principle-theorem}, we deduce that $\cD_{\kappa}$ is a counterexample to the Hasse principle explained by the Brauer-Manin obstruction for each $\kappa \in \C$. We contend that $\cD_{\kappa_1}$ is not isomorphic to $\cD_{\kappa_2}$ for any two distinct integers $\kappa_1, \kappa_2 \in \C$. Assume the contrary, that is, $\cD_{\kappa_1}$ is isomorphic to $\cD_{\kappa_2}$ for some distinct integers $\kappa_1, \kappa_2 \in \C$. Hence we deduce that the rational number $\dfrac{3^6\kappa^6_1}{3^6\kappa_2^6}$ is a $12n$-th power. Since $\kappa_1$ and $\kappa_2$ are in $\C$, we can write
\begin{align*}
\kappa_1 &= 3^{2m - 1}\kappa_{\ast, 1}^r, \\
\kappa_2 &= 3^{2m - 1}\kappa_{\ast, 2}^r,
\end{align*}
where $\kappa_{\ast, 1}$ and $\kappa_{\ast, 2}$ are distinct odd primes such that $\kappa_{\ast, 1} \ne 3$ and $\kappa_{\ast, 2} \ne 3$. Thus it follows that $\left(\dfrac{\kappa_{\ast, 1}}{\kappa_{\ast, 2}}\right)^{6r}$ is a $12n$-th power. Hence there exists a rational number $\rho \in \bQ$ such that
\begin{align*}
\left(\dfrac{\kappa_{\ast, 1}}{\kappa_{\ast, 2}}\right)^{6r} = (\rho)^{12n}.
\end{align*}
Applying the $\kappa_{\ast, 1}$-adic valuation to both sides of the above identity, we deduce that
\begin{align*}
6r = v_{\kappa_{\ast, 1}}\left(\left(\dfrac{\kappa_{\ast, 1}}{\kappa_{\ast, 2}}\right)^{6r}\right) = v_{\kappa_{\ast, 1}}\left((\rho)^{12n}\right) = 12nv_{\kappa_{\ast, 1}}(\rho),
\end{align*}
and hence
\begin{align*}
r = 2nv_{\kappa_{\ast, 1}}(\rho).
\end{align*}
Thus we deduce that $n$ divides $r$, which is a contradiction since $1 \le r \le m < n$. Therefore $\cD_{\kappa_1}$ is not isomorphic to $\cD_{\kappa_2}$ for any two distinct integers $\kappa_1, \kappa_2 \in \C$. Hence we deduce that $\left(\cD_{\kappa}\right)_{\kappa \in \C}$ is an infinite family of non-isomorphic generalized Mordell curves of degree $12n$ such that each member in the family $\left(\cD_{\kappa}\right)_{\kappa \in \C}$ is a counterexample to the Hasse principle explained by the Brauer-Manin obstruction. Therefore our contention follows.

\end{proof}

\section{The descending chain condition on sequences of curves}
\label{the-descending-chain-condition-on-sequences-of-curves-section}

In this section, we will introduce the notion of the \textit{descending chain condition (DCC) on sequences of curves}. This notation provides an analogue of the notion of the descending chain condition on partially ordered sets. Moreover, we will show that the families of generalized Mordell curves constructed in Section \ref{certain-generalized-Mordell-curves-violate-the-Hasse-principle-section} satisfy the DCC.

\begin{definition}
\label{the-descending-chain-condition-on-a-sequence-of-curves-definition}

Let $\left(\cX_i\right)_{i \ge 1}$ be a sequence of smooth geometrically irreducible projective curves over a global field $k$. For each integer $i \ge 1$, let $\phi_i : \cX_{i + 1} \longrightarrow \cX_i$ be a $k$-morphism of curves. We say that $\left(\cX_i, \phi_i\right)_{i \ge 1}$ satisfies the \textit{descending chain condition} (DCC) if there exists a positive integer $n$ such that the following are true:

\begin{itemize}

\item [(DCC1)] $\cX_i(k) \ne \emptyset$ for each $1 \le i \le n - 1$.

\item [(DCC2)] $\cX_n$ is a counterexample to the Hasse principle explained by the Brauer-Manin obstruction, i.e., $\cX_n(\bA_k)^{\mathrm{Br}} = \emptyset$ but $\cX_n(\bA_k) \ne \emptyset$.

\item [(DCC3)] $g_i < g_j$ for any two positive integers $i, j$ with $1 \le i < j$, where $g_i$ is the genus of $\cX_i$ for each positive integer $i$.

\end{itemize}

\end{definition}

It is not difficult to see that the integer $n$ in Definition \ref{the-descending-chain-condition-on-a-sequence-of-curves-definition} is \textit{unique}. When $n$ satisfies (DCC1)-(DCC3), we define the \textit{length} of $\left(\cX_i, \phi_i\right)_{i \ge 1}$ to be $n$.

We now explain why Definition \ref{the-descending-chain-condition-on-a-sequence-of-curves-definition} provides an analogue of the descending chain condition on partially ordered sets. Suppose that $\left(\cX_i, \phi_i\right)_{i \ge 1}$ is a sequence of smooth geometrically irreducible projective curves over a global field $k$, where the $\phi_i : \cX_{i + 1} \longrightarrow \cX_i$ are $k$-morphisms of curves such that $\left(\cX_i, \phi_i\right)_{i \ge 1}$ satisfies the DCC in the sense of Definition \ref{the-descending-chain-condition-on-a-sequence-of-curves-definition}. Let $n$ be the length of $\left(\cX_i, \phi_i\right)_{i \ge 1}$. For $i = 1$, set
\begin{align*}
a_1 := \#\cX_1(k) \in \bZ_{\ge 0} \cup \{\infty\},
\end{align*}
and for each integer $i \ge 2$, define
\begin{align}
\label{the-size-of-the-set-of-rational-points-on-X-i-equations}
a_i := \#(\phi_1 \circ \phi_2 \circ \cdots \circ \phi_{i - 1})\left(\cX_i(k)\right) \in \bZ_{\ge 0} \cup \{\infty\}.
\end{align}

The sequence $\left(\cX_i, \phi_i\right)_{i \ge 1}$ can be written in the form
\begin{align}
\label{another-form-of-the-sequence-of-curves-equation}
\cdots \longrightarrow \cX_{n + 1} \xrightarrow{\phi_{n}} \cX_{n} \xrightarrow{\phi_{n - 1}} \cX_{n - 1} \xrightarrow{\phi_{n - 2}} \cdots \longrightarrow \cX_{5} \xrightarrow{\phi_{4}} \cX_{4} \xrightarrow{\phi_{3}} \cX_3 \xrightarrow{\phi_{2}} \cX_{2} \xrightarrow{\phi_{1}} \cX_1.
\end{align}
It follows from $(\ref{the-size-of-the-set-of-rational-points-on-X-i-equations})$ and $(\ref{another-form-of-the-sequence-of-curves-equation})$ that the sequence $(a_i)_{i \ge 1}$ forms a descending chain of integers, that is, we have that
\begin{align}
\label{the-descending-chain-of-integers-a-i-equations}
\cdots \le a_{n + 2} \le a_{n + 1} \le a_n \le \cdots \le a_5 \le a_4 \le a_3 \le a_2 \le a_1.
\end{align}
By (DCC2), we know that $\cX_n(\bA_k)^{\mathrm{Br}} = \emptyset$. Since $\cX_n(k) \subseteq \cX_n(\bA_k)^{\mathrm{Br}}$, it follows that $\cX_n(k)$ is empty,
and hence $a_n = 0$. Since $(\phi_i)_{i \ge 1}$ is a sequence of $k$-morphisms of curves and $\cX_n(k) = \emptyset$, it follows from $(\ref{another-form-of-the-sequence-of-curves-equation})$ that $\cX_m(k) = \emptyset$ for each $m \ge n$, and thus $a_m = 0$ for each $m \ge n$. By (DCC1), we know that $a_1 = \#\cX_1(k) \ge 1$. Furthermore, we also know that $\#\cX_i(k) \ge 1$ for each $2 \le i \le n - 1$, and therefore
\begin{align*}
a_i := \#(\phi_1 \circ \phi_2 \circ \cdots \circ \phi_{i - 1})\left(\cX_i(k)\right) \ge 1
\end{align*}
for each $2 \le i \le n - 1$. Thus we see that the sequence $(a_i)_{i \ge 1}$ satisfies the descending chain condition in the usual sense; in other words, the descending chain $(\ref{the-descending-chain-of-integers-a-i-equations})$ eventually terminates, and satisfies the following two conditions:

\begin{itemize}

\item [(i)] $1 \le a_{n - 1} \le a_{n - 2} \le \cdots \le a_3 \le a_2 \le a_1$, and

\item [(ii)] $a_m = a_n = 0$ for each $m \ge n$.

\end{itemize}
Note that the descending sequence $(a_i)_{i \ge 1}$ not only eventually stabilizes but eventually becomes zero.

The above discussion shows that there exists a mapping from the collection of sequences of curves satisfying the DCC in the sense of Definition \ref{the-descending-chain-condition-on-a-sequence-of-curves-definition} to the collection of descending sequences of integers satisfying the DCC in the usual sense and eventually becoming zero. More explicitly, we have that
\begin{align*}
\begin{pmatrix}
\text{sequences of curves $\left(\cX_i, \phi_i\right)_{i \ge 1}$} \\
\text{satisfying the DCC } \\
\text{in the sense of Definition \ref{the-descending-chain-condition-on-a-sequence-of-curves-definition}} \\
\end{pmatrix}
\xrightarrow{\Phi}
\begin{pmatrix}
\text{sequences of integers $(a_i)_{i \ge 1}$} \\
\text{satisfying the DCC} \\
\text{in the usual sense} \\
\text{and eventually becoming zero}
\end{pmatrix}
,
\end{align*}
where $\Phi$ is the map sending a sequence $\left(\cX_i, \phi_i\right)_{i \ge 1}$ of curves to a sequence $(a_i)_{i \ge 1}$ of integers defined by
\begin{align*}
a_i :=
\begin{cases}
\#\cX_1(k) \; \; &\text{if $i = 1$,} \\
\#(\phi_1 \circ \phi_2 \circ \cdots \circ \phi_{i - 1})\left(\cX_i(k)\right) \; \; &\text{if $i \ge 2$.}
\end{cases}
\end{align*}

\begin{remark}

Let $\left(\cX_i, \phi_i\right)_{i \ge 1}$ be a sequence of smooth geometrically irreducible curves over a global field $k$ such that $\left(\cX_i, \phi_i\right)_{i \ge 1}$ satisfies the DCC and the length of $\left(\cX_i, \phi_i\right)_{i \ge 1}$ is $n$. Our main motivation to study the DCC on $\left(\cX_i, \phi_i\right)_{i \ge 1}$ is to address the situation where $\cX_n$ has no rational points over $k$ although it tries hard to possess a $k$-rational point. To be more precise, we define
\begin{align*}
\psi_i := \phi_i \circ \phi_{i + 1} \circ \cdots \circ \phi_{n - 1} : \cX_n \longrightarrow \cX_i
\end{align*}
for each $i \ge 1$. We see that there are $n - 1$ morphisms $\psi_i : \cX_n \longrightarrow \cX_i$ for $1 \le i \le n - 1$. By assumption, we know that $\#\cX_i(k) \ge 1$ for each $1 \le i \le n - 1$ and $\#\cX_n(k) = 0$. For each $1 \le i \le n - 1$, we see that although $\cX_i$ has at least one $k$-rational point, there exist no $k$-rational points in $\cX_i(k)$ whose preimage under the $k$-morphism $\psi_i$ belongs to $\cX_n(k)$. Furthermore, by (DDC2), we know that $\cX_n(\bA_k) \ne \emptyset$ and $\cX_n(k)$ is a subset of $\cX_n(\bA_k)$. Therefore, although $\cX_n$ tries many ways to obtain at least one $k$-rational point, there still exist no $k$-rational points on $\cX_n$.

\end{remark}

For each positive integer, it is natural to ask whether there exists a sequence $\left(\cX_i, \phi_i\right)_{i \ge 1}$ of curves that satisfies the DCC of length $n$. We will answer this question in the affirmative in a stronger form. More precisely, we will prove that for any positive integer $n$, there exist infinitely many sequences $\left(\cX_i, \phi_i\right)_{i \ge 1}$ of curves that satisfy the DCC of length $n$.

\begin{lemma}
\label{the-existence-of-the-set-C-n-m-such-that-kappa-in-C-n-m-generates-the-sequence-satisfying-the-DCC-lemma}

Let $p$ be a prime such that $p \equiv 1 \pmod{8}$ and $p \equiv 2 \pmod{3}$. Let $n$ and $m$ be positive integers such that $2m < n$. Then there exists an infinite set $\C_{n, m}$ of integers such that for each integer $\kappa \in \C_{n, m}$, $\cD_{\kappa}^{(m)}$ contains at least two rational points in its affine locus whereas $\cD_{\kappa}^{(n)}$ is a counterexample to the Hasse principle explained by the Brauer-Manin obstruction, where for each $\kappa \in \C_{n, m}$, $\cD_{\kappa}^{(m)}$ and $\cD_{\kappa}^{(n)}$ are the generalized Mordell curves of degree $12m$ and $12n$ defined by
\begin{align}
\label{the-curve-D-kappa-m-equation}
\cD_{\kappa}^{(m)}: pz^2 = 3^6\kappa^6x^{12m} - 1
\end{align}
and
\begin{align}
\label{the-curve-D-kappa-n-equation}
\cD_{\kappa}^{(n)}: pz^2 = 3^6\kappa^6x^{12n} - 1,
\end{align}
respectively.

\end{lemma}

\begin{proof}

Let $\C_{n, m}$ be the set of nonzero integers $\kappa$ satisfying the following two conditions:

\begin{itemize}

\item [(i)] $\kappa = 3^{4m - 1}\kappa_{\ast}^{2m}$ for some odd prime $\kappa_{\ast} \ne 3$, and

\item [(ii)] $\kappa$ satisfies $(B3)-(B5)$ in Section \ref{a-subset-of-the-set-of-rational-points-on-X-p-section} and $(C1)-(C3)$ in Theorem \ref{the-first-infinite-family-of-generalized-Mordell-curves-violates-the-Hasse-principle-theorem}.

\end{itemize}
Using Lemma \ref{infinitude-of-the-triples-p-n-kappa-lemma} with $(m, r)$ replaced by $(2m, 2m)$, we deduce that $\C_{m, n}$ is of infinite cardinality. Let $\kappa$ be any nonzero integer in $\C_{m, n}$, and let $\cD_{\kappa}^{(m)}$ and $\cD_{\kappa}^{(n)}$ be the smooth projective models defined by $(\ref{the-curve-D-kappa-m-equation})$ and $(\ref{the-curve-D-kappa-n-equation})$, respectively. Since $\kappa \in \C_{m, n}$, we know that there exists an odd prime $\kappa_{\ast} \ne 3$ such that $\kappa = 3^{4m - 1}\kappa_{\ast}^{2m}$. The defining equation of  $\cD_{\kappa}^{(m)}$ can be written in the form
\begin{align*}
\cD_{\kappa}^{(m)}: pz^2 = \left(3^2\kappa_{\ast}x\right)^{12m} - 1,
\end{align*}
and hence we see that the points $(x, z) = \left(\pm \dfrac{1}{3^2\kappa_{\ast}}, 0 \right)$ belong to $\cD_{\kappa}^{(m)}(\bQ)$. Thus $\cD_{\kappa}^{(m)}$ has at least two $\bQ$-rational points.

We now prove that $\cD_{\kappa}^{(n)}$ is a counterexample to the Hasse principle explained by the Brauer-Manin obstruction. Indeed, we know that $\kappa$ satisfies $(B3)-(B5)$ and $(C1)-(C3)$. By Theorem \ref{the-first-infinite-family-of-generalized-Mordell-curves-violates-the-Hasse-principle-theorem}, we deduce that $\cD_{\kappa}^{(n)}$ is a counterexample to the Hasse principle explained by the Brauer-Manin obstruction, and thus our contention follows.

\end{proof}

\begin{remark}
\label{Remark-the-degree-12n-is-optimal-for-generalized-Mordell-curves-D}

Let $p$ be a prime such that $p \equiv 1 \pmod{8}$ and $p \equiv 2 \pmod{3}$. Take a positive integer $n \ge 3$, and let $\kappa$ be a nonzero integer. For any positive integer $s$ dividing $n$ with $s \ge 2$, define $m_s := \dfrac{n}{s}$. Let $\cD_{\kappa}^{(n)}$ and $\cD_{\kappa}^{(m_s)}$ be the generalized Mordell curves of degree $12n$ and $12m_s$ defined by
\begin{align}
\label{Definition-D-kappa-n-in-Remark-about-the-optimal-degree-for-D}
\cD_{\kappa}^{(n)}: pz^2 = 3^6\kappa^6x^{12n} - 1
\end{align}
and
\begin{align}
\label{Definition-D-kappa-m-s-in-Remark-about-the-optimal-degree-for-D}
\cD_{\kappa}^{(m_s)}: pz^2 = 3^6\kappa^6x^{12m_s} - 1,
\end{align}
respectively. Then there is a $\bQ$-morphism $\phi_{m_s} : \cD_{\kappa}^{(n)} \rightarrow \cD_{\kappa}^{(m_s)}$ defined by
\begin{align*}
\phi_{m_s} : \cD_{\kappa}^{(n)} &\rightarrow \cD_{\kappa}^{(m_s)} \\
         (x, z) &\mapsto (x^s, z).
\end{align*}
If $\kappa$ satisfies $(B3)-(B5)$ and $(C1)-(C3)$ in Theorem \ref{the-first-infinite-family-of-generalized-Mordell-curves-violates-the-Hasse-principle-theorem}, then we know that $\cD_{\kappa}^{(n)}$ is a counterexample to the Hasse principle explained by the Brauer-Manin obstruction. Under the morphisms $\phi_{m_s}$, it is natural to ask whether one can imply nonexistence of rational points on $\cD_{\kappa}^{(n)}$ by first showing that there exists a positive integer $s \ge 2$ such that $s$ divides $n$ and the curve $\cD_{\kappa}^{(m_s)}$ has no rational points, where $m_s = \dfrac{n}{s}$. Lemma \ref{the-existence-of-the-set-C-n-m-such-that-kappa-in-C-n-m-generates-the-sequence-satisfying-the-DCC-lemma} shows that it is impossible, and thus it proves that the degree $12n$ of the generalized Mordell curves $\cD$ in Theorem \ref{the-first-infinite-family-of-generalized-Mordell-curves-violates-the-Hasse-principle-theorem} is \textit{optimal} in the sense that one can not replace $n$ by a positive divisor $m_s$ of $n$ with $m_s \ne n$. Indeed, assume further that $n$ is odd, and take any positive integer $s \ge 2$ such that $s$ divides $n$. Define $m_s := \dfrac{n}{s}$. Since $s \ge 2$, we deduce that $2m_s \le sm_s = n$. Since $n$ is odd, we deduce that $2m_s < n$. Hence Lemma \ref{the-existence-of-the-set-C-n-m-such-that-kappa-in-C-n-m-generates-the-sequence-satisfying-the-DCC-lemma} shows that there exists an infinite set $\C_{n, m_s}$ of integers such that for each $\kappa \in \C_{n, m_s}$, $\cD^{(m_s)}_{\kappa}$ contains at least two rational points in its affine locus whereas $\cD^{(n)}_{\kappa}$ is a counterexample to the Hasse principle explained by the Brauer-Manin obstruction, where $\cD^{(m_s)}_{\kappa}$ and $\cD^{(n)}_{\kappa}$ are defined by $(\ref{Definition-D-kappa-m-s-in-Remark-about-the-optimal-degree-for-D})$ and $(\ref{Definition-D-kappa-n-in-Remark-about-the-optimal-degree-for-D})$, respectively.

\end{remark}

We now prove the main result in this section, which says that for a given positive integer $h$, there are infinitely many sequences of curves satisfying the DCC of length $h$.

\begin{corollary}
\label{the-existence-of-infinitude-of-sequences-of-curves-satisfying-DCC-corollary}

Let $h$ be a positive integer. Then there are infinitely many sequences $\left(\cX_i, \phi_i\right)_{i \ge 1}$ of generalized Mordell curves that satisfy the DCC of length $h$ in the sense of Definition \ref{the-descending-chain-condition-on-a-sequence-of-curves-definition}, where for each $i \ge 1$, $\cX_i$ is a generalized Mordell curve, and for each $i \ge 1$,
\begin{align*}
\phi_i : \cX_{i + 1} \rightarrow \cX_i
\end{align*}
is a $k$-morphism of curves.

\end{corollary}

\begin{proof}

Let $p$ be a prime such that $p \equiv 1 \pmod{8}$ and $p \equiv 2 \pmod{3}$. Let $n_0$ and $n_1$ be integers such that $n_0 \ge 2$ and $n_1 \ge 1$. Set
\begin{align*}
n &:= n_0^{h + \epsilon}n_1, \\
m &:= n_0^{h}n_1,
\end{align*}
where $\epsilon$ is a positive integer such that $n_0^{\epsilon} > 2$. This means that if $n_0 = 2$, then $\epsilon$ is at least two, and that if $n_0 > 2$, then $\epsilon \ge 1$. We see that
\begin{align*}
2m < mn_0^{\epsilon} = n.
\end{align*}
Applying Lemma \ref{the-existence-of-the-set-C-n-m-such-that-kappa-in-C-n-m-generates-the-sequence-satisfying-the-DCC-lemma} for the triple $(p, n, m)$, we deduce that there exists an infinite set $\C_{n, m}$ of nonzero integers $\kappa$ satisfying conditions $(i)$ and $(ii)$ as in the proof of Lemma \ref{the-existence-of-the-set-C-n-m-such-that-kappa-in-C-n-m-generates-the-sequence-satisfying-the-DCC-lemma} such that for any $\kappa \in \C_{n, m}$, $\cX_{h - 1}^{(\kappa)}$ contains at least two rational points in its affine locus whereas $\cX_h^{(\kappa)}$ is a counterexample to the Hasse principle explained by the Brauer-Manin obstruction, where for each $\kappa \in \C_{n, m}$, $\cX_{h - 1}^{(\kappa)}$ and $\cX_{h}^{(\kappa)}$ are the generalized Mordell curves defined by
\begin{align*}
\cX_{h - 1}^{(\kappa)} : pz^2 = 3^6\kappa^6x^{12n_0^{h}n_1} - 1
\end{align*}
and
\begin{align*}
\cX_{h}^{(\kappa)} : pz^2 = 3^6\kappa^6x^{12n_0^{h + \epsilon}n_1} - 1,
\end{align*}
respectively.

Now take an integer $\kappa \in \C_{n, m}$, and let $\cX_{h - 1}^{(\kappa)}$ and $\cX_h^{(\kappa)}$ be the smooth projective models as above. Following the proof of Lemma \ref{the-existence-of-the-set-C-n-m-such-that-kappa-in-C-n-m-generates-the-sequence-satisfying-the-DCC-lemma} and by the definition of $\C_{n, m}$, we know that $\kappa = 3^{4m - 1}\kappa_{\ast}^{2m}$ for some odd prime $\kappa_{\ast} \ne 3$. We define
\begin{align*}
\psi_{h - 1}^{(\kappa)} : \cX_{h}^{(\kappa)} &\rightarrow \cX_{h - 1}^{(\kappa)} \\
                        (x,z)   &\mapsto (x^{n_0^{\epsilon}},z).
\end{align*}
For each integer $i \ge h + 1$, let $\cX_{i}^{(\kappa)}$ be the smooth projective model of the affine curve defined by
\begin{align*}
\cX_{i}^{(\kappa)} : pz^2 = 3^6\kappa^6x^{12n_0^{i + \epsilon}n_1} - 1,
\end{align*}
and for each integer $i \ge h$, let $\psi_{i}: \cX_{i + 1}^{(\kappa)} \longrightarrow \cX_{i}^{(\kappa)}$ be the $\bQ$-morphism of curves defined by
\begin{align*}
\psi_{i}^{(\kappa)} : \cX_{i + 1}^{(\kappa)} &\rightarrow \cX_{i}^{(\kappa)} \\
                        (x,z)   &\mapsto (x^{n_0},z).
\end{align*}
For each integer $1 \le i \le h - 2$, let $\cX_{i}^{(\kappa)}$ be the smooth projective model of the affine curve defined by
\begin{align*}
\cX_{i}^{(\kappa)} : pz^2 = 3^6\kappa^6x^{12n_0^{i + 1}n_1} - 1,
\end{align*}
and for each integer $1 \le i \le h - 2$, let $\psi_{i}: \cX_{i + 1}^{(\kappa)} \longrightarrow \cX_{i}^{(\kappa)}$ be the $\bQ$-morphism of curves defined by
\begin{align*}
\psi_{i}^{(\kappa)} : \cX_{i + 1}^{(\kappa)} &\rightarrow \cX_{i}^{(\kappa)} \\
                        (x,z)   &\mapsto (x^{n_0},z).
\end{align*}
Hence we have defined a sequence of curves $\left(\cX_{i}^{(\kappa)}, \psi_i^{(\kappa)}\right)_{i \ge 1}$. We contend that $\left(\cX_{i}^{(\kappa)}, \psi_i^{(\kappa)}\right)_{i \ge 1}$ satisfies the DCC of length $h$ in the sense of Definition \ref{the-descending-chain-condition-on-a-sequence-of-curves-definition}. Indeed, we have shown above that $\cX_h^{(\kappa)}$ is a counterexample to the Hasse principle explained by the Brauer-Manin obstruction, and hence $\left(\cX_{i}^{(\kappa)}, \psi_i^{(\kappa)}\right)_{i \ge 1}$ satisfies (DCC2) in Definition \ref{the-descending-chain-condition-on-a-sequence-of-curves-definition}. We also know that $\cX_{h - 1}^{(\kappa)}$ contains at least two rational points in its affine locus. Let $(x_0, z_0)$ be any rational point in $\cX_{h - 1}^{(\kappa)}(\bQ)$. Then one can check that for each integer $1 \le i \le h - 2$, the point $(x, z) := (x_0^{n_0^{h - i - 1}}, z_0)$ belongs to $\cX_{i}^{(\kappa)}(\bQ)$. Hence (DCC1) is true. It remains to prove that $\left(\cX_{i}^{(\kappa)}, \psi_i^{(\kappa)}\right)_{i \ge 1}$ satisfies (DCC3). For each $i \ge 1$, we denote by $g_i^{(\kappa)}$ the genus of the curve $\cX_i^{(\kappa)}$. We see that
\begin{align*}
g_i^{(\kappa)} =
\begin{cases}
6n_0^{i + \epsilon}n_1 - 1 \; \; &\text{if $i \ge h$,} \\
6n_0^{i + 1}n_1 - 1 \; \; &\text{if $1 \le i \le h - 1$.}
\end{cases}
\end{align*}
Hence it follows that $g_i^{(\kappa)} < g_j^{(\kappa)}$ for any positive integers $i, j$ with $1 \le i < j$, and thus (DCC3) is true. Therefore $\left(\cX_{i}^{(\kappa)}, \psi_i^{(\kappa)}\right)_{i \ge 1}$ satisfies the DCC of length $h$ in the sense of Definition \ref{the-descending-chain-condition-on-a-sequence-of-curves-definition}. Since $\C_{n, m}$ is of infinite cardinality, our contention follows immediately.

\end{proof}

The above corollary shows that for a given positive integer $h \ge 1$, there are infinitely many sequences of smooth geometrically irreducible curves over $\bQ$ such that they satisfy the DCC of length $h$. It is natural to ask whether or not there exists a sequence of smooth geometrically irreducible curves such that it does not satisfy the DCC. The following result shows that there exist infinitely many such sequences of curves over $\bQ$.

\begin{proposition}
\label{the-existence-of-infinitude-of-sequences-of-curves-not-satisfying-the-DCC-proposition}

There exist infinitely many sequences of smooth geometrically irreducible curves over $\bQ$ that do not satisfy the DCC in the sense of Definition \ref{the-descending-chain-condition-on-a-sequence-of-curves-definition}.

\end{proposition}

\begin{proof}

Let $n$ and $m$ be positive integers such that $n \ge 2$ and $m \ge 2$. Let $F(x)$ be a separable polynomial of degree $n$ in $\bQ[x]$ such that $F(0) \ne 0$ and $F(1) = 0$. Note that there are infinitely many such polynomials $F(x)$; for example, one can take $F(x) = x^n - 1$. Since $F(0)$ is nonzero and $F(x)$ is separable, we see that $F(x^{m^i})$ is separable for each positive integer $i \ge 1$. For each positive integer $i \ge 1$, let $\cX_i$ be the smooth projective model of the affine curve defined by
\begin{align*}
\cX_i : z^2 = F(x^{m^i}).
\end{align*}
For each $i \ge 1$, we denote by $g_i$ the genus of $\cX_i$. For each $i \ge 1$, we see that
\begin{align*}
g_i =
\begin{cases}
\dfrac{nm^i - 2}{2} \; \; &\text{if $nm \equiv 0 \pmod{2}$,} \\
\dfrac{nm^i - 1}{2} \; \; &\text{if $nm \equiv 1 \pmod{2}$.} \\
\end{cases}
\end{align*}
Hence we deduce that $g_i < g_j$ for any positive integers $i, j$ with $1 \le i < j$. Since $F(1) = 0$, we see that the point $(x, z) = (1, 0)$ belongs to $\cX_i(\bQ)$ for each $i \ge 1$. Furthermore, for each $i \ge 1$, we define
\begin{align*}
\psi_i : \cX_{i + 1} &\longrightarrow \cX_i \\
           (x, z)    &\mapsto (x^{m}, z).
\end{align*}
Hence we have defined a sequence $\left(\cX_{i}, \psi_i \right)_{i \ge 1}$ of smooth geometrically irreducible curves over $\bQ$ such that $\cX_i(\bQ) \ne 0$ for each $i \ge 1$ and $g_i < g_j$ for any positive integers $i, j$ with $1 \le i < j$. Thus $\left(\cX_{i}, \psi_i \right)_{i \ge 1}$ does not satisfy the DCC, and therefore our contention follows.

\end{proof}

\section{Certain generalized Fermat curves violating the Hasse principle}
\label{certain-generalized-Fermat-curves-violate-the-Hasse-principle-section}

In this section, we will give a sufficient condition under which certain generalized Fermat curves of signature $(12n, 12n, 12n)$ with $n \ge 2$ are counterexamples to the Hasse principle explained by the Brauer-Manin obstruction. In the next section, using this sufficient condition, we will show that for each positive integer $n \ge 2$, there exist infinitely many generalized Fermat curves of signature $(12n, 12n, 12n)$ that are counterexamples to the Hasse principle explained by the Brauer-Manin obstruction. We begin by recalling the following useful result.

\begin{lemma}
\label{functoriality-azumaya-lemma}
(see \cite[Lemma $4.8$]{coray-manoil})

Let $k$ be a number field and let $\cV_1$ and $\cV_2$ be (proper) $k$-varieties. Assume that there is a $k$-morphism $\Psi : \cV_1 \rightarrow \cV_2$ and $\cV_2(\bA_k)^{\mathrm{Br}} = \emptyset$. Then $\cV_1(\bA_k)^{\mathrm{Br}} = \emptyset$.

\end{lemma}

The following theorem is our main result in this section.

\begin{theorem}
\label{the-first-family-of-certain-generalized-Fermat-curves-violating-the-Hasse-principle-theorem}

Let $p$ be a prime such that $p \equiv 1 \pmod{8}$ and $p \equiv 2 \pmod{3}$. Let $n$ be an integer such that $n \ge 2$, and let $\chi$ be a nonzero odd integer. Let $\kappa$ be a nonzero integer satisfying $(B3)$, $(B4)$ and $(B5)$ in Section \ref{a-subset-of-the-set-of-rational-points-on-X-p-section}. Assume further that the following are true:

\begin{itemize}

\item [(D1)] $\kappa \equiv \dfrac{1}{3} \pmod{p^{2v_p(n) + 1}}$.

\item [(D2)] $p\chi^2 \equiv 3^6\kappa^6 \pmod{2^{2v_2(n) + 5}}$.

\item [(D3)] $p\chi^2 \equiv -1 \pmod{3^{2v_3(n) + 3}}$.

\item [(D4)] let $\D$ be the set of odd primes $l$ satisfying the following two conditions:

\begin{itemize}

\item [(i)] $\gcd(l, 3) = \gcd(l, p) = 1$, and

\item [(ii)] $l$ divides $\kappa$.

\end{itemize}
For each prime $l \in \D$, we assume that $l \equiv 1 \pmod{4}$ and $p\chi^2 \equiv -1 \pmod{l^{2v_l(n) + 1}}$.

\item [(D5)] let $\E$ be the set of odd primes $l$ satisfying the following two conditions:

\begin{itemize}

\item [(i)] $\gcd(l, 3) = \gcd(l, p) = \gcd(l, \kappa) = 1$, and

\item [(ii)] $l$ divides $\chi$.

\end{itemize}
For each prime $l \in \E$, we assume that there exists an integer $\zeta_l$ with $\zeta_l \not\equiv 0 \pmod{l}$ such that $\kappa \equiv \dfrac{\zeta_l^{2n}}{3} \pmod{l^{2v_l(n) + 1}}$ or $\kappa \equiv -\dfrac{\zeta_l^{2n}}{3} \pmod{l^{2v_l(n) + 1}}$.

\item [(D6)] let $\F$ be the set of odd primes $l$ satisfying the following two conditions:

\begin{itemize}

\item [(i)] $\gcd(l, 3) = \gcd(l, p) = \gcd(l, \kappa) = \gcd(l, \chi) = 1$, and

\item [(ii)] $l$ divides $n$.

\end{itemize}
For each prime $l \in \F$, we assume that $\kappa \equiv \dfrac{1}{3} \pmod{l^{2v_l(n) + 1}}$.

\item [(D7)] let $\G$ be the set of odd primes $l$ satisfying the following two conditions:

\begin{itemize}

\item [(i)] $\gcd(l, 3) = \gcd(l, p) = \gcd(l, \kappa) = \gcd(l, \chi) = \gcd(l, n) = 1$, and

\item [(ii)] $l \le 4(6n - 1)^2(12n - 1)^2$.

\end{itemize}
For each prime $l \in \G$, we assume that $\kappa \equiv \dfrac{1}{3} \pmod{l}$.

\end{itemize}
Let $\cF$ be the generalized Fermat curve of signature $(12n, 12n, 12n)$ defined by
\begin{align}
\label{the-generalized-Fermat-curve-F-equation}
\cF : 3^6\kappa^6x^{12n} - y^{12n} - p\chi^2z^{12n} = 0.
\end{align}
Then $\cF$ is a counterexample to the Hasse principle explained by the Brauer-Manin obstruction.

\end{theorem}

\begin{remark}
\label{the-genus-of-generalized-Fermat-curves-of-signature-12n-12n-12n-remark}

Let $\cF$ be the generalized Fermat curve of signature $(12n, 12n, 12n)$ given by $(\ref{the-generalized-Fermat-curve-F-equation})$. It is not difficult to see that the genus of $\cF$ is $(6n - 1)(12n - 1)$.

\end{remark}

\begin{proof}

Let $\cD$ be the smooth projective model in Corollary \ref{the-generalized-Mordell-curve-D-have-no-rational-points-corollary} defined by
\begin{align*}
\cD : pz^2 = 3^6\kappa^6x^{12n} - 1.
\end{align*}
By Corollary \ref{the-generalized-Mordell-curve-D-have-no-rational-points-corollary}, we know that $\cD(\bA_{\bQ})^{\mathrm{Br}} = \emptyset$. We define
\begin{align*}
\Psi :  \cF &\rightarrow \cD \\
(x : y : z) &\mapsto \left(x : y : \chi z^{6n}\right).
\end{align*}
It is clear that $\Psi$ is a $\bQ$-morphism from $\cF$ to $\cD$. Since $\cD(\bA_{\bQ})^{\mathrm{Br}} = \emptyset$, it follows from Lemma \ref{functoriality-azumaya-lemma} that $\cF(\bA_{\bQ})^{\mathrm{Br}} = \emptyset$. Hence it remains to prove that $\cF$ is everywhere locally solvable.

Note that if $l$ is an odd prime such that $l \ne 2, 3, p$ and $\gcd(l, \kappa\chi n) = 1$, then $\cF$ is nonsingular modulo $l$, and thus the genus of $\cF$ over the finite field $\bF_l$ is $(6n - 1)(12n - 1)$. Hence the Hasse-Weil bound (see Lemma \ref{Hasse-Weil-bound-lemma}) assures that $\cF$ is locally solvable at the primes $l$ such that $l > 4(6n - 1)^2(12n - 1)^2$, $l \ne 2, 3, p$ and $\gcd(l, \kappa\chi n) = 1$. Hence it suffices to consider the following cases.

$\star$ \textit{Case 1. $l = p$.}

We consider the system of equations
\begin{align}
\label{the-Hensel-system-at-p-of-the-generalized-Fermat-curve-F-equations}
\begin{cases}
G(x, y, z) := 3^6\kappa^6x^{12n} - y^{12n} - p\chi^2z^{12n} &\equiv 0 \pmod{p^{2v_p(n) + 1}} \\
\tfrac{\partial G}{\partial y}(x, y, z) = -12ny^{12n - 1}   &\equiv 0 \pmod{p^{v_p(n)}} \\
\tfrac{\partial G}{\partial y}(x, y, z) = -12ny^{12n - 1}   &\not\equiv 0 \pmod{p^{v_p(n) + 1}}.
\end{cases}
\end{align}
By $(D1)$, we know that
\begin{align*}
G(1, 1, 0) = 3^6\kappa^6 - 1 \equiv 0 \pmod{p^{2v_p(n) + 1}}.
\end{align*}
Since $p \ne 2, 3$, we deduce that
\begin{align*}
\dfrac{\partial G}{\partial y}(1, 1, 0) = -12n   &\equiv 0 \pmod{p^{v_p(n)}}, \\
\dfrac{\partial G}{\partial y}(1, 1, 0) = -12n   &\not\equiv 0 \pmod{p^{v_p(n) + 1}}.
\end{align*}
Hence $(x, y, z) = (1, 1, 0)$ is a solution to the system $(\ref{the-Hensel-system-at-p-of-the-generalized-Fermat-curve-F-equations})$, and thus it follows from Hensel's lemma that $\cF$ is locally solvable at $p$.

$\star$ \textit{Case 2. $l = 2$.}

We consider the system of equations
\begin{align}
\label{the-Hensel-system-at-2-of-the-generalized-Fermat-curve-F-equations}
\begin{cases}
G(x, y, z) := 3^6\kappa^6x^{12n} - y^{12n} - p\chi^2z^{12n} &\equiv 0 \pmod{2^{2v_2(n) + 5}} \\
\tfrac{\partial G}{\partial z}(x, y, z) = -12np\chi^2z^{12n - 1}   &\equiv 0 \pmod{2^{v_2(n) + 2}} \\
\tfrac{\partial G}{\partial z}(x, y, z) = -12np\chi^2z^{12n - 1}   &\not\equiv 0 \pmod{2^{v_2(n) + 3}}.
\end{cases}
\end{align}
By $(D2)$, we know that
\begin{align*}
G(1, 0, 1) = 3^6\kappa^6 - p\chi^2 \equiv 0 \pmod{2^{2v_2(n) + 5}}.
\end{align*}
Since $p \ne 2$ and $\chi$ is a nonzero odd integer, we see that
\begin{align*}
\dfrac{\partial G}{\partial z}(1, 0, 1) = -12np\chi^2 = -2^2\cdot3np\chi^2  &\equiv 0 \pmod{2^{v_2(n) + 2}}, \\
\dfrac{\partial G}{\partial z}(1, 0, 1) = -12np\chi^2 = -2^2\cdot3np\chi^2  &\not\equiv 0 \pmod{2^{v_2(n) + 3}}.
\end{align*}
Hence $(x, y, z) = (1, 0, 1)$ is a solution to the system $(\ref{the-Hensel-system-at-2-of-the-generalized-Fermat-curve-F-equations})$, and it thus follows from Hensel's lemma that $\cF$ is locally solvable at $2$.

$\star$ \textit{Case 3. $l = 3$.}

Using the same arguments as in \textit{Case 1} and \textit{Case 2}, it follows from $(D3)$ that $(x, y, z) = (0, 1, 1)$ is a solution to the system of equations
\begin{align*}
\begin{cases}
G(x, y, z) := 3^6\kappa^6x^{12n} - y^{12n} - p\chi^2z^{12n} &\equiv 0 \pmod{3^{2v_3(n) + 3}} \\
\tfrac{\partial G}{\partial y}(x, y, z) = -12ny^{12n - 1}   &\equiv 0 \pmod{3^{v_3(n) + 1}} \\
\tfrac{\partial G}{\partial z}(x, y, z) = -12ny^{12n - 1}    &\not\equiv 0 \pmod{3^{v_3(n) + 2}}.
\end{cases}
\end{align*}
By Hensel's lemma, we deduce that $\cF$ is locally solvable at $3$.

$\star$ \textit{Case 4. $l \in \D$.}

Using the same arguments as in \textit{Case 1} and \textit{Case 2}, we deduce from $(D4)$ that $(x, y, z) = (0, 1, 1)$ is a solution to the system of equations
\begin{align*}
\begin{cases}
G(x, y, z) := 3^6\kappa^6x^{12n} - y^{12n} - p\chi^2z^{12n} &\equiv 0 \pmod{l^{2v_l(n) + 1}} \\
\tfrac{\partial G}{\partial y}(x, y, z) = -12ny^{12n - 1}   &\equiv 0 \pmod{l^{v_l(n)}} \\
\tfrac{\partial G}{\partial y}(x, y, z) = -12ny^{12n - 1}   &\not\equiv 0 \pmod{l^{v_l(n) + 1}}.
\end{cases}
\end{align*}
By Hensel's lemma, we deduce that $\cF$ is locally solvable at $l$.

$\star$ \textit{Case 5. $l \in \E$.}

Using the same arguments as in \textit{Case 1} and \textit{Case 2}, we deduce from $(D5)$ that $(x, y, z) = (1, \zeta_l, 0)$ is a solution to the system of equations
\begin{align*}
\begin{cases}
G(x, y, z) := 3^6\kappa^6x^{12n} - y^{12n} - p\chi^2z^{12n} &\equiv 0 \pmod{l^{2v_l(n) + 1}} \\
\tfrac{\partial G}{\partial y}(x, y, z) = -12ny^{12n - 1}   &\equiv 0 \pmod{l^{v_l(n)}} \\
\tfrac{\partial G}{\partial y}(x, y, z) = -12ny^{12n - 1}   &\not\equiv 0 \pmod{l^{v_l(n) + 1}}.
\end{cases}
\end{align*}
By Hensel's lemma, we deduce that $\cF$ is locally solvable at $l$.

$\star$ \textit{Case 6. $l \in \F$.}

Using the same arguments as in \textit{Case 1} and \textit{Case 2}, we deduce from $(D6)$ that $(x, y, z) = (1, 1, 0)$ is a solution to the system of equations
\begin{align*}
\begin{cases}
G(x, y, z) := 3^6\kappa^6x^{12n} - y^{12n} - p\chi^2z^{12n} &\equiv 0 \pmod{l^{2v_l(n) + 1}} \\
\tfrac{\partial G}{\partial y}(x, y, z) = -12ny^{12n - 1}   &\equiv 0 \pmod{l^{v_l(n)}} \\
\tfrac{\partial G}{\partial y}(x, y, z) = -12ny^{12n - 1}   &\not\equiv 0 \pmod{l^{v_l(n) + 1}}.
\end{cases}
\end{align*}
By Hensel's lemma, we deduce that $\cF$ is locally solvable at $l$.

$\star$ \textit{Case 7. $l \in \G$.}

Using the same arguments as in \textit{Case 1} and \textit{Case 2}, we deduce from $(D7)$ that $(x, y, z) = (1, 1, 0)$ is a solution to the system of equations
\begin{align*}
\begin{cases}
G(x, y, z) := 3^6\kappa^6x^{12n} - y^{12n} - p\chi^2z^{12n} &\equiv 0 \pmod{l} \\
\tfrac{\partial G}{\partial y}(x, y, z) = -12ny^{12n - 1}   &\not\equiv 0 \pmod{l}.
\end{cases}
\end{align*}
By Hensel's lemma, we deduce that $\cF$ is locally solvable at $l$.

$\star$ \textit{Case 8. $l = \infty$.}

We see that the point $(x, y, z) = (1, (3^6\kappa^6)^{1/12n},0)$ belongs to $\cF(\bR)$.

Thus, by what we have shown, $\cF$ is everywhere locally solvable, and therefore our contention follows.

\end{proof}

\section{Infinitude of the quadruples $(p, n, \kappa, \chi)$}
\label{infinitude-of-the-quadruples-p-n-kappa-chi-section}

In this section, we will prove that there are infinitely many couples $(\kappa, \chi)$ satisfying $(D1)-(D7)$ in Theorem \ref{the-first-family-of-certain-generalized-Fermat-curves-violating-the-Hasse-principle-theorem}. Hence this implies that for each integer $n \ge 2$, there exist infinitely many generalized Fermat curves of signature $(12n, 12n, 12n)$ that are counterexamples to the Hasse principle explained by the Brauer-Manin obstruction. We begin by proving some elementary but very useful lemmas that we will need in the proof of the main result of this section.

\begin{lemma}
\label{the-simple-case-when-every-integer-h-is-an-s-th-power-in-certain-finite-fields-lemma}

Let $s$ be a positive odd integer. Let $(P, Q, R, S)$ be a quadruple of integers such that $R \ne 0$. Assume that the following are true:

\begin{itemize}

\item [(i)] $sP \equiv 1 \pmod{R}$ and $Q = 1 + \dfrac{sP - 1}{R}$.

\item [(ii)] $q := sS + Q$ is an odd prime.

\end{itemize}
Then every integer is an $s$-th power in $\bZ/q\bZ$.

\end{lemma}

\begin{proof}

Let $h$ be an integer. If $h \equiv 0 \pmod{q}$, then $h$ is an $s$-th power in $\bZ/q\bZ$. If $h \not\equiv 0 \pmod{q}$, then we see that
\begin{align*}
(h^{RS + P})^s= h^{sRS + sP} = h^{sS + Q}h^{(R - 1)(sS + Q - 1)} = h^qh^{(R - 1)(q - 1)} \equiv h \pmod{q}.
\end{align*}
Therefore $h$ is an $s$-th power in $\bZ/q\bZ$, which proves our contention.

\end{proof}

\begin{lemma}
\label{the-existence-of-primes-such-that-any-integer-is-s-th-power-lemma}

Let $s$ be a positive odd integer, and let $r$ be an integer such that $\gcd(r -  1, s) = 1$. Let $S$ be an integer, and define $q := sS + r$. Assume that $q$ is an odd prime. Then every integer is an $s$-th power in $\bZ/q\bZ$.

\end{lemma}

\begin{proof}

We show that there is a triple $(P, Q, R)$ of integers with $R \ne 0$ such that the quintuple $(q, P, Q, R, S)$ satisfies conditions $(i)$ and $(ii)$ in Lemma \ref{the-simple-case-when-every-integer-h-is-an-s-th-power-in-certain-finite-fields-lemma}. Indeed, since $\gcd(r - 1, s) = 1$, there exist nonzero integers $P, R$ such that
\begin{align*}
sP + (1 - r)R = 1.
\end{align*}
Hence we deduce that
\begin{align*}
sP - 1 = R(r - 1),
\end{align*}
and thus $sP \equiv 1 \pmod{R}$. Define
\begin{align*}
Q := 1 + \dfrac{sP - 1}{R}.
\end{align*}
Since $R$ divides $sP - 1$, we know that $Q$ is an integer. Since $sP - 1 = R(r - 1)$, we deduce that
\begin{align*}
Q = 1 + \dfrac{sP - 1}{R} = 1 + r - 1 = r,
\end{align*}
and hence $q = sS + r = sS + Q$. Thus it follows from Lemma \ref{the-simple-case-when-every-integer-h-is-an-s-th-power-in-certain-finite-fields-lemma} that every integer is an $s$-th power in $\bZ/q\bZ$.

\end{proof}

\begin{lemma}
\label{the-choice-of-sign-for-Delta-lemma}

Let $l$ be an odd prime, and let $r$ be an integer such that $\gcd(r, l) = 1$. Then at least one of the integers $r - 1$ and $-r - 1$ is relatively prime to $l$.

\end{lemma}

\begin{proof}

Assume the contrary, that is, $r - 1 \equiv 0 \pmod{l}$ and $-r - 1 \equiv 0 \pmod{l}$. Hence we deduce that
\begin{align*}
-2 \equiv (r - 1) + (-r - 1) \equiv 0 \pmod{l},
\end{align*}
which is a contradiction since $l$ is an odd prime. Thus our contention follows.

\end{proof}

\begin{lemma}
\label{every-square-is-a-2-power-in-the-finite-field-F-l-with-l-congruent-to-3-mod-4-lemma}

Let $l$ be an odd prime such that $l \equiv 3 \pmod{4}$. Let $x$ be an integer such that $x$ is a square in $\bF_l^{\times}$. Then, for any positive integer $n$, $x$ is a $2^n$-th power in $\bF_l^{\times}$, that is, there exists an integer $x_{\ast}$ such that $x_{\ast} \not\equiv 0 \pmod{l}$ and $x_{\ast}^{2^n} \equiv x \pmod{l}$.

\end{lemma}

\begin{proof}

We prove Lemma \ref{every-square-is-a-2-power-in-the-finite-field-F-l-with-l-congruent-to-3-mod-4-lemma} by induction over $n$.

If $n = 1$, then it follows immediately from the assumption that $x$ is a square in $\bF_l^{\times}$. Assume that Lemma \ref{every-square-is-a-2-power-in-the-finite-field-F-l-with-l-congruent-to-3-mod-4-lemma} is true for $n - 1$ with $n \ge 2$. We will prove that it is also true for $n$. Indeed, by the induction hypothesis, we know that there is an integer $h$ such that $h^{2^{n - 1}} \equiv x \pmod{l}$. Since $x$ belongs to $\bF_l^{\times}$, the last congruence implies that $h$ belongs to $\bF_l^{\times}$. If $h$ is a square in $\bF_l^{\times}$, then it follows that $x_{\ast}^{2^n} \equiv x \pmod{l}$, where $x_{\ast}$ is an integer such that $x_{\ast}^2 \equiv h \pmod{l}$. Assume now that $h$ is not a square in $\bF_l^{\times}$, that is, $\left(\dfrac{h}{l}\right) = - 1$, where $\left(\dfrac{\cdot}{\cdot}\right)$ denotes the Jacobi symbol. Since $l \equiv 3 \pmod{4}$, we know that $\left(\dfrac{-1}{l}\right) = -1$, and hence it follows that
\begin{align*}
\left(\dfrac{-h}{l}\right) = \left(\dfrac{h}{l}\right)\left(\dfrac{-1}{l}\right) = 1.
\end{align*}
Thus there is an integer $x_{\ast}$ such that $x_{\ast}^2 \equiv -h \pmod{l}$. Therefore we deduce that
\begin{align*}
x_{\ast}^{2^n} = (x_{\ast}^2)^{2^{n - 1}} \equiv (-h)^{2^{n-1}} \equiv (-1)^{2^{n - 1}}h^{2^{n - 1}} \equiv x \pmod{l},
\end{align*}
which proves our contention.

\end{proof}

\begin{lemma}
\label{infinitude-of-the-quadruples-p-n-kappa-chi-lemma}

Let $p$ be a prime such that $p \equiv 1 \pmod{8}$ and $p \equiv 2 \pmod{3}$. Let $n$ and $m$ be positive integers such that $n \ge 2$ and $1 \le m < n$. Then, for any positive integer $r$ dividing $m$, there are infinitely many couples $(\kappa, \chi)$ of integers such that the following are true:
\begin{itemize}

\item [(i)] $\kappa$ satisfies $(B3)-(B5)$ in Section \ref{a-subset-of-the-set-of-rational-points-on-X-p-section}.

\item [(ii)] $\kappa = 3^{2m - 1}\kappa_{\ast}^r$ for some odd prime $\kappa_{\ast}$ with $\kappa_{\ast} \ne 3$.

\item [(iii)] $\chi$ is an odd prime.

\item [(iv)] $(p, n, \kappa, \chi)$ satisfies $(D1)-(D7)$ in Theorem \ref{the-first-family-of-certain-generalized-Fermat-curves-violating-the-Hasse-principle-theorem}.

\end{itemize}

\end{lemma}

\begin{proof}

Let $r$ be a positive integer such that $r$ divides $m$, and define
\begin{align*}
s := \dfrac{m}{r}.
\end{align*}
Let $\F^{\ast}$ be the set of odd primes $l$ satisfying the following two conditions:
\begin{itemize}

\item [(i)] $\gcd(l, 3) = \gcd(l, p) = 1$, and

\item [(ii)] $l$ divides $n$.

\end{itemize}
Let $\G^{\ast}$ be the set of odd primes $l$ satisfying the following two conditions:
\begin{itemize}

\item [(i)] $\gcd(l, 3) = \gcd(l, p) = \gcd(l, n) = 1$, and

\item [(ii)] $l \le 4(6n - 1)^2(12n - 1)^2$.

\end{itemize}

$\star$ \textit{Step 1. Choosing $\kappa$.}

Note that $2,3$ and $p$ do not belong to $\F^{\ast}$ and $\G^{\ast}$, and that $\F^{\ast} \cap \G^{\ast} = \emptyset$. Hence, by the Chinese Remainder Theorem, there exists an integer $\kappa_{\ast, 0}$ such that the following are true:
\begin{itemize}

\item [(E1)] $\kappa_{\ast, 0} \equiv 1 \pmod{4}$,

\item [(E2)] $\kappa_{\ast, 0} \equiv \dfrac{1}{3^{2s}} \pmod{p^{2v_p(n) + 1}}$,

\item [(E3)] $\kappa_{\ast, 0} \equiv \dfrac{1}{3^{2s}} \pmod{l^{2v_l(n) + 1}}$ for each $l \in \F^{\ast}$, and

\item [(E4)] $\kappa_{\ast, 0} \equiv \dfrac{1}{3^{2s}} \pmod{l}$ for each $l \in \G^{\ast}$.

\end{itemize}
Let $Q(X) \in \bZ[X]$ be the linear polynomial defined by
\begin{align}
\label{the-linear-polynomial-Q(X)-defines-integers-kappa-ast-definition}
Q(X) := 2^2\cdot p^{2v_p(n) + 1}\left(\prod_{l \in \F^{\ast}} l^{2v_l(n) + 1} \right)\left(\prod_{l \in \G^{\ast}} l\right)X + \kappa_{\ast, 0}.
\end{align}
By $(E1), (E2), (E3)$ and $(E4)$ above, we know that
\begin{align*}
\gcd\left(\kappa_{\ast, 0}, 2^2\cdot p^{2v_p(n) + 1}\left(\prod_{l \in \F^{\ast}} l^{2v_l(n) + 1} \right)\left(\prod_{l \in \G^{\ast}} l\right)\right) = 1.
\end{align*}
Applying Dirichlet's theorem on arithmetic progressions, we deduce that there are infinitely many integers $X$ such that $Q(X) \ne 3$, $Q(X) \ne p$ and $Q(X)$ is an odd prime. Take such an integer $X$, and set
\begin{align*}
\kappa_{\ast} := Q(X).
\end{align*}
Define
\begin{align}
\label{the-definition-of-kappa-satisfying-conditions-Di-definition}
\kappa := 3^{2m - 1}\kappa_{\ast}^r.
\end{align}

$\star$ \textit{Step 2. Choosing $\chi_{\ast}$.}

By $(\ref{the-linear-polynomial-Q(X)-defines-integers-kappa-ast-definition})$ and $(E2)$, we see that
\begin{align*}
\kappa_{\ast} \equiv \kappa_{\ast, 0} \equiv \dfrac{1}{3^{2s}} \pmod{p^{2v_p(n) + 1}}.
\end{align*}
This implies that $\kappa_{\ast} \equiv \dfrac{1}{3^{2s}} \pmod{p}$, and hence it follows that $\left(\dfrac{\kappa_{\ast}}{p}\right) = 1$, where $\left(\dfrac{\cdot}{\cdot}\right)$ denotes the Jacobi symbol. Since $\kappa_{\ast}$ is an odd prime and $p \equiv 1 \pmod{8}$, it follows from the quadratic reciprocity law that
\begin{align*}
\left(\dfrac{p}{\kappa_{\ast}}\right) = 1.
\end{align*}
By $(\ref{the-linear-polynomial-Q(X)-defines-integers-kappa-ast-definition})$ and $(E1)$, we deduce that $\kappa_{\ast} \equiv \kappa_{\ast, 0} \equiv 1 \pmod{4}$. Hence $-1$ is a square in the finite field $\bF_{\kappa_{\ast}}$, and thus it follows that
\begin{align*}
\left(\dfrac{-p}{\kappa_{\ast}}\right) = \left(\dfrac{-1}{\kappa_{\ast}}\right)\left(\dfrac{p}{\kappa_{\ast}}\right)= 1.
\end{align*}
Therefore $-\dfrac{1}{p}$ is a square in $\bZ_{\kappa_{\ast}}^{\times}$, and hence there exist an element $\Gamma_{\kappa_{\ast}}$ in $\bZ_{\kappa_{\ast}}^{\times}$ and an integer $\Gamma_{\kappa_{\ast}, 0}$ such that
\begin{align}
\label{the-definition-of-Gamma-kappa-ast-definition}
\begin{cases}
\Gamma_{\kappa_{\ast}}^2 &= -\dfrac{1}{p} \\
\Gamma_{\kappa_{\ast}}   &\equiv \Gamma_{\kappa_{\ast}, 0}  \pmod{\kappa_{\ast}^{2v_{\kappa_{\ast}}(n) + 1}}.
\end{cases}
\end{align}

By assumption, we know that $-p \equiv -2 \equiv 1 \pmod{3}$, and hence we deduce that $-\dfrac{1}{p}$ is a square in $\bZ_3^{\times}$. Thus there exist an element $\Gamma_{3}$ in $\bZ_3^{\times}$ and an integer $\Gamma_{3, 0}$ such that
\begin{align}
\label{the-definition-of-Gamma-3-definition}
\begin{cases}
\Gamma_{3}^2 &= -\dfrac{1}{p} \\
\Gamma_{3}   &\equiv \Gamma_{3, 0}  \pmod{3^{2v_{3}(n) + 3}}.
\end{cases}
\end{align}

By assumption, we know that $p \equiv 1 \pmod{8}$, and hence we deduce that $\dfrac{1}{p}$ is a square in $\bZ_2^{\times}$. Thus there exist an element $\Gamma_{2}$ in $\bZ_2^{\times}$ and an integer $\Gamma_{2, 0}$ such that
\begin{align}
\label{the-definition-of-Gamma-2-definition}
\begin{cases}
\Gamma_{2}^2 &= \dfrac{1}{p} \\
\Gamma_{2}   &\equiv \Gamma_{2, 0}  \pmod{2^{2v_{2}(n) + 5}}.
\end{cases}
\end{align}

Let $\Delta_2, \Delta_3$ and $\Delta_{\kappa_{\ast}}$ be integers in $\{\pm 1\}$, which will be determined later. Let $\H$ be the set of odd primes $l$ satisfying the following two conditions:
\begin{itemize}

\item [(i)] $\gcd(l, 3) = \gcd(l, \kappa_{\ast}) = 1$, and

\item [(ii)] $l$ divides $n$.

\end{itemize}
Since $\kappa_{\ast} \ne 2, 3$ and $\H$ does not contain $2, 3$ and $\kappa_{\ast}$, it follows from the Chinese remainder theorem that there exists an integer $\chi_{\ast}$ such that the following are true:
\begin{itemize}

\item [(E5)] $\chi_{\ast} \equiv \Theta_2 \pmod{2^{2v_2(n) + 5}}$,

\item [(E6)] $\chi_{\ast} \equiv \Theta_3 \pmod{3^{2v_3(n) + 3}}$,

\item [(E7)] $\chi_{\ast} \equiv \Theta_{\kappa_{\ast}} \pmod{\kappa_{\ast}^{2v_{\kappa_{\ast}}(n) + 1}}$, and

\item [(E8)] $\chi_{\ast} \equiv \Theta_l \pmod{l^{v_l(n)}}$ for each prime $l \in \H$.

\end{itemize}
Here
\begin{align}
\label{the-definitions-of-Theta-2-3-kappa-ast-definition}
\begin{cases}
\Theta_2 &:= \Delta_2\Gamma_{2, 0}3^3\kappa^3 \\
\Theta_3 &:= \Delta_3\Gamma_{3, 0} \\
\Theta_{\kappa_{\ast}} &:= \Delta_{\kappa_{\ast}}\Gamma_{\kappa_{\ast}, 0},
\end{cases}
\end{align}
and for each prime $l \in \H$, take $\Theta_l$ to be an integer such that
\begin{align}
\label{the-definitions-of-Theta-l-for-each-prime-l-in-H-definition}
\gcd(\Theta_l, l) = \gcd(\Theta_l - 1, l) = 1.
\end{align}
Note that for each prime $l \in \H$, there exist infinitely many integers $\Theta_l$ satisfying $(\ref{the-definitions-of-Theta-l-for-each-prime-l-in-H-definition})$; for example, one can take $\Theta_l$ to be any integer such that $\Theta_l \equiv 2 \pmod{l}$ for each prime $l \in \H$.

$\star$ \textit{Step 3. Choosing $\Delta_2, \Delta_3$ and $\Delta_{\kappa_{\ast}}$.}

We first choose $\Delta_2$. Since $\Gamma_{2, 0}^2 \equiv \Gamma_{2}^2 = \dfrac{1}{p} \pmod{2^{2v_{2}(n) + 5}}$, it follows that $\Gamma_{2, 0}$ is odd, and hence either $\Gamma_{2, 0} \equiv 1 \pmod{4}$ or $\Gamma_{2, 0} \equiv -1 \pmod{4}$. By $(\ref{the-linear-polynomial-Q(X)-defines-integers-kappa-ast-definition})$, $(\ref{the-definition-of-kappa-satisfying-conditions-Di-definition})$ and $(E1)$, we see that
\begin{align*}
3^3\kappa^3 = 3^3\left(3^{2m - 1}\kappa_{\ast}^r\right)^3 \equiv (-1)^3(-1)^{3(2m - 1)}\kappa_{\ast, 0}^{3r} \equiv 1 \pmod{4}.
\end{align*}
Thus we deduce that either $\Gamma_{2, 0}3^3\kappa^3 \equiv 1 \pmod{4}$ or $\Gamma_{2, 0}3^3\kappa^3 \equiv -1 \pmod{4}$. We choose $\Delta_2 \in \{\pm 1\}$ such that
\begin{align}
\label{the-congruence-of-Theta-2-mod-4-equation}
\Theta_2 = \Delta_2\Gamma_{2, 0}3^3\kappa^3 \equiv -1 \equiv 3 \pmod{4}.
\end{align}

We now choose $\Delta_3$. Since $\Gamma_{3, 0}^2 \equiv \Gamma_{3}^2 = -\dfrac{1}{p} \pmod{3^{2v_{2}(n) + 3}}$ and $-p \equiv -2 \equiv 1 \pmod{3}$, it follows that $\Gamma_{3, 0}^2 \equiv -\dfrac{1}{p} \equiv 1 \pmod{3}$, and hence either $\Gamma_{3, 0} \equiv 1 \pmod{3}$ or $\Gamma_{3, 0} \equiv -1 \pmod{3}$. We choose $\Delta_3 \in \{\pm 1\}$ such that
\begin{align}
\label{the-congruence-of-Theta-3-mod-3-equation}
\Theta_3 = \Delta_3\Gamma_{3, 0} \equiv -1 \equiv 2 \pmod{3}.
\end{align}

We now define $\Delta_{\kappa_{\ast}}$. By $(\ref{the-definition-of-Gamma-kappa-ast-definition})$, we know that $\Gamma_{\kappa_{\ast}, 0} \not\equiv 0 \pmod{\kappa_{\ast}}$, and hence $\gcd\left(\Gamma_{\kappa_{\ast}, 0}, \kappa_{\ast}\right) = 1$. Using Lemma \ref{the-choice-of-sign-for-Delta-lemma} with $(r, l)$ replaced by $(\Gamma_{\kappa_{\ast}, 0}, \kappa_{\ast})$, we deduce that there is a choice of $\Delta_{\kappa_{\ast}} \in \{\pm 1\}$ such that
\begin{align}
\label{the-congruence-of-Theta-kappa-ast-mod-kappa-ast-equation}
\gcd\left(\Theta_{\kappa_{\ast}} - 1, \kappa_{\ast}\right) = \gcd\left(\Delta_{\kappa_{\ast}}\Gamma_{\kappa_{\ast}, 0} - 1, \kappa_{\ast}\right) = 1.
\end{align}

$\star$ \textit{Step 4. Choosing $\chi$.}

Define
\begin{align}
\label{the-definition-of-Upsilon-definition}
\Upsilon := 2^{2v_2(n) + 5}\cdot 3^{2v_3(n) + 3}\cdot \kappa_{\ast}^{2v_{\kappa_{\ast}}(n) + 1}\prod_{l \in \H}l^{v_l(n)},
\end{align}
where $\H$ was defined in \textit{Step 2}. Let $R(Y) \in \bZ[Y]$ be the linear polynomial defined by
\begin{align}
\label{the-definition-of-R(Y)-generating-chi-equation}
R(Y) := \Upsilon Y + \chi_{\ast},
\end{align}
where $\chi_{\ast}$ was chosen in \textit{Step 2}. By $(E5), (E6)$, $(E7)$ and $(E8)$, we see that
\begin{align*}
\gcd\left(\chi_{\ast}, \Upsilon\right) = 1.
\end{align*}
By Dirichlet's theorem on arithmetic progressions, we deduce that there are infinitely many integers $Y$ such that $R(Y)$ is an odd prime. Take such an integer $Y$, and define
\begin{align}
\label{the-definition-of-chi-satisfying-Di-equation}
\chi := R(Y).
\end{align}

Set
\begin{align}
\label{the-definition-of-n-ast-definition}
n_{\ast} := \dfrac{n}{2^{v_2(n)}} \in \bZ.
\end{align}
By $(\ref{the-definition-of-Upsilon-definition})$ and noting that $n_{\ast}$ can be written in the form
\begin{align*}
n_{\ast} = \dfrac{n}{2^{v_2(n)}} = 3^{v_3(n)}\kappa_{\ast}^{v_{\kappa_{\ast}}(n)}\prod_{l \in \H}l^{v_l(n)},
\end{align*}
we see that $\Upsilon$ can be written in the form
\begin{align*}
\Upsilon =  \left(2^{2v_2(n) + 5} \cdot 3^{v_3(n) + 3}\cdot \kappa_{\ast}^{v_{\kappa_{\ast}}(n) + 1}\right)n_{\ast}.
\end{align*}
Hence we deduce that
\begin{align*}
\chi = R(Y) = \Upsilon Y + \chi_{\ast} = \left(2^{2v_2(n) + 5} \cdot 3^{v_3(n) + 3}\cdot \kappa_{\ast}^{v_{\kappa_{\ast}}(n) + 1}\right)n_{\ast} Y + \chi_{\ast},
\end{align*}
and thus we have that
\begin{align}
\label{the-representation-of-the-prime-chi-in-terms-of-n-ast-and-chi-ast-equation}
\chi = n_{\ast}\Sigma_{n_{\ast}} + \chi_{\ast},
\end{align}
where
\begin{align*}
\Sigma_{n_{\ast}} := \left(2^{2v_2(n) + 5} \cdot 3^{v_3(n) + 3}\cdot \kappa_{\ast}^{v_{\kappa_{\ast}}(n) + 1}\right)Y.
\end{align*}

$\star$ \textit{Step 5. Verifying $(B3)-(B5)$ in Section \ref{a-subset-of-the-set-of-rational-points-on-X-p-section} and $(D1)-(D7)$ in Theorem \ref{the-first-family-of-certain-generalized-Fermat-curves-violating-the-Hasse-principle-theorem}.}

We first verify $(B3)-(B5)$ in Section \ref{a-subset-of-the-set-of-rational-points-on-X-p-section}. Recall that $\kappa_{\ast}$ is an odd prime such that $\kappa_{\ast} \ne 3$. Hence it follows from $(\ref{the-definition-of-kappa-satisfying-conditions-Di-definition})$ that
\begin{align*}
v_3(\kappa) = v_3\left(3^{2m - 1}\kappa_{\ast}^r\right) = 2m - 1 < 2n - 1,
\end{align*}
which proves that $(B3)$ is true. In \textit{Step 1}, we have chosen $\kappa_{\ast}$ to be an odd prime such that $\kappa_{\ast} \ne p$. Since $\kappa = 3^{2m - 1}\kappa_{\ast}^r \not\equiv 0 \pmod{p}$, we deduce that $(B4)$ holds.

We now prove that $(B5)$ holds. Indeed, let $l$ be any odd prime such that $\gcd(l, 3) = 1$ and $l$ divides $\kappa$. By $(\ref{the-definition-of-kappa-satisfying-conditions-Di-definition})$, we deduce that $l = \kappa_{\ast}$. By $(\ref{the-linear-polynomial-Q(X)-defines-integers-kappa-ast-definition})$ and $(E2)$, we know that
\begin{align*}
\kappa_{\ast} = Q(X) \equiv \kappa_{\ast, 0} \equiv \dfrac{1}{3^{2s}} \pmod{p^{2v_p(n) + 1}}.
\end{align*}
This implies that $\kappa_{\ast} \equiv \dfrac{1}{3^{2s}} \pmod{p}$, and thus $\left(\dfrac{\kappa_{\ast}}{p}\right) = 1$. Since $p \equiv 1 \pmod{8}$, it follows from the quadratic reciprocity law that $\left(\dfrac{p}{\kappa_{\ast}}\right) = 1$. In other words, $p$ is a square in $\bQ_{\kappa_{\ast}}^{\times}$, and hence $(B5)$ is true.

We now verify $(D1)-(D7)$ in Theorem \ref{the-first-family-of-certain-generalized-Fermat-curves-violating-the-Hasse-principle-theorem}. Let $\D, \E, \F$ and $\G$ the sets of odd primes that were defined in Theorem \ref{the-first-family-of-certain-generalized-Fermat-curves-violating-the-Hasse-principle-theorem}. It is not difficult to see that $(D1), (D6)$ and $(D7)$ are true. Indeed, we see that $\F \subseteq \F^{\ast}$ and $\G \subseteq \G^{\ast}$, where the sets $\F^{\ast}$ and $\G^{\ast}$ were defined in the paragraph preceding \textit{Step 1}.

For any odd prime $l \in \F^{\ast}$, it follows from $(E3)$, $(\ref{the-linear-polynomial-Q(X)-defines-integers-kappa-ast-definition})$ and $(\ref{the-definition-of-kappa-satisfying-conditions-Di-definition})$ that
\begin{align*}
\kappa = 3^{2m - 1}\kappa_{\ast}^r = 3^{2m - 1}Q(X)^r &\equiv 3^{2m - 1}\kappa_{\ast, 0}^r \equiv 3^{2m - 1}\dfrac{1}{3^{2rs}} \\
                                                      &\equiv 3^{2m - 1}\dfrac{1}{3^{2m}} \; \; (\text{since $rs = m$}) \\
                                                      &\equiv \dfrac{1}{3} \pmod{l^{2v_l(n) + 1}}.
\end{align*}
Since $\F \subseteq \F^{\ast}$, we deduce that $\kappa \equiv \dfrac{1}{3} \pmod{l^{2v_l(n) + 1}}$ for any $l \in \F^{\ast}$, and thus $(D6)$ is true. Using the same arguments, one can show that $(D1)$ and $(D7)$ are true.

We now prove that $(D2), (D3)$ and $(D4)$ hold. By $(\ref{the-definition-of-Gamma-2-definition})$, $(E5)$, $(\ref{the-definitions-of-Theta-2-3-kappa-ast-definition})$, $(\ref{the-definition-of-Upsilon-definition})$, $(\ref{the-definition-of-R(Y)-generating-chi-equation})$ and $(\ref{the-definition-of-chi-satisfying-Di-equation})$, we see that
\begin{align*}
p\chi^2 = pR(Y)^2 \equiv p\chi_{\ast}^2 \equiv p\Theta_2^2 &\equiv p\Delta_2^2\Gamma_{2, 0}^2 3^6\kappa^6 \\
                                                           &\equiv p\Gamma_{2}^2 3^6\kappa^6 \; \; (\text{since} \; \Delta_2^2 = 1) \\
                                                           &\equiv p\dfrac{1}{p}3^6\kappa^6 \; \; (\text{by} \; (\ref{the-definition-of-Gamma-2-definition})) \\
                                                           &\equiv 3^6\kappa^6 \pmod{2^{2v_2(n) + 5}},
\end{align*}
and hence $(D2)$ is true.

By $(\ref{the-definition-of-Gamma-3-definition})$, $(E6)$, $(\ref{the-definitions-of-Theta-2-3-kappa-ast-definition})$, $(\ref{the-definition-of-Upsilon-definition})$, $(\ref{the-definition-of-R(Y)-generating-chi-equation})$ and $(\ref{the-definition-of-chi-satisfying-Di-equation})$, we see that
\begin{align*}
p\chi^2 = pR(Y)^2 \equiv p\chi_{\ast}^2 \equiv p\Theta_3^2 &\equiv p\Delta_3^2\Gamma_{3, 0}^2 \\
                                                           &\equiv p\Gamma_{3}^2  \; \; (\text{since} \; \Delta_3^2 = 1) \\
                                                           &\equiv p\left(-\dfrac{1}{p}\right)  \; \; (\text{by} \; (\ref{the-definition-of-Gamma-3-definition})) \\
                                                           &\equiv -1 \pmod{3^{2v_3(n) + 3}},
\end{align*}
and hence $(D3)$ is true.

We now show that $(D4)$ holds. By $(\ref{the-definition-of-kappa-satisfying-conditions-Di-definition})$ and since $\kappa_{\ast}$ is an odd prime such that $\kappa_{\ast} \ne 3$ and $\kappa_{\ast} \ne p$, we see that $\D = \{\kappa_{\ast}\}$, where $\D$ is the set of odd primes that was defined in Theorem \ref{the-first-family-of-certain-generalized-Fermat-curves-violating-the-Hasse-principle-theorem}. By $(E1)$ and $(\ref{the-linear-polynomial-Q(X)-defines-integers-kappa-ast-definition})$, we know that
\begin{align*}
\kappa_{\ast} = Q(X) \equiv \kappa_{\ast, 0} \equiv 1 \pmod{4}.
\end{align*}
By $(\ref{the-definition-of-Gamma-kappa-ast-definition})$, $(E7)$, $(\ref{the-definitions-of-Theta-2-3-kappa-ast-definition})$, $(\ref{the-definition-of-Upsilon-definition})$, $(\ref{the-definition-of-R(Y)-generating-chi-equation})$ and $(\ref{the-definition-of-chi-satisfying-Di-equation})$, we see that
\begin{align*}
p\chi^2 = pR(Y)^2 \equiv p\chi_{\ast}^2 \equiv p\Theta_{\kappa_{\ast}}^2 &\equiv p\Delta_{\ast}^2\Gamma_{\kappa_{\ast}, 0}^2 \\
                                                           &\equiv p\Gamma_{\kappa_{\ast}}^2  \; \; (\text{since} \; \Delta_{\ast}^2 = 1) \\
                                                           &\equiv p\left(-\dfrac{1}{p}\right)  \; \; (\text{by} \; (\ref{the-definition-of-Gamma-kappa-ast-definition})) \\
                                                           &\equiv -1 \pmod{\kappa_{\ast}^{2v_{\kappa_{\ast}}(n) + 1}},
\end{align*}
and hence $(D4)$ is true.

Finally, we prove that $(D5)$ holds. Since $\chi$ is an odd prime, we deduce that either $\E = \emptyset$ or $\E = \{\chi\}$. By $(E5), (E6), (E7), (E8)$ and $(\ref{the-definition-of-chi-satisfying-Di-equation})$, we deduce that
\begin{align}
\label{Congruence-chi-ast-not-0-mod-2}
\chi \equiv \chi_{\ast} \not\equiv 0 \pmod{2},
\end{align}
\begin{align}
\label{Congruence-chi-ast-not-0-mod-3}
\chi \equiv \chi_{\ast} \not\equiv 0 \pmod{3},
\end{align}
\begin{align}
\label{Congruence-chi-ast-not-0-mod-kappa-ast}
\chi &\equiv \chi_{\ast} \not\equiv 0 \pmod{\kappa_{\ast}}.
\end{align}
By $(E8)$ and $(\ref{the-definitions-of-Theta-l-for-each-prime-l-in-H-definition})$, we know that
\begin{align}
\label{Congruence-chi-ast-not-0-mod-l-in-H}
\chi \equiv \chi_{\ast} \equiv \Theta_l \not\equiv 0 \pmod{l}
\end{align}
for every odd prime $l \in \H$. Since
\begin{align*}
n = 2^{v_2(n)}\cdot 3^{v_3(n)}\cdot \kappa_{\ast}^{v_{\kappa_{\ast}}(n)}\prod_{l \in \H}l^{v_l(n)},
\end{align*}
it follows from the last congruences that $\gcd(\chi, n) = 1$, and thus $v_{\chi}(n) = 0$. Hence we see that in order to prove that $(D5)$ holds, it suffices to prove that there exists an integer $\zeta_{\chi}$ with $\zeta_{\chi} \not\equiv 0 \pmod{\chi}$ such that either $\kappa \equiv \dfrac{\zeta_{\chi}^{2n}}{3} \pmod{\chi}$ or $\kappa \equiv -\dfrac{\zeta_{\chi}^{2n}}{3} \pmod{\chi}$.

By $(\ref{the-definition-of-kappa-satisfying-conditions-Di-definition})$ and since $\chi \not\equiv 0 \pmod{3}$ and $\chi \not\equiv 0 \pmod{\kappa_{\ast}}$, we see that $\gcd(\chi, 3\kappa) = 1$. By $(E5)$, $(\ref{the-congruence-of-Theta-2-mod-4-equation})$ and $(\ref{the-definition-of-chi-satisfying-Di-equation})$, we deduce that
\begin{align*}
\chi = R(Y) \equiv \chi_{\ast} \equiv \Theta_2 \equiv 3 \pmod{4}.
\end{align*}
We contend that that there is an integer $\rho \in \{\pm 1\}$ such that $\rho 3\kappa$ is a square in $\bF_{\chi}^{\times}$. Indeed, if $3\kappa$ is a square in $\bF_{\chi}^{\times}$, then we let $\rho = 1$. Assume that $3\kappa$ is not a square in $\bF_{\chi}^{\times}$, i.e., $\left(\dfrac{3\kappa}{\chi}\right) = -1$. Since $\chi \equiv 3 \pmod{4}$, we know that $-1$ is a quadratic non-residue in $\bF_{\chi}^{\times}$. Hence it follows that
\begin{align*}
\left(\dfrac{-3\kappa}{\chi}\right) = \left(\dfrac{3\kappa}{\chi}\right)\left(\dfrac{-1}{\chi}\right) = 1,
\end{align*}
which proves that $\rho3\kappa$ is a square in $\bF_{\chi}^{\times}$, where $\rho = -1$. For the rest of the proof, fix an integer $\rho \in \{\pm 1\}$ such that $\rho 3\kappa$ is a square in $\bF_{\chi}^{\times}$.

Since $\chi \equiv 3 \pmod{4}$, using Lemma \ref{every-square-is-a-2-power-in-the-finite-field-F-l-with-l-congruent-to-3-mod-4-lemma} with $(l, x)$ replaced by $(\chi, \rho3\kappa)$, we deduce that $\rho3\kappa$ is a $2^{v_2(n) + 1}$-th power in $\bF_{\chi}^{\times}$. In other words, there is an integer $\tau$ such that $\tau \not\equiv 0 \pmod{\chi}$ and
\begin{align}
\label{tau-raised-to-2-power-equals-rho-3-kappa-modulo-chi-equation}
\tau^{2^{v_2(n) + 1}} \equiv \rho3\kappa \pmod{\chi}.
\end{align}

We prove that $\gcd(\chi_{\ast} - 1, n_{\ast}) = 1$, where $n_{\ast}$ is defined by $(\ref{the-definition-of-n-ast-definition})$. Indeed, by $(E6)$ and $(\ref{the-congruence-of-Theta-3-mod-3-equation})$, we know that
\begin{align}
\label{chi-ast-minus-1-congruent-to-1-modulo-3-equation}
\chi_{\ast} - 1 \equiv \Theta_3 - 1 \equiv 1 \pmod{3}.
\end{align}
By $(E7)$ and $(\ref{the-congruence-of-Theta-kappa-ast-mod-kappa-ast-equation})$, we see that
\begin{align}
\label{chi-ast-minus-1-congruent-to-non-zero-modulo-kappa-ast-equation}
\chi_{\ast} - 1 \equiv \Theta_{\kappa_{\ast}} - 1 \not\equiv 0 \pmod{\kappa_{\ast}}.
\end{align}
Recall from $(E8)$ and $(\ref{the-definitions-of-Theta-l-for-each-prime-l-in-H-definition})$ that for each prime $l \in \H$, we have that
\begin{align}
\label{chi-ast-minus-1-congruent-to-non-zero-modulo-l-with-l-in-H-equation}
\chi_{\ast} - 1 \equiv \Theta_l - 1 \not\equiv 0 \pmod{l}.
\end{align}
Since
\begin{align*}
n_{\ast} = \dfrac{n}{2^{v_2(n)}} = 3^{v_3(n)}\cdot \kappa_{\ast}^{v_{\kappa_{\ast}}(n)}\prod_{l \in \H}l^{v_l(n)},
\end{align*}
it follows from $(\ref{chi-ast-minus-1-congruent-to-1-modulo-3-equation})$, $(\ref{chi-ast-minus-1-congruent-to-non-zero-modulo-kappa-ast-equation})$ and $(\ref{chi-ast-minus-1-congruent-to-non-zero-modulo-l-with-l-in-H-equation})$ that
\begin{align*}
\gcd(\chi_{\ast} - 1, n_{\ast}) = 1.
\end{align*}

By $(\ref{Congruence-chi-ast-not-0-mod-2})$, $(\ref{Congruence-chi-ast-not-0-mod-3})$, $(\ref{Congruence-chi-ast-not-0-mod-kappa-ast})$ and $(\ref{Congruence-chi-ast-not-0-mod-l-in-H})$ and since
\begin{align*}
n_{\ast} = \dfrac{n}{2^{v_2(n)}} = 3^{v_3(n)}\cdot \kappa_{\ast}^{v_{\kappa_{\ast}}(n)}\prod_{l \in \H}l^{v_l(n)},
\end{align*}
we see that $\gcd(\chi_{\ast}, n_{\ast}) = 1$, and hence it follows that
\begin{align*}
\gcd(\chi_{\ast}, n_{\ast}) = 1 = \gcd(\chi_{\ast} - 1, n_{\ast}) = 1.
\end{align*}
We recall from $(\ref{the-representation-of-the-prime-chi-in-terms-of-n-ast-and-chi-ast-equation})$ that
\begin{align*}
\chi = n_{\ast}\Sigma_{n_{\ast}} + \chi_{\ast}.
\end{align*}
Hence, using Lemma \ref{the-existence-of-primes-such-that-any-integer-is-s-th-power-lemma} with $(r, s, q, S)$ replaced by $(\chi_{\ast}, n_{\ast}, \chi, \Sigma_{n_{\ast}})$, we deduce that every integer is an $n_{\ast}$-th power in $\bF_{\chi}$. Since $\tau \not\equiv 0 \pmod{\chi}$, this implies that $\tau$ is an $n_{\ast}$-th power in $\bF_{\chi}^{\times}$. In other words, there exists an integer $\zeta_{\chi}$ such that $\zeta_{\chi} \not\equiv 0 \pmod{\chi}$ and
\begin{align}
\label{tau-modulo-chi-equals-zeta-chi-to-the-n-ast-power-equation}
\tau \equiv \zeta_{\chi}^{n_{\ast}} \pmod{\chi}.
\end{align}
By $(\ref{tau-raised-to-2-power-equals-rho-3-kappa-modulo-chi-equation})$, $(\ref{tau-modulo-chi-equals-zeta-chi-to-the-n-ast-power-equation})$ and since $n = 2^{v_2(n)}n_{\ast}$, we deduce that
\begin{align*}
\rho3\kappa &\equiv \tau^{2^{v_2(n) + 1}} \equiv \left(\zeta_{\chi}^{n_{\ast}}\right)^{2^{v_2(n) + 1}} \\
            &\equiv \left(\zeta_{\chi}\right)^{n_{\ast} 2^{v_2(n) + 1}} \\
            &\equiv \zeta_{\chi}^{2n} \pmod{\chi}.
\end{align*}
Since either $\rho = 1$ or $\rho = -1$, we see that $\rho^2 = 1$, and hence it follows that
\begin{align*}
\kappa \equiv \dfrac{\zeta_{\chi}^{2n}}{3\rho} \equiv \rho\dfrac{\zeta_{\chi}^{2n}}{3\rho^2} \equiv \rho\dfrac{\zeta_{\chi}^{2n}}{3} \pmod{\chi},
\end{align*}
which proves that $(D5)$ holds. Hence our contention follows.

\end{proof}

\begin{corollary}
\label{the-existence-of-the-first-family-of-the-generalized-Fermat-curves-violating-the-Hasse-principle-corollary}

Let $p$ be a prime such that $p \equiv 1 \pmod{8}$ and $p \equiv 2 \pmod{3}$. Let $n$ and $m$ be positive integers such that $n \ge 2$ and $1 \le m < n$. For any positive integer $r$ dividing $m$, there are infinitely many integers $\kappa$ and infinitely many integers $\chi$ satisfying the following five conditions:
\begin{itemize}

\item [(i)] $\kappa$ satisfies $(B3)-(B5)$ in Section \ref{a-subset-of-the-set-of-rational-points-on-X-p-section}.

\item [(ii)] $\kappa = 3^{2m - 1}\kappa_{\ast}^r$ for some odd prime $\kappa_{\ast}$ with $\kappa_{\ast} \ne 3$.

\item [(iii)] $\chi$ is an odd prime.

\item [(iv)] $(p, n, \kappa, \chi)$ satisfies $(D1)-(D7)$ in Theorem \ref{the-first-family-of-certain-generalized-Fermat-curves-violating-the-Hasse-principle-theorem}.

\item [(v)] let $\cF_{(\kappa, \chi)}^{(p, n, m, r)}$ be the generalized Fermat curve of signature $(12n, 12n, 12n)$ defined by
\begin{align}
\label{the-generalized_Fermat-curves-F-kappa-chi-p-n-m-r-equation}
\cF_{(\kappa, \chi)}^{(p, n, m, r)} : 3^6\kappa^6x^{12n} - y^{12n} - p\chi^2z^{12n} = 0.
\end{align}
Then $\cF_{(\kappa, \chi)}^{(p, n, m, r)}$ is a counterexample to the Hasse principle explained by the Brauer-Manin obstruction.

\end{itemize}

\end{corollary}

\begin{proof}

By Lemma \ref{infinitude-of-the-quadruples-p-n-kappa-chi-lemma}, we see that for any positive integer $r$ dividing $m$, there are infinitely many integers $\kappa$ and infinitely many integers $\chi)$ satisfying $(i), (ii), (iii)$ and $(iv)$ in Corollary \ref{the-existence-of-the-first-family-of-the-generalized-Fermat-curves-violating-the-Hasse-principle-corollary}. For any couple $(\kappa, \chi)$ satisfying $(i), (ii), (iii)$ and $(iv)$ in Corollary \ref{the-existence-of-the-first-family-of-the-generalized-Fermat-curves-violating-the-Hasse-principle-corollary}, applying Theorem \ref{the-first-family-of-certain-generalized-Fermat-curves-violating-the-Hasse-principle-theorem} for the curve $\cF_{(\kappa, \chi)}^{(p, n, m, r)}$, we deduce that $\cF_{(\kappa, \chi)}^{(p, n, m, r)}$ is a counterexample to the Hasse principle explained by the Brauer-Manin obstruction. Hence $(v)$ in Corollary \ref{the-existence-of-the-first-family-of-the-generalized-Fermat-curves-violating-the-Hasse-principle-corollary} holds, and thus our contention follows.

\end{proof}

\section{The descending chain condition on sequences of generalized Fermat curves}
\label{the-DCC-on-sequences-of-generalized-Fermat-curves-section}

In this section, we will study the descending chain condition on sequences of generalized Fermat curves. We will prove that there exist infinitely many sequences of generalized Fermat curves satisfying the descending chain condition in the sense of Definition \ref{the-descending-chain-condition-on-a-sequence-of-curves-definition}. We begin by proving the main lemma in this section.

\begin{lemma}
\label{for-2m-less-than-n-there-are-two-generalized-Fermat-curves-one-has-rational-points-but-the-other-one-has-no-rational-point-lemma}

Let $p$ be a prime such that $p \equiv 1 \pmod{8}$ and $p \equiv 2 \pmod{3}$. Let $n$ and $m$ be positive integers such that $2m < n$. Then there exist an infinite set $\I_{n, m}$ of integers and an infinite set $\J_{n, m}$ of integers such that for each integer $\kappa \in \I_{n, m}$ and each integer $\chi \in \J_{n, m}$, $\cF_{(\kappa, \chi)}^{(m)}$ contains at least four rational points whereas $\cF_{(\kappa, \chi)}^{(n)}$ is a counterexample to the Hasse principle explained by the Brauer-Manin obstruction, where for each integer $\kappa \in \I_{n, m}$ and each integer $\chi \in \J_{n, m}$, $\cF_{(\kappa, \chi)}^{(m)}$ and $\cF_{(\kappa, \chi)}^{(n)}$ are the generalized Fermat curves of signatures $(12m, 12m, 12m)$ and $(12n, 12n, 12n)$, respectively and defined by
\begin{align}
\label{the-generalized-Fermat-curve-F-kappa-chi-m-equation}
\cF_{(\kappa, \chi)}^{(m)} : 3^6\kappa^6x^{12m} - y^{12m} - p\chi^2z^{12m} = 0
\end{align}
and
\begin{align}
\label{the-generalized-Fermat-curve-F-kappa-chi-n-equation}
\cF_{(\kappa, \chi)}^{(n)} : 3^6\kappa^6x^{12n} - y^{12n} - p\chi^2z^{12n} = 0.
\end{align}

\end{lemma}

\begin{proof}

Let $\I_{n, m}$ be the set of nonzero integers $\kappa$ and $\J_{n, m}$ be the set of integers $\chi$ such that the following are true:
\begin{itemize}

\item [(i)] $\kappa$ satisfies $(B3)-(B5)$ in Section \ref{a-subset-of-the-set-of-rational-points-on-X-p-section},

\item [(ii)] $\kappa = 3^{4m - 1}\kappa_{\ast}^{2m}$ for some odd prime $\kappa_{\ast}$ with $\kappa_{\ast} \ne 3$,

\item [(iii)] $\chi$ is an odd prime, and

\item [(iv)] $(p, n, \kappa, \chi)$ satisfies $(D1)-(D7)$ in Theorem \ref{the-first-family-of-certain-generalized-Fermat-curves-violating-the-Hasse-principle-theorem}.

\end{itemize}
Using Lemma \ref{infinitude-of-the-quadruples-p-n-kappa-chi-lemma} with $(m, r)$ replaced by $(2m, 2m)$, we deduce that $\I_{n, m}$ and $\J_{n, m}$ are of infinite cardinality. Let $\kappa$ be a nonzero integer in $\I_{n, m}$, and let $\chi$ be an integer in $\J_{n, m}$. Let $\cF_{(\kappa, \chi)}^{(m)}$ and $\cF_{(\kappa, \chi)}^{(n)}$ be the generalized Fermat curves defined by $(\ref{the-generalized-Fermat-curve-F-kappa-chi-m-equation})$ and $(\ref{the-generalized-Fermat-curve-F-kappa-chi-n-equation})$, respectively. Since $\kappa \in \I_{n, m}$, there exists an odd prime $\kappa_{\ast}$ such that $\kappa_{\ast} \ne 3$ and $\kappa = 3^{4m - 1}\kappa_{\ast}^{2m}$. The defining equation of $\cF_{(\kappa, \chi)}^{(m)}$ can be written in the form
\begin{align*}
\cF_{(\kappa, \chi)}^{(m)} : \left(3^{2}\kappa_{\ast}x\right)^{12m} - y^{12m} - p\chi^2z^{12m} = 0.
\end{align*}
Hence we see that the points $(x, y, z) = \left(\pm \dfrac{1}{ 3^{2}\kappa_{\ast}}, \pm 1, 0\right)$ belong to $\cF_{(\kappa, \chi)}^{(m)}(\bQ)$, and thus $\cF_{(\kappa, \chi)}^{(m)}$ has at least four $\bQ$-rational points.

Since $\kappa \in \I_{n, m}$ and $\chi \in \J_{n, m}$, we know that $\kappa$ satisfies $(B3)-(B5)$ in Section \ref{a-subset-of-the-set-of-rational-points-on-X-p-section} and the quadruple $(p, n, \kappa, \chi)$ satisfies $(D1)-(D7)$ in Theorem \ref{the-first-family-of-certain-generalized-Fermat-curves-violating-the-Hasse-principle-theorem}. By Theorem \ref{the-first-family-of-certain-generalized-Fermat-curves-violating-the-Hasse-principle-theorem}, we deduce that $\cF_{(\kappa, \chi)}^{(n)}$ is a counterexample to the Hasse principle explained by the Brauer-Manin obstruction. Thus our contention follows.

\end{proof}

\begin{remark}
\label{the-optimal-bound-for-the-signature-of-the-first-family-of-generalized-Fermat-curves-remark}

Let $(p, n, \kappa, \chi)$ be the quadruple satisfying the conditions in Theorem \ref{the-first-family-of-certain-generalized-Fermat-curves-violating-the-Hasse-principle-theorem}, and assume further that $n \ge 3$. Let $\cF$ be the generalized Fermat curve defined as in Theorem \ref{the-first-family-of-certain-generalized-Fermat-curves-violating-the-Hasse-principle-theorem}. By Lemma \ref{for-2m-less-than-n-there-are-two-generalized-Fermat-curves-one-has-rational-points-but-the-other-one-has-no-rational-point-lemma} and using the same arguments as in Remark \ref{Remark-the-degree-12n-is-optimal-for-generalized-Mordell-curves-D}, we deduce that the signature $(12n, 12n, 12n)$ of $\cF$ is \textit{optimal} in the sense that in Theorem \ref{the-first-family-of-certain-generalized-Fermat-curves-violating-the-Hasse-principle-theorem}, one can not replace $n$ by a positive divisor $m$ of $n$ with $m \ne n$.

\end{remark}

We now prove the main result in this section, which says that there are infinitely many sequences of generalized Fermat curves that satisfy the DCC of arbitrary length.

\begin{corollary}

Let $h$ be a positive integer. Then there exist infinitely many sequences of generalized Fermat curves that satisfy the DCC of length $h$ in the sense of Definition \ref{the-descending-chain-condition-on-a-sequence-of-curves-definition}.

\end{corollary}

\begin{proof}

Let $p$ be a prime such that $p \equiv 1 \pmod{8}$ and $p \equiv 2 \pmod{3}$. Let $n_0$ and $n_1$ be integers such that $n_0 \ge 3$ and $n_1 \ge 1$. Set
\begin{align*}
n := n_0^hn_1, \\
m := n_0^{h - 1}n_1.
\end{align*}
Since $n_0 \ge 3$, we see that
\begin{align*}
2m = 2n_0^{h - 1}n_1 < n_0n_0^{h - 1}n_1 = n.
\end{align*}
Applying Lemma \ref{for-2m-less-than-n-there-are-two-generalized-Fermat-curves-one-has-rational-points-but-the-other-one-has-no-rational-point-lemma} for the triple $(p, n, m)$, we deduce that there exist an infinite set $\I_{n, m}$ of integers and an infinite set $\J_{n, m}$ of integers such that for each integer $\kappa \in \I_{n, m}$ and each integer $\chi \in \J_{n, m}$, $\cF^{(\kappa, \chi)}_{h - 1}$ contains at least four rational points whereas $\cF^{(\kappa, \chi)}_{h}$ is a counterexample to the Hasse principle explained by the Brauer-Manin obstruction, where for each integer $\kappa \in \I_{n, m}$ and each integer $\chi \in \J_{n, m}$, $\cF^{(\kappa, \chi)}_{h - 1}$ and $\cF^{(\kappa, \chi)}_{h}$ are the generalized Fermat curves defined by
\begin{align*}
\cF^{(\kappa, \chi)}_{h - 1} : 3^6\kappa^6x^{12n_0^{h - 1}n_1} - y^{12n_0^{h - 1}n_1} - p\chi^2z^{12n_0^{h - 1}n_1} = 0
\end{align*}
and
\begin{align*}
\cF^{(\kappa, \chi)}_{h} : 3^6\kappa^6x^{12n_0^hn_1} - y^{12n_0^hn_1} - p\chi^2z^{12n_0^hn_1} = 0,
\end{align*}
respectively.

Take any nonzero integer $\kappa \in \I_{n, m}$, and let $\chi$ be any nonzero integer in $\J_{n, m}$. For each integer $i \ge 1$, let $\cF^{(\kappa, \chi)}_{i}$ be the generalized Fermat curve defined by
\begin{align*}
\cF^{(\kappa, \chi)}_{i} : 3^6\kappa^6x^{12n_0^in_1} - y^{12n_0^in_1} - p\chi^2z^{12n_0^in_1} = 0,
\end{align*}
and for each integer $i \ge 1$, let $\psi_{i}^{(\kappa, \chi)} : \cF^{(\kappa, \chi)}_{i + 1} \rightarrow \cF^{(\kappa, \chi)}_{i}$ be the $\bQ$-morphism of curves defined by
\begin{align*}
\psi_{i}^{(\kappa, \chi)} : \cF^{(\kappa, \chi)}_{i + 1} &\rightarrow \cF^{(\kappa, \chi)}_{i} \\
(x, y, z) &\mapsto (x^{n_0}, y^{n_0}, z^{n_0}).
\end{align*}
We contend that the sequence $\left(\cF^{(\kappa, \chi)}_{i}, \psi_{i}^{(\kappa, \chi)} \right)_{i \ge 1}$ satisfies the DCC of length $h$ in the sense of Definition \ref{the-descending-chain-condition-on-a-sequence-of-curves-definition}. Indeed, since $\kappa \in \I_{n, m}$ and $\chi \in \J_{n, m}$, we know that $\cF^{(\kappa, \chi)}_{h - 1}$ contains at least four rational points whereas $\cF^{(\kappa, \chi)}_{h}$ is a counterexample to the Hasse principle explained by the Brauer-Manin obstruction. Thus (DCC2) in Definition \ref{the-descending-chain-condition-on-a-sequence-of-curves-definition} holds.

Let $(x_0, y_0, z_0)$ be any rational point in $\cF^{(\kappa, \chi)}_{h - 1}(\bQ)$. We see that for each integer $1 \le i \le h - 2$, the point $(x, y, z) = (x_0^{n_0^{h - 1 - i}}, y_0^{n_0^{h - 1- i}}, z_0^{n_0^{h - 1 - i}})$ belongs to $\cF_{i}^{(\kappa, \chi)}(\bQ)$. Hence (DCC1) in Definition \ref{the-descending-chain-condition-on-a-sequence-of-curves-definition} holds. For each integer $i \ge 1$, let $g_i^{(\kappa, \chi)}$ denote the genus of the curve $\cF_{i}^{(\kappa, \chi)}$. We see that
\begin{align*}
g_i^{(\kappa, \chi)} = (12n_0^in_1 - 1)(6n_0^in_1 - 1)
\end{align*}
for each $i \ge 1$. Hence it follows that $g_i^{(\kappa, \chi)} < g_j^{(\kappa, \chi)}$ for any positive integers $i, j$ with $1 \le i < j$, and thus (DCC3) holds. Therefore $\left(\cF^{(\kappa, \chi)}_{i}, \psi_{i}^{(\kappa, \chi)} \right)_{i \ge 1}$ satisfies the DCC of length $h$ in the sense of Definition \ref{the-descending-chain-condition-on-a-sequence-of-curves-definition}. Since $\I_{n, m}$ and $\J_{n, m}$ are of infinite cardinality, our contention follows.

\end{proof}

\begin{remark}

Note that there exist infinitely many sequences of generalized Fermat curves that do not satisfy the DCC of any length. Indeed, take any triple $(A, B, C)$ of nonzero integers such that $A + B + C = 0$. Let $n$ be an integer such that $n \ge 2$. For each integer $i \ge 1$, let $\cF_{i}$ be the generalized Fermat curve defined by
\begin{align*}
\cF_i : Ax^{n^i} + By^{n^i} + Cz^{n^i} = 0,
\end{align*}
and for each integer $i \ge 1$, let $\psi_i : \cF_{i + 1} \rightarrow \cF_{i}$ be the morphism of curves defined by
\begin{align*}
\psi_i : \cF_{i + 1} &\rightarrow \cF_{i} \\
(x, y, z) &\mapsto (x^{n}, y^{n}, z^{n}).
\end{align*}
Since $A + B + C = 0$, we deduce that for each $i \ge 1$, the curve $\cF_i$ contains the rational point $(x, y, z) = (1, 1, 1)$, and hence $\cF_i(\bQ) \ne 0$ for each integer $i \ge 1$. Thus the sequence $\left(\cF_{i}, \psi_{i} \right)_{i \ge 1}$ does not satisfy the DCC.

\end{remark}

\section{Epilogue}
\label{the-second-subset-of-the-set-of-rational-points-on-X-p-section}

In Sections \ref{certain-generalized-Mordell-curves-violate-the-Hasse-principle-section} and \ref{certain-generalized-Fermat-curves-violate-the-Hasse-principle-section}, for each positive integer $n \ge 2$, we have constructed certain families of generalized Mordell curves of degree $12n$ and certain families of generalized Fermat curves of signature $(12n, 12n, 12n)$ arising from rational points of the subset of $\cX_p(\bQ)$ defined by $(\ref{the-first-parameterization-of-septuples-(A-B-C-D-E-F-G)-satisfying-Hypothesis-FM-equations})$ in Section \ref{a-subset-of-the-set-of-rational-points-on-X-p-section} that are counterexamples to the Hasse principle explained by the Brauer-Manin obstruction. Because of condition $(B3)$ in Section \ref{a-subset-of-the-set-of-rational-points-on-X-p-section}, it is impossible to construct generalized Mordell curves of degree $12$ and certain families of generalized Fermat curves of signature $(12, 12, 12)$ that are counterexamples to the Hasse principle; in other words, we rule out the case $n = 1$ in Theorems \ref{the-first-infinite-family-of-generalized-Mordell-curves-violates-the-Hasse-principle-theorem} and \ref{the-first-family-of-certain-generalized-Fermat-curves-violating-the-Hasse-principle-theorem}. In this section, we will introduce another subset of the set of rational points on $\cX_p$, which allows us to include the case $n = 1$ in Theorems \ref{the-first-infinite-family-of-generalized-Mordell-curves-violates-the-Hasse-principle-theorem} and \ref{the-first-family-of-certain-generalized-Fermat-curves-violating-the-Hasse-principle-theorem}. Using this subset of the set of rational points on $\cX_p$ and assuming Schinzel's Hypothesis H, we will prove that for each prime $p$ such that $p \equiv 1 \pmod{8}$ and $p \equiv 2 \pmod{3}$ and each integer $n \ge 1$, there should be infinitely many septuples $(A, B, C, D, E, F, G)$ of integers that satisfy Hypothesis FM with respect to the couple $(p, n)$ in the sense of Definition \ref{septuple-satisfies-hypothesis-FM-definition}. Using these septuples and repeating the same arguments as in Sections \ref{certain-generalized-Mordell-curves-violate-the-Hasse-principle-section} and \ref{certain-generalized-Fermat-curves-violate-the-Hasse-principle-section}, one can show that for each $n \ge 1$, there should exist families of generalized Mordell curves of degree $12n$ and families of generalized Fermat curves of signature $(12n, 12n, 12n)$ that are counterexamples to the Hasse principle explained by the Brauer-Manin obstruction.

In this section, we restrict ourselves to describing a new subset of $\cX_p(\bQ)$ for each prime $p$ that allows us to produce new families of generalized Mordell curves and new families of generalized Fermat curves that are counterexamples to the Hasse principle, but not proceeding to describe these families of curves. The interested reader can use similar arguments as in Sections \ref{certain-generalized-Mordell-curves-violate-the-Hasse-principle-section} and \ref{certain-generalized-Fermat-curves-violate-the-Hasse-principle-section} to construct new families of generalized Mordell curves and new families of generalized Fermat curves arising from new rational points in $\cX_p(\bQ)$ for each prime $p$ that are counterexamples to the Hasse principle explained by the Brauer-Manin obstruction. We begin by introducing a new subset of $\cX_p(\bQ)$ for each prime $p$ such that $p \equiv 1 \pmod{8}$ and $p \equiv 2 \pmod{3}$.

Let $p$ be an odd prime such that $p \ne 3$. Let $\lambda$ and $\gamma$ be nonzero \textit{odd integers} such that
\begin{align*}
\gcd(\lambda, 3\gamma) = \gcd(p, 3\gamma) = \gcd(p, \lambda) = 1.
\end{align*}
Since $\gcd(p\lambda^2, 9\gamma^2) = 1$, there exist nonzero integers $\epsilon_0$ and $\delta_0$ such that
\begin{align*}
p\lambda^2\epsilon_0 + 9\gamma^2\delta_0 = 1.
\end{align*}
For integers $\mu, t_0, F_0 \in \bZ$, we define
\begin{align}
\label{the-second-parametrization-of-septuples-A-B-C-D-E-F-G-satisfying-Hypothesis-FM-equations}
\begin{cases}
A &:= \dfrac{p\lambda^2 - 9\gamma^2}{2} \\
B &:= 2pF_0^2\left(\delta_0 - \epsilon_0 - \mu (p\lambda^2 + 9\gamma^2)\right) + (p\lambda^2 + 9\gamma^2)t_0F_0 \\
C &:= 2pF_0^2\left(\delta_0 + \epsilon_0 - \mu (p\lambda^2 - 9\gamma^2)\right) + (p\lambda^2 - 9\gamma^2)t_0F_0 \\
D &= \dfrac{p\lambda^2 + 9\gamma^2}{2} \\
E &:= F_0(2pF_0(\epsilon_0 + 9\mu\gamma^2) - 9\gamma^2t_0)(2F_0(\delta_0 - p\mu\lambda^2) + \lambda^2t_0) \\
F &:= 2F_0 \\
G &= 3\lambda \gamma.
\end{cases}
\end{align}
Note that $B$ and $C$ can be written in the following form
\begin{align}
\label{the-representation-of-B-and-C-in-terms-of-D-and-A-in-the-second-parametrization-of-the-septuples-equations}
\begin{cases}
B &:= 2pF_0^2\left(\delta_0 - \epsilon_0 - 2\mu D\right) + 2Dt_0F_0 \\
C &:= 2pF_0^2\left(\delta_0 + \epsilon_0 - 2\mu A\right) + 2At_0F_0.
\end{cases}
\end{align}

It is not difficult to verify that the point $\cP := (a : b : c : d : e : f : g) = (A : B : C : D : E : F : G)$ belongs to $\cX_p(\bQ)$; hence, the septuple $(A, B, C, D, E, F, G)$ satisfies $(A1)$ in Definition \ref{septuple-satisfies-hypothesis-FM-definition}. We see that $(\ref{the-second-parametrization-of-septuples-A-B-C-D-E-F-G-satisfying-Hypothesis-FM-equations})$ defines the parametrization of a subset of $\cX_p(\bQ)$ by parameters $\lambda, \gamma, \mu, t_0$ and $F_0$. Using this parametrization, we will show that there are infinitely many septuples $(A, B, C, D, E, F, G)$ satisfying Hypothesis FM with respect to the couples $(p, n)$, where $n$ is a sufficiently large integer. The following lemma is the main result in this section.

\begin{lemma}
\label{the-infinitude-of-septuples-A-B-C-D-E-F-G-satisfying-A1-A3-A4-A5-A6-A7-lemma}

Let $p$ be a prime such that $p \equiv 1 \pmod{8}$ and $p \equiv 2 \pmod{3}$. Then there exist infinitely many septuples $(A, B, C, D, E, F, G) \in \bZ^7$ such that they satisfy $(\ref{the-second-parametrization-of-septuples-A-B-C-D-E-F-G-satisfying-Hypothesis-FM-equations})$ and $(A1), (A3), (A4), (A5), (A7)$ in Definition \ref{septuple-satisfies-hypothesis-FM-definition}, and such that for any integer $n \ge 1$, they satisfy $(A6)$ in Definition \ref{septuple-satisfies-hypothesis-FM-definition} with respect to the couple $(p, n)$.

\end{lemma}

\begin{proof}

Let $\lambda$ and $\gamma$ be odd integers such that
\begin{align}
\label{lambda-gamma-divisibility-properties-equations}
\gcd(\lambda, 3\gamma) = \gcd(p, 3\gamma) = \gcd(p, \lambda) = 1.
\end{align}
Since $\gcd(p\lambda^2, 9\gamma^2) = 1$, there exist non-zero integers $\epsilon^{\ast}$ and $\delta^{\ast}$ such that
\begin{align*}
p\lambda^2\epsilon^{\ast} + 9\gamma^2\delta^{\ast} = 1.
\end{align*}
We see that $\epsilon_0 = \epsilon^{\ast} + 9\gamma^2s^{\ast}$ and $\delta_0 = \delta^{\ast} - p\lambda^2s^{\ast}$ satisfy the following equation
\begin{align}
\label{epsilon0-delta0-equations}
p\lambda^2\epsilon_0 + 9\gamma^2\delta_0 = 1,
\end{align}
where $s^{\ast}$ is an arbitrary integer. For our purpose, we choose $s^{\ast}$ such that $\delta^{\ast} + s^{\ast} \equiv 2 \pmod{3}$. Since $\lambda \not\equiv 0 \pmod{3}$ and $p \equiv 2 \pmod{3}$, we deduce that
\begin{align}
\label{delta0-modulo-3-equal-2-equation}
\delta_0 = \delta^{\ast} - p\lambda^2s^{\ast}  \equiv \delta^{\ast} - 2s^{\ast} \equiv \delta^{\ast} + s^{\ast} \equiv 2 \pmod{3}.
\end{align}

Let $(A, B, C, D, E, F, G)$ be the septuple of integers defined by $(\ref{the-second-parametrization-of-septuples-A-B-C-D-E-F-G-satisfying-Hypothesis-FM-equations})$, where $\mu, t_0$ and $F_0$ will be determined later. It is not difficult to prove that $(A, B, C, D, E, F, G)$ belongs to $X_p(\bQ)$, where $X_p$ is the threefold defined by $(\ref{threefold-Xp-equations})$, and hence it satisfies $(A1)$ in Definition \ref{septuple-satisfies-hypothesis-FM-definition}. By $(\ref{the-second-parametrization-of-septuples-A-B-C-D-E-F-G-satisfying-Hypothesis-FM-equations})$ and $(\ref{the-representation-of-B-and-C-in-terms-of-D-and-A-in-the-second-parametrization-of-the-septuples-equations})$, we see that
\begin{align}
\label{AC-BD-equation}
AC - BD &= A(2pF_0^2\left(\delta_0 + \epsilon_0 - 2\mu A\right) + 2At_0F_0) - D(2pF_0^2\left(\delta_0 - \epsilon_0 - 2\mu D\right) + 2Dt_0F_0) \nonumber \\
        &= 4pF_0^2(D^2 - A^2)\mu + 2pF_0^2\left(A(\delta_0 + \epsilon_0) - D(\delta_0 - \epsilon_0)\right) + 2t_0F_0(A^2 - D^2) \nonumber \\
        &= 4p^2F_0^2G^2\mu + 2pF_0^2\left(\epsilon_0(A + D) + \delta_0(A - D) \right) - 2p t_0F_0G^2 \; \; (\text{since} \; A^2 - D^2 + pG^2 = 0)  \nonumber \\
        &= 2pF_0 \left(2pF_0G^2 \mu + F_0\left(\epsilon_0(A + D) + \delta_0(A - D) \right) - t_0G^2 \right) \nonumber \\
        &= 2pF_0Q^{\ast}.
\end{align}
Here,
\begin{align*}
Q^{\ast} = 2pF_0G^2 \mu + F_0\left(\epsilon_0(A + D) + \delta_0(A - D) \right) - t_0G^2.
\end{align*}
By $(\ref{the-second-parametrization-of-septuples-A-B-C-D-E-F-G-satisfying-Hypothesis-FM-equations})$, we see that
\begin{align*}
A + D &= p\lambda^2, \\
A - D &= -9\gamma^2.
\end{align*}
By the above identities, $(\ref{epsilon0-delta0-equations})$, and since $G = 3\lambda \gamma$, we deduce that
\begin{align}
\label{the-original-equation-of-Q-ast-equation}
Q^{\ast} &= 2pF_0G^2 \mu + F_0\left(p\lambda^2\epsilon_0 - 9\gamma^2\delta_0\right) - t_0G^2  \nonumber \\
         &= 2pF_0G^2 \mu + F_0\left(2p\lambda^2\epsilon_0 - 1\right) - t_0G^2     \nonumber \\
         &= (18p\lambda^2\gamma^2F_0)\mu + F_0\left( 2p\lambda^2\epsilon_0 - 1  \right) - 9\lambda^2\gamma^2t_0.
\end{align}

$\star$ \textit{Step 1. Choosing $t_0$.}

We define
\begin{align}
\label{t0-equation-in-terms-of-t1-and-F0-equation}
t_0 := -3\lambda \gamma F_0t_1,
\end{align}
where $t_1$ is an integer which will be chosen below in this step, and $F_0$ will be chosen in \textit{Step 2}. By $(\ref{the-original-equation-of-Q-ast-equation})$, one can write $Q^{\ast}$ in the form
\begin{align}
\label{Q-ast-first-definition-equation}
Q^{\ast} = F_0\left((18p\lambda^2\gamma^2)\mu + 2p\lambda^2\epsilon_0 - 1 + 27\lambda^3\gamma^3t_1\right) = F_0\left(P^{\ast}_1\mu + R^{\ast}_1 \right),
\end{align}
where
\begin{align}
\label{definition-of-P-ast-1-and-R-ast-1-equations}
\begin{cases}
P^{\ast}_1 := 18p\lambda^2\gamma^2, \\
R^{\ast}_1 := 2p\lambda^2\epsilon_0 - 1 + 27\lambda^3\gamma^3t_1.
\end{cases}
\end{align}
Since $\gcd(27\lambda^3\gamma^3, p) = 1$, there exist non-zero integers $t_2$ and $t_3$ such that
\begin{align}
\label{t2-t3-equations-in-terms-of-p-and-lambda-gamma-equation}
27\lambda^3\gamma^3t_2 - pt_3 = 1.
\end{align}
Take any non-zero integer $\pi$ such that $\pi$ is a \textit{quadratic residue} in $\bF_p^{\times}$, and let $t_5$ be any non-zero integer such that
\begin{align}
\label{definition-of-t5-in-terms-of-pi-and-p-lambda-gamma-equation}
\begin{cases}
t_5 \equiv \dfrac{\pi - 2\lambda^2\epsilon_0 - t_3}{27\lambda^3\gamma^3} \pmod{p}\\
t_5 \equiv 1 - t_3 \pmod{2}.
\end{cases}
\end{align}
Note that there are infinitely many such integers $\pi$ and $t_5$. Define
\begin{align}
\label{t1-definition-in-terms-of-t2-and-t5-final-definition-of-t1-equation}
t_1 := t_2 + pt_5,
\end{align}
and
\begin{align}
\label{t4-definition-in-terms-of-t3-and-t5-definition}
t_4 := t_3 + 27\lambda^3\gamma^3t_5.
\end{align}
By $(\ref{t2-t3-equations-in-terms-of-p-and-lambda-gamma-equation})$, $(\ref{t1-definition-in-terms-of-t2-and-t5-final-definition-of-t1-equation})$ and $(\ref{t4-definition-in-terms-of-t3-and-t5-definition})$, we see that
\begin{align}
\label{t1-t4-equations-in-terms-of-p-and-lambda-gamma-equation}
27\lambda^3\gamma^3t_1 - pt_4 = 1.
\end{align}

In summary, by $(\ref{t0-equation-in-terms-of-t1-and-F0-equation})$ and $(\ref{t1-definition-in-terms-of-t2-and-t5-final-definition-of-t1-equation})$, $t_0$ is of the form
\begin{align}
\label{t0-final-definition-in-terms-of-t2-t5-F0-equation}
t_0 = -3\lambda \gamma F_0(t_2 + pt_5),
\end{align}
where $t_2$ is an integer satisfying $(\ref{t2-t3-equations-in-terms-of-p-and-lambda-gamma-equation})$ and $t_5$ is any non-zero integer satisfying $(\ref{definition-of-t5-in-terms-of-pi-and-p-lambda-gamma-equation})$.

$\star$ \textit{Step 2. Choosing $F_0$.}

Define
\begin{align}
\label{the-definition-of-u-in-Step-2-choosing-F0-equation}
u :=
\begin{cases}
1 \; \; &\text{if $\lambda\gamma$ is a quadratic residue modulo $p$,} \\
0 \; \; &\text{if $\lambda\gamma$ is a quadratic non-residue modulo $p$.}
\end{cases}
\end{align}
We take $F_0$ to be any non-zero integer such that the following are true:
\begin{itemize}

\item [(F1)] $F_0 = 3^uF_1$, where $F_1$ is an integer such that $\gcd(F_1, 3) = \gcd(F_1, p) = 1$.

\item [(F2)] $p$ is a square in $\bQ_l^{\times}$ for each odd prime $l$ dividing $F_1$.

\end{itemize}

$\star$ \textit{Step 3. Defining $H$.}

We prove that $\tfrac{3\lambda\gamma}{F_0}$ is a quadratic residue in $\bF_p^{\times}$. Assume first that $\lambda\gamma$ is a quadratic residue in $\bF_p^{\times}$. By $(\ref{the-definition-of-u-in-Step-2-choosing-F0-equation})$ in \textit{Step 2}, we know that $u = 1$. By $(F1)$ in \textit{Step 2}, we see that
\begin{align*}
\dfrac{3\lambda\gamma}{F_0} = \dfrac{3\lambda\gamma}{3^{u}F_1} = \dfrac{\lambda\gamma}{F_1} .
\end{align*}
Hence, in order to prove that $\tfrac{3\lambda\gamma}{F_0}$ is a quadratic residue in $\bF_p^{\times}$, it suffices to prove that $F_1$ is a square in $\bF_p^{\times}$. Write $F_1$ in the form
\begin{align*}
F_1 = 2^{v_2(F_1)}\prod_{l \mid F_1} l^{v_l(F_1)},
\end{align*}
where the product is taken over all odd prime $l$ dividing $F_1$. Since $p \equiv 1 \pmod{8}$, we know that $2$ is a quadratic residue in $\bF_p^{\times}$. Hence it follows from $(F2)$ in \textit{Step 2} and the quadratic reciprocity law that
\begin{align*}
\left(\dfrac{F_1}{p}\right) &= \left(\dfrac{2^{v_2(F_1)}\prod_{l \mid F_1} l^{v_l(F_1)}}{p}\right) \\
&= \left(\dfrac{2^{v_2(F_1)}}{p}\right)\prod_{l \mid F_1}\left(\dfrac{l^{v_l(F_1)}}{p}\right) \\
&= \left(\dfrac{2}{p}\right)^{v_2(F_1)}\prod_{l \mid F_1}\left(\dfrac{p}{l}\right)^{v_l(F_1)} \\
&= 1.
\end{align*}
Thus $F_1$ is a quadratic residue in $\bF_p^{\times}$, and hence $\tfrac{3\lambda\gamma}{F_0}$ is a quadratic residue in $\bF_p^{\times}$.

Assume now that $\lambda\gamma$ is a quadratic non-residue in $\bF_p^{\times}$. By $(\ref{the-definition-of-u-in-Step-2-choosing-F0-equation})$ in \textit{Step 2}, we know that $u = 0$. Hence $3^u = 1$, and hence $F_0 = 3^uF_1 = F_1$. As was shown above, $F_1$ is a square in $\bF_p^{\times}$, and thus $F_0$ is a square in $\bF_p^{\times}$. Since $p \equiv 2 \pmod{3}$, we deduce that $p$ is a quadratic non-residue in $\bF_3^{\times}$. By the quadratic reciprocity law, we deduce that $3$ is a quadratic non-residue in $\bF_p^{\times}$. Thus $3\lambda\gamma$ is a square in $\bF_p^{\times}$, and therefore it follows that $\tfrac{3\lambda\gamma}{F_0}$ is a quadratic residue in $\bF_p^{\times}$.

Therefore, in any event, $\tfrac{3\lambda\gamma}{F_0}$ is a square in $\bF_p^{\times}$. We now define $H$ to be any non-zero integer such that
\begin{align}
\label{H-definition-equation}
H \equiv \left(\dfrac{3\lambda\gamma}{F_0}\right)^{1/2} \pmod{p}.
\end{align}

$\star$ \textit{Step 4. Defining $\mu$.}

By $(\ref{definition-of-P-ast-1-and-R-ast-1-equations})$ and $(\ref{t1-t4-equations-in-terms-of-p-and-lambda-gamma-equation})$, $R^{\ast}_1$ can be written in the form
\begin{align*}
R^{\ast}_1 = 2p\lambda^2\epsilon_0 - 1 + 27\lambda^3\gamma^3t_1 = 2p\lambda^2\epsilon_0 + pt_4 = pR^{\ast}_2,
\end{align*}
where
\begin{align}
\label{R-ast-2-definition-equation}
R^{\ast}_2 := 2\lambda^2\epsilon_0 + t_4.
\end{align}
By $(\ref{Q-ast-first-definition-equation})$ and $(\ref{definition-of-P-ast-1-and-R-ast-1-equations})$, we can write $Q^{\ast}$ in the form
\begin{align}
\label{Q-ast-second-definition-equation}
Q^{\ast} = F_0(P_1^{\ast}\mu + R_1^{\ast}) = pF_0(P^{\ast}_2\mu + R^{\ast}_2) = pF_0Q^{\ast}_1,
\end{align}
where
\begin{align}
\label{Q-ast-1-definition-equation}
Q^{\ast}_1 = P^{\ast}_2\mu + R^{\ast}_2
\end{align}
and
\begin{align}
\label{P-ast-2-definition-equation}
P^{\ast}_2 = 18\lambda^2\gamma^2.
\end{align}

We contend that $\gcd(3pP^{\ast}_2, R^{\ast}_2) = 1$. Since $\lambda$ and $\gamma$ is odd, it follows from $(\ref{definition-of-t5-in-terms-of-pi-and-p-lambda-gamma-equation})$, $(\ref{t4-definition-in-terms-of-t3-and-t5-definition})$ and $(\ref{R-ast-2-definition-equation})$ that
\begin{align}
\label{R-ast-2-congruent-to-1-modulo-2-equation}
R^{\ast}_2 \equiv t_4 \equiv t_3 + 27\lambda^3\gamma^3t_5 \equiv t_3 + t_5 \equiv 1 \pmod{2}.
\end{align}
Since $\gcd(p, \lambda) = 1$, it follows from $\ref{t1-t4-equations-in-terms-of-p-and-lambda-gamma-equation})$ and $(\ref{R-ast-2-definition-equation})$ that
\begin{align}
\label{R-ast-2-congruent-to-non-zero-modulo-lambda-equation}
R^{\ast}_2 \equiv t_4 \equiv  -\dfrac{1}{p} \not\equiv 0 \pmod{l}
\end{align}
for each odd prime $l$ dividing $\lambda$, and hence $\gcd(R^{\ast}_2, \lambda) = 1$. Since $\gcd(p, 3\gamma) = 1$, it follows from $(\ref{epsilon0-delta0-equations})$, $(\ref{t1-t4-equations-in-terms-of-p-and-lambda-gamma-equation})$ and $(\ref{R-ast-2-definition-equation})$ that
\begin{align}
\label{R-ast-2-mod-3-gamma-equals-non-zero-equation}
R^{\ast}_2 = 2\lambda^2\epsilon_0 + t_4 \equiv \dfrac{2}{p} + t_4 \equiv \dfrac{2}{p} - \dfrac{1}{p} \equiv \dfrac{1}{p} \not\equiv 0 \pmod{l}
\end{align}
for each odd prime $l$ dividing $3\gamma$, and hence $\gcd(R^{\ast}_2, 3\gamma) = 1$. By $(\ref{definition-of-t5-in-terms-of-pi-and-p-lambda-gamma-equation})$, $(\ref{t4-definition-in-terms-of-t3-and-t5-definition})$ and $(\ref{R-ast-2-definition-equation})$, we see that
\begin{align}
\label{R-ast-2-mod-p-equal-pi-equation}
R^{\ast}_2 = 2\lambda^2\epsilon_0 + t_4 = 2\lambda^2\epsilon_0 + t_3 + 27\lambda^3\gamma^3t_5 \equiv \pi \not\equiv 0 \pmod{p},
\end{align}
and hence $\gcd(R^{\ast}_2, p) = 1$. Since $P^{\ast}_2 = 18\lambda^2\gamma^2$, it follows from $(\ref{R-ast-2-congruent-to-1-modulo-2-equation})$, $(\ref{R-ast-2-congruent-to-non-zero-modulo-lambda-equation})$, $(\ref{R-ast-2-mod-3-gamma-equals-non-zero-equation})$ and $(\ref{R-ast-2-mod-p-equal-pi-equation})$ that
\begin{align*}
\gcd(3pP^{\ast}_2, R^{\ast}_2) = 1.
\end{align*}

We now define $\mu$. Since $\gcd(3pP^{\ast}_2, R^{\ast}_2) = 1$, it follows from the Dirichlet's theorem on arithmetic progressions that there are infinitely many integers $\mu_1$ such that $3pP^{\ast}_2\mu_1 + R^{\ast}_2$ is an odd prime. Take such an integer $\mu_1$, and define
\begin{align}
\label{the-definition-of-mu-in-terms-of-p-and-mu-1-equation}
\mu = 3p\mu_1.
\end{align}
By $(\ref{Q-ast-1-definition-equation})$, the choice of $\mu_1$ and the definition of $\mu$, we see that $Q^{\ast}_1$ is an odd prime.

$\star$ \textit{Step 5. Verifying $(A3)$.}

By $(\ref{the-second-parametrization-of-septuples-A-B-C-D-E-F-G-satisfying-Hypothesis-FM-equations})$ and $(\ref{lambda-gamma-divisibility-properties-equations})$, we see that $\gcd(A, D, G) = 1$ and $G \not\equiv 0 \pmod{p}$. Hence, it remains to verify that $E \not\equiv 0 \pmod{p}$. We see that
\begin{align}
\label{E-modulo-p-equation}
E &= F_0(2pF_0(\epsilon_0 + 9\mu\gamma^2) - 9\gamma^2t_0)(2F_0(\delta_0 - p\mu\lambda^2) + \lambda^2t_0) \; \; \; (\text{by} \; (\ref{the-second-parametrization-of-septuples-A-B-C-D-E-F-G-satisfying-Hypothesis-FM-equations})) \nonumber \\
  &\equiv -9\gamma^2F_0t_0 (\lambda^2t_0 + 2\delta_0F_0)   \nonumber \\
  &\equiv 27\lambda\gamma^3F_0^2t_2(-3\lambda^3\gamma F_0t_2 + 2\delta_0F_0)  \; \; \; (\text{by} \; (\ref{t0-final-definition-in-terms-of-t2-t5-F0-equation}))   \nonumber \\
  &\equiv -27\lambda\gamma^3F_0^3t_2\left(3\lambda^3\gamma t_2 - \dfrac{2}{9\gamma^2}\right) \; \; \; (\text{by} \; (\ref{epsilon0-delta0-equations})) \nonumber \\
  &\equiv -\dfrac{F_0^3}{\lambda^2}\left(\dfrac{1}{9\gamma^2} - \dfrac{2}{9\gamma^2}\right)  \; \; \; (\text{by} \; (\ref{t2-t3-equations-in-terms-of-p-and-lambda-gamma-equation})) \nonumber \\
  &\equiv \dfrac{F_0^3}{9\lambda^2\gamma^2} \not\equiv 0 \pmod{p}.
\end{align}
Hence $(A3)$ holds.

$\star$ \textit{Step 6. Verifying $(A4)$.}

Let $l$ be any odd prime such that $\gcd(l, 3) = \gcd(l, p) = 1$ and $l$ divides $AC - BD$. We will prove that $p$ is a square in $\bQ_l^{\times}$. Indeed, by $(\ref{AC-BD-equation})$ and $(\ref{Q-ast-second-definition-equation})$, we can write $AC - BD$ in the form
\begin{align}
\label{AC-BD-equation-in-terms-of-p-F0-Q-ast-1-equation}
AC - BD = 2pF_0Q^{\ast} = 2p^2F_0^2Q^{\ast}_1,
\end{align}
where $Q^{\ast}_1$ is given by $(\ref{Q-ast-1-definition-equation})$. Recall that by the choice of $\mu$ in \textit{Step 4}, $Q^{\ast}_1$ is an odd prime. If $l$ divides $F_0$, then it follows from $(F1)$ in \textit{Step 2} that $l$ divides $F_1$. Hence it follows from $(F2)$ in \textit{Step 2} that $p$ is a square in $\bQ_l^{\times}$. If $\gcd(l, F_0) = 1$, we see that since $Q^{\ast}_1$ is an odd prime, $l$ divides $AC - BD$ and $\gcd(l, 2pF_0) = 1$, it follows from $(\ref{AC-BD-equation-in-terms-of-p-F0-Q-ast-1-equation})$ that $l = Q^{\ast}_1$. Hence it suffices to show that $p$ is a square in $\bQ_{Q^{\ast}_1}^{\times}$, where $\bQ_{Q^{\ast}_1}$ denotes the $Q^{\ast}_1$-adic field.

By $(\ref{Q-ast-1-definition-equation})$, $(\ref{R-ast-2-mod-p-equal-pi-equation})$ and $(\ref{the-definition-of-mu-in-terms-of-p-and-mu-1-equation})$, we see that
\begin{align*}
Q^{\ast}_1 =  P^{\ast}_2\mu + R^{\ast}_2 =  3pP^{\ast}_2\mu_1 + R^{\ast}_2 \equiv R^{\ast}_2 \equiv \pi \pmod{p}.
\end{align*}
By the choice of $\pi$ in \textit{Step 1}, we know that $\pi$ is a quadratic residue in $\bF_p^{\times}$. Hence it follows from the last congruence that $\left(\dfrac{Q^{\ast}_1}{p}\right) = 1$, where $\left(\dfrac{\cdot}{\cdot}\right)$ denotes the Jacobi symbol. By the quadratic reciprocity law, we see that
\begin{align*}
\left(\dfrac{p}{Q^{\ast}_1}\right) = \left(\dfrac{Q^{\ast}_1}{p}\right) = 1,
\end{align*}
and thus $p$ is a square in $\bQ_{Q^{\ast}_1}^{\times}$. Hence $(A4)$ holds.

$\star$ \textit{Step 7. Verifying $(A5)$.}

Since $p \equiv 2 \pmod{3}$, $p$ is a quadratic non-residue in $\bF_3^{\times}$. By the quadratic reciprocity law, we deduce that $3$ is a quadratic non-residue in $\bF_p^{\times}$, and hence $-3$ is not a square in $\bF_p^{\times}$. Thus the group of all cube roots of unity in $\bF_p^{\times}$ is trivial. We will prove that $H$ satisfies $(A5)$, where $H$ is defined in \textit{Step 3}. Note that since the group of all cube roots of unity in $\bF_p^{\times}$ is trivial, the second condition in $(A5)$ is tantamount to saying that $A + BH^4$ is a quadratic non-residue in $\bF_p^{\times}$.

By $(\ref{the-second-parametrization-of-septuples-A-B-C-D-E-F-G-satisfying-Hypothesis-FM-equations})$, $(\ref{H-definition-equation})$ and $(\ref{E-modulo-p-equation})$, we see that
\begin{align*}
G - EH^6 &\equiv 3\lambda\gamma - \left(\dfrac{F_0^3}{9\lambda^2\gamma^2}\right)\left(\dfrac{27\lambda^3\gamma^3}{F_0^3}\right) \\
         &\equiv 3\lambda\gamma - 3\lambda\gamma \equiv 0 \pmod{p}.
\end{align*}
We see that
\begin{align*}
A + BH^4 &\equiv -\dfrac{9\gamma^2}{2} + 9\gamma^2t_0F_0\left(\dfrac{9\lambda^2\gamma^2}{F_0^2}\right) \; \; \; (\text{by $(\ref{the-second-parametrization-of-septuples-A-B-C-D-E-F-G-satisfying-Hypothesis-FM-equations})$ and $(\ref{H-definition-equation})$})  \\
         &\equiv -\dfrac{9\gamma^2}{2} + 9\gamma^2(-3\lambda\gamma F_0(t_2 + pt_5))F_0\left(\dfrac{9\lambda^2\gamma^2}{F_0^2}\right) \; \; \; (\text{by $(\ref{t0-final-definition-in-terms-of-t2-t5-F0-equation})$}) \\
         &\equiv -\dfrac{9\gamma^2}{2} - 27\lambda\gamma^3 F_0^2t_2\left(\dfrac{9\lambda^2\gamma^2}{F_0^2}\right)\\
         &\equiv -\dfrac{9\gamma^2}{2} - (27\lambda^3\gamma^3 t_2) 9\gamma^2 \\
         &\equiv -\dfrac{9\gamma^2}{2} - 9\gamma^2 \; \; \; (\text{by $(\ref{t2-t3-equations-in-terms-of-p-and-lambda-gamma-equation})$}) \\
         &\equiv 3\left(\dfrac{-1}{2}\right)(3\gamma)^2 \pmod{p}.
\end{align*}
Since $-1$ and $2$ are quadratic residues in $\bF_p^{\times}$ and $3$ is a quadratic non-residue in $\bF_p^{\times}$, we deduce that $3\left(\dfrac{-1}{2}\right)(3\gamma)^2$ is a quadratic non-residue in $\bF_p^{\times}$. Hence $A + BH^4$ is a quadratic non-residue in $\bF_p^{\times}$, and thus $(A5)$ holds.

$\star$ \textit{Step 8. Verifying $(A6)$.}

Since $\lambda \not\equiv 0 \pmod{3}$ and $p \equiv 2 \pmod{3}$, it follows from $(\ref{epsilon0-delta0-equations})$ that
\begin{align}
\label{epsilon0-modulo-3-equal-2-equation}
\epsilon_0 \equiv \dfrac{1}{p\lambda^2} \equiv \dfrac{1}{2} \equiv 2 \pmod{3}.
\end{align}
By $(\ref{the-second-parametrization-of-septuples-A-B-C-D-E-F-G-satisfying-Hypothesis-FM-equations})$ and $(\ref{t0-final-definition-in-terms-of-t2-t5-F0-equation})$, we see that
\begin{align*}
E &= F_0(2pF_0(\epsilon_0 + 9\mu\gamma^2) - 9\gamma^2t_0)(2F_0(\delta_0 - p\mu\lambda^2) + \lambda^2t_0) \\
  &= F_0(2pF_0(\epsilon_0 + 9\mu\gamma^2) - 9\gamma^2(-3\lambda \gamma F_0(t_2 + pt_5)))(2F_0(\delta_0 - p\mu\lambda^2) + \lambda^2(-3\lambda \gamma F_0(t_2 + pt_5))) \; \; (\text{by} \; (\ref{t0-final-definition-in-terms-of-t2-t5-F0-equation}))  \\
  &= F_0^3(2p(\epsilon_0 + 9\mu\gamma^2) + 27\lambda\gamma^3(t_2 + pt_5))(2(\delta_0 - p\mu\lambda^2) - 3\lambda^3\gamma(t_2 + pt_5)).
\end{align*}
Hence we deduce that
\begin{align}
\label{3-adic-valuation-of-E-equation}
v_3(E) = v_3(F_0^3) + v_3\left(2p(\epsilon_0 + 9\mu\gamma^2) + 27\lambda\gamma^3(t_2 + pt_5)\right) + v_3\left(2(\delta_0 - p\mu\lambda^2) - 3\lambda^3\gamma(t_2 + pt_5)\right).
\end{align}

By $(F1)$ in \textit{Step 2}, we see that
\begin{align}
\label{the-1st-summand-of-3-adic-valuation-of-E-equation}
v_3(F_0^3) = 3v_3(F_0) = 3v_3(3^uF_1) = 3v_3(3^u) + 3v_3(F_1) = 3u.
\end{align}
By $(\ref{epsilon0-modulo-3-equal-2-equation})$ and since $p \equiv 2 \pmod{3}$, we see that
\begin{align*}
2p(\epsilon_0 + 9\mu\gamma^2) + 27\lambda\gamma^3(t_2 + pt_5) \equiv 2p\epsilon_0 \equiv 8 \equiv 2 \pmod{3},
\end{align*}
and hence
\begin{align}
\label{the-2nd-summand-of-3-adic-valuation-of-E-equation}
v_3\left(2p(\epsilon_0 + 9\mu\gamma^2) + 27\lambda\gamma^3(t_2 + pt_5)\right) = 0.
\end{align}
By $(\ref{the-definition-of-mu-in-terms-of-p-and-mu-1-equation})$, we see that
\begin{align*}
\mu &= 3p\mu_1 \equiv 0 \pmod{3},
\end{align*}
and hence it follows from $(\ref{delta0-modulo-3-equal-2-equation})$ that
\begin{align*}
2(\delta_0 - p\mu\lambda^2) - 3\lambda^3\gamma(t_2 + pt_5) \equiv 2\delta_0 \equiv 4 \equiv 1 \pmod{3}.
\end{align*}
Thus we deduce that
\begin{align}
\label{the-3rd-summand-of-3-adic-valuation-of-E-equation}
v_3\left(2(\delta_0 - p\mu\lambda^2) - 3\lambda^3\gamma(t_2 + pt_5)\right) = 0.
\end{align}
Hence it follows from $(\ref{3-adic-valuation-of-E-equation})$, $(\ref{the-1st-summand-of-3-adic-valuation-of-E-equation})$, $(\ref{the-2nd-summand-of-3-adic-valuation-of-E-equation})$ and $(\ref{the-3rd-summand-of-3-adic-valuation-of-E-equation})$ that
\begin{align*}
v_3(E) = 3u,
\end{align*}
and thus we deduce from $(\ref{the-second-parametrization-of-septuples-A-B-C-D-E-F-G-satisfying-Hypothesis-FM-equations})$, $(\ref{lambda-gamma-divisibility-properties-equations})$ and $(\ref{the-definition-of-u-in-Step-2-choosing-F0-equation})$ that
\begin{align*}
v_3(E) - v_3(G) = 3u - v_3(3\lambda\gamma) = 3u - 1 \le 3 - 1 = 2 < 3n
\end{align*}
for any integer $n \ge 1$. Therefore $(A6)$ holds.

$\star$ \textit{Step 9. Verifying $(A7)$.}

By $(\ref{t0-final-definition-in-terms-of-t2-t5-F0-equation})$, we know that
\begin{align*}
t_0 = -3\lambda \gamma F_0(t_2 + pt_5) \equiv 0 \pmod{3}.
\end{align*}
Since $p \equiv 2 \pmod{3}$, $\lambda \not\equiv 0 \pmod{3}$ and $\mu = 3p\mu_1 \equiv 0 \pmod{3}$, it follows from $(\ref{the-second-parametrization-of-septuples-A-B-C-D-E-F-G-satisfying-Hypothesis-FM-equations})$, $(\ref{lambda-gamma-divisibility-properties-equations})$, $(\ref{delta0-modulo-3-equal-2-equation})$ and $(\ref{epsilon0-modulo-3-equal-2-equation})$ that
\begin{align*}
A + B &\equiv \dfrac{p\lambda^2}{2} + 2pF_0^2(\delta_0 - \epsilon_0) \\
      &\equiv 1 + 4F_0^2(2 - 2) \equiv 1 \not\equiv 0 \pmod{3}.
\end{align*}
By $(\ref{the-second-parametrization-of-septuples-A-B-C-D-E-F-G-satisfying-Hypothesis-FM-equations})$, we know that
\begin{align*}
G = 3\lambda\gamma \equiv 0 \pmod{3}.
\end{align*}
Hence $(A7)$ holds.

Therefore, by what we have shown above, our contention follows.

\end{proof}

\begin{remark}

Lemma \ref{the-infinitude-of-septuples-A-B-C-D-E-F-G-satisfying-A1-A3-A4-A5-A6-A7-lemma} shows that for a prime $p$ with $p \equiv 1 \pmod{8}$ and $p \equiv 2 \pmod{3}$ and an integer $n \ge 1$, there are infinitely many septuples $(A, B, C, D, E, F, G)$ defined by $(\ref{the-second-parametrization-of-septuples-A-B-C-D-E-F-G-satisfying-Hypothesis-FM-equations})$ that satisfy $(A1), (A3), (A4), (A5), (A6)$ and $(A7)$ with respect to the couple $(p, n)$. It is not difficult to choose $n$ sufficiently large so that the septuples $(A, B, C, D, E, F, G)$ in Lemma \ref{the-infinitude-of-septuples-A-B-C-D-E-F-G-satisfying-A1-A3-A4-A5-A6-A7-lemma} satisfy Hypothesis FM with respect to the couple $(p, n)$. Such septuples produce infinitely many generalized Mordell curves of degree $12n$ that have no $\bQ$-rational points. More precisely, we prove the following.

\end{remark}

\begin{corollary}
\label{Corollary-the-2nd-parametrization-produces-generalized-Mordell-curves-of-big-degree-having-no-rational-points}

Let $p$ be a prime such that $p \equiv 1 \pmod{8}$ and $p \equiv 2 \pmod{3}$. There exists an infinite set $\fC_p \subseteq \bZ^7$ consisting of the septuples $(A, B, C, D, E, F, G) \in \bZ^7$ defined by $(\ref{the-second-parametrization-of-septuples-A-B-C-D-E-F-G-satisfying-Hypothesis-FM-equations})$ that satisfy $(A1), (A3), (A4), (A5), (A7)$ in Definition \ref{septuple-satisfies-hypothesis-FM-definition}. Furthermore, for each septuple $\cT := (A, B, C, D, E, F, G) \in \fC_p$, there exists a positive real number $n_{E, G} \ge 1$ such that for any integer $n > n_{E, G}$, the septuple $\cT$ satisfies Hypothesis FM with respect to $(p, n)$ in the sense of Definition \ref{septuple-satisfies-hypothesis-FM-definition}, and the smooth projective model $\cC_{n, \cT}$ of the affine curve defined by
\begin{align*}
\cC_{n, \cT} : pz^2 = E^2x^{12n} - G^2
\end{align*}
satisfies $\cC_{n, \cT}(\bA_{\bQ})^{\Br} = \emptyset$.

\end{corollary}

\begin{proof}

Let $\fC_p$ be the set of the septuples $(A, B, C, D, E, F, G) \in \bZ^7$ satisfying the following two conditions:
\begin{itemize}

\item [(i)] $(A, B, C, D, E, F, G)$ satisfies $(\ref{the-second-parametrization-of-septuples-A-B-C-D-E-F-G-satisfying-Hypothesis-FM-equations})$ and $(A1), (A3), (A4), (A5), (A7)$ in Definition \ref{septuple-satisfies-hypothesis-FM-definition}, and

\item [(ii)] for any $n \ge 1$, $(A, B, C, D, E, F, G)$ satisfies $(A6)$ in Definition \ref{septuple-satisfies-hypothesis-FM-definition} with respect to the couple $(p, n)$.

\end{itemize}
By Lemma \ref{the-infinitude-of-septuples-A-B-C-D-E-F-G-satisfying-A1-A3-A4-A5-A6-A7-lemma}, we know that $\fC_p$ is of infinite cardinality.

Take any septuple $\cT := (A, B, C, D, E, F, G) \in \fC_p$, and let $\cS_{E, G}$ be the set of odd primes $l$ such that $\gcd(l, 3) = \gcd(l, p) = 1$, $l$ divides $E$ and $p$ is not a square in $\bQ_l^{\times}$. Set
\begin{align*}
n^{\ast}_{E, G} :=
\begin{cases}
\max_{l \in \cS_{E, G}}\left(\dfrac{v_l(E) - v_l(G)}{6}\right) \; \; \; &\text{if $\cS_{E, G} \ne \emptyset$,} \\
1 \; \; \; &\text{if $\cS_{E, G} = \emptyset$,}
\end{cases}
\end{align*}
and define
\begin{align*}
n_{E, G} := \max\left(1, n^{\ast}_{E, G}\right).
\end{align*}
Let $n$ be any integer such that $n > n_{E, G}$. Let $l$ be any odd prime such that $\gcd(l, 3) = \gcd(l, p) = 1$ and $l$ divides $E$. If $p$ is not a square in $\bQ_l^{\times}$, then $l$ belongs to $\cS_{E, G}$. Hence, by definition of $n_{E, G}$, we see that
\begin{align*}
6n > 6n_{E, G} \ge 6n^{\ast}_{E, G} \ge v_l(E) - v_l(G),
\end{align*}
which proves that the septuple $\cT$ satisfies $(A2)$ with respect to $(p, n)$, and thus it satisfies Hypothesis FM with respect to the couple $(p, n)$. The last assertion follows immediately from Theorem \ref{non-existence-of-rational-points-for-generalized-mordell-curves-using-rational-points-on-threefold-Xp-theorem}.

\end{proof}

\begin{remark}

Let $p$ be a prime such that $p \equiv 1 \pmod{8}$ and $p \equiv 2 \pmod{3}$, and let $n$ be a positive integer. In order to use the septuples $(A, B, C, D, E, F, G)$ arising from Lemma \ref{the-infinitude-of-septuples-A-B-C-D-E-F-G-satisfying-A1-A3-A4-A5-A6-A7-lemma} to produce generalized Mordell curves of degree $12n$ and generalized Fermat curves of signature $(12n, 12n, 12n)$ violating the Hasse principle explained by the Brauer-Manin obstruction following the approach described in Sections \ref{certain-generalized-Mordell-curves-violate-the-Hasse-principle-section} and \ref{certain-generalized-Fermat-curves-violate-the-Hasse-principle-section}, we need to show that there exist infinitely many septuples $(A, B, C, D, E, F, G)$ in Lemma \ref{the-infinitude-of-septuples-A-B-C-D-E-F-G-satisfying-A1-A3-A4-A5-A6-A7-lemma} satisfying Hypothesis FM with respect to the couple $(p, n)$. It suffices to show that there are infinitely many septuples $(A, B, C, D, E, F, G)$ in Lemma \ref{the-infinitude-of-septuples-A-B-C-D-E-F-G-satisfying-A1-A3-A4-A5-A6-A7-lemma} satisfying $(A2)$ with respect to $(p, n)$, that is,
\begin{align*}
v_l(E) - v_l(G) < 6n,
\end{align*}
where $l$ is any odd prime such that  $\gcd(l, 3) = \gcd(l, p) = 1$, $l$ divides $E$, and $p$ is not a square in $\bQ_l^{\times}$. We will use \textit{Schinzel's Hypothesis H} to prove that there should be infinitely many septuples $(A, B, C, D, E, F, G)$ in Lemma \ref{the-infinitude-of-septuples-A-B-C-D-E-F-G-satisfying-A1-A3-A4-A5-A6-A7-lemma} that satisfy Hypothesis FM with respect to the couple $(p, n)$.

\end{remark}

We recall the statement of Schinzel's Hypothesis H (see \cite{schinzel-sierpinski}).

\begin{conjecture}
\label{schinzel-hypothesis-H-conjecture}
$(\text{Schinzel's Hypothesis H})$

Let $F_1(x), F_2(x), \ldots, F_n(x)$ be nonconstant polynomials in $\bZ[x]$ such that the polynomials $F_1(x)$, $F_2(x), \ldots, F_{n - 1}(x)$ and $F_n(x)$ have positive leading coefficients and irreducible over $\bQ$. Assume that the polynomial
\begin{align*}
F(x) := \prod_{i = 1}^n F_i(x) \in \bZ[x]
\end{align*}
has no fixed prime divisor, that is, there is no prime $q$ dividing $F(m)$ for all integers $m$. Then there are infinitely many arbitrarily large positive integers $x$ such that $F_1(x), F_2(x), \ldots, F_{n - 1}(x)$ and $F_n(x)$ are simultaneously primes.

\end{conjecture}

\begin{corollary}
\label{Corollary-Schinzel-Hypothesis-H-produces-septuples-satisfying-Hypothesis-FM}

Let $p$ be a prime such that $p \equiv 1 \pmod{8}$ and $p \equiv 2 \pmod{3}$. Let $\fC_p$ be the set defined in Corollary \ref{Corollary-the-2nd-parametrization-produces-generalized-Mordell-curves-of-big-degree-having-no-rational-points}. Assume Schinzel's Hypothesis H. Let $n$ be a positive integer. Then there exist infinitely many septuples $(A, B, C, D, E, F, G)$ in $\fC_p$ that satisfy Hypothesis FM with respect to the couple $(p, n)$.

\end{corollary}

\begin{proof}

We maintain the same notation as in the proof of Lemma \ref{the-infinitude-of-septuples-A-B-C-D-E-F-G-satisfying-A1-A3-A4-A5-A6-A7-lemma}. Using exactly the same words and repeating the same arguments from the beginning of the proof of Lemma \ref{the-infinitude-of-septuples-A-B-C-D-E-F-G-satisfying-A1-A3-A4-A5-A6-A7-lemma} to the end of \textit{Step 3} in the proof of Lemma \ref{the-infinitude-of-septuples-A-B-C-D-E-F-G-satisfying-A1-A3-A4-A5-A6-A7-lemma}, we let $\lambda, \gamma, \epsilon_0, \delta_0$ as in the proof of Lemma \ref{the-infinitude-of-septuples-A-B-C-D-E-F-G-satisfying-A1-A3-A4-A5-A6-A7-lemma}, and define $t_0$, $F_0$ and $H$ as in \textit{Steps 1, 2} and \textit{3} in the proof of Lemma \ref{the-infinitude-of-septuples-A-B-C-D-E-F-G-satisfying-A1-A3-A4-A5-A6-A7-lemma}. Let $(A, B, C, D, E, F, G)$ be the septuple of integers defined by $(\ref{the-second-parametrization-of-septuples-A-B-C-D-E-F-G-satisfying-Hypothesis-FM-equations})$, where $\mu$ will be determined now.

By $(\ref{t0-equation-in-terms-of-t1-and-F0-equation})$, we see that
\begin{align*}
E &= F_0(2pF_0(\epsilon_0 + 9\mu\gamma^2) - 9\gamma^2t_0)(2F_0(\delta_0 - p\mu\lambda^2) + \lambda^2t_0)  \\
&= F_0(2pF_0(\epsilon_0 + 9\mu\gamma^2) - 9\gamma^2(-3\lambda\gamma F_0t_1))(2F_0(\delta_0 - p\mu\lambda^2) + \lambda^2(-3\lambda\gamma F_0t_1)) \\
&= F_0^3(2p(\epsilon_0 + 9\mu\gamma^2) + 27\lambda\gamma^3t_1)(2(\delta_0 - p\mu\lambda^2) - 3\lambda^3\gamma t_1) \\
&= -F_0^3((18p\gamma^2)\mu + 2p\epsilon_0 + 27\lambda\gamma^3t_1)((2p\lambda^2)\mu - 2\delta_0 + 3\lambda^3\gamma t_1),
\end{align*}
and hence
\begin{align}
\label{Definition-Equation-E-in-terms-of-E1-ast-and-E2-ast-in-using-Schinzel-Hypothesis-H-to-produce-septuples-satisfying-Hypothesis-FM}
E = -F_0^3E_1^{\ast}E_2^{\ast},
\end{align}
where
\begin{align*}
E_1^{\ast} := (18p\gamma^2)\mu + 2p\epsilon_0 + 27\lambda\gamma^3t_1
\end{align*}
and
\begin{align*}
E_2^{\ast} := (2p\lambda^2)\mu - 2\delta_0 + 3\lambda^3\gamma t_1.
\end{align*}
Let
\begin{align}
\label{Definition-mu-in-terms-of-p-and-mu-1-in-Schinzel-Hypothesis-H-producing-septuples-satisfying-Hypothesis-FM}
\mu := 3p\mu_1,
\end{align}
where $\mu_1$ will be determined below. We see that
\begin{align*}
E_1^{\ast} := (54p^2\gamma^2)\mu_1 + 2p\epsilon_0 + 27\lambda\gamma^3t_1
\end{align*}
and
\begin{align*}
E_2^{\ast} := (6p^2\lambda^2)\mu_1 - 2\delta_0 + 3\lambda^3\gamma t_1.
\end{align*}

Write $t_1 = 2^vt_1^{\ast}$, where $v$ is a non-negative integer and $t_1^{\ast}$ is an odd integer. We contend that $t_1$ is an even integer, that is, $v \ge 1$. By $(\ref{definition-of-t5-in-terms-of-pi-and-p-lambda-gamma-equation})$, $(\ref{t4-definition-in-terms-of-t3-and-t5-definition})$, $(\ref{t1-t4-equations-in-terms-of-p-and-lambda-gamma-equation})$ and since $\lambda$ and $\gamma$ are odd, we see that
\begin{align*}
1 = 27\lambda^3\gamma^3t_1 - pt_4 \equiv t_1 + t_4 \equiv t_1 + t_3 + 27\lambda^3\gamma^3t_5 \equiv t_1 + t_3 + t_5 \equiv t_1 + 1 \pmod{2},
\end{align*}
and hence
\begin{align*}
t_1 \equiv 0 \pmod{2}.
\end{align*}
Thus $t_1$ is an even integer, and therefore $v \ge 1$.

We see that
\begin{align}
\label{Equation-E-1-ast-equals-2-E-1}
E_1^{\ast} = 2E_1
\end{align}
and
\begin{align}
\label{Equation-E-2-ast-equals-2-E-2}
E_2^{\ast} = 2E_2,
\end{align}
where
\begin{align}
\label{Definition-Equation-E-1-in-using-Schinzel-Hypothesis-H-to-produce-septuples-satisfying-Hypothesis-FM}
E_1 := (27p^2\gamma^2)\mu_1 + p\epsilon_0 + 27\lambda\gamma^3 2^{v - 1}t_1^{\ast}
\end{align}
and
\begin{align}
\label{Definition-Equation-E-2-in-using-Schinzel-Hypothesis-H-to-produce-septuples-satisfying-Hypothesis-FM}
E_2 := (3p^2\lambda^2)\mu_1 - \delta_0 + 3\lambda^3\gamma 2^{v - 1}t_1^{\ast}.
\end{align}

Let $Q_1^{\ast}$ be the integer defined by $(\ref{Q-ast-1-definition-equation})$ in \textit{Step 4} in the proof of Lemma \ref{the-infinitude-of-septuples-A-B-C-D-E-F-G-satisfying-A1-A3-A4-A5-A6-A7-lemma}. We can write $Q_1^{\ast}$ in the form
\begin{align}
\label{Definition-Equation-Q-1-ast-in-terms-of-mu-1-in-using-Schinzel-Hypothesis-H-to-produce-septuples-satisfying-Hypothesis-FM}
Q_1^{\ast} = P_2^{\ast}\mu + R_2^{\ast} = 3pP_2^{\ast}\mu_1 + R_2^{\ast},
\end{align}
where we recall from \textit{Step 4} in the proof of Lemma \ref{the-infinitude-of-septuples-A-B-C-D-E-F-G-satisfying-A1-A3-A4-A5-A6-A7-lemma} that
\begin{align*}
P_2^{\ast} = 18\lambda^2\gamma^2
\end{align*}
and
\begin{align*}
R_2^{\ast} = 2\lambda^2\epsilon_0 + t_4.
\end{align*}

Viewing $\mu_1$ as a variable, we see that $E_1, E_2$ and $Q_1^{\ast}$ are polynomials with integral coefficients in the variable $\mu_1$, i.e., $E_1, E_2$ and $Q_1^{\ast}$ belong to $\bZ[\mu_1]$. Upon assuming Schinzel's Hypothesis H, we will show that there should be infinitely many arbitrarily large positive integers $\mu_1$ such that $E_1, E_2$ and $Q_1^{\ast}$ are simultaneously primes. In order to apply Schinzel's Hypothesis H, we need to prove that the polynomial $\Psi(\mu_1) \in \bZ[\mu_1]$ defined by
\begin{align}
\label{Equation-Psi-mu-1-satisfies-Schinzel-Hypothesis-H}
\Psi(\mu_1) := E_1E_2Q_1^{\ast}
\end{align}
has no fixed divisors, that is, there is no prime $q$ dividing $\Psi(m)$ for every integer $m$. To prove the latter, it suffices to show that
\begin{align*}
\gcd\left(27p^2\gamma^2, p\epsilon_0 + 27\lambda\gamma^3 2^{v - 1}t_1^{\ast}\right) = 1,
\end{align*}
\begin{align*}
\gcd\left(3p^2\lambda^2, -\delta_0 + 3\lambda^3\gamma 2^{v - 1}t_1^{\ast} \right) = 1,
\end{align*}
and
\begin{align*}
\gcd\left(3pP_2^{\ast}, R_2^{\ast}\right) = 1.
\end{align*}
Using the same arguments as in \textit{Step 4} in the proof of Lemma \ref{the-infinitude-of-septuples-A-B-C-D-E-F-G-satisfying-A1-A3-A4-A5-A6-A7-lemma}, we see that
\begin{align}
\label{Q-1-ast-satisfies-conditions-in-Schinzel-Hypothesis-H}
\gcd\left(3pP_2^{\ast}, R_2^{\ast}\right) = 1.
\end{align}

We now prove that
\begin{align*}
\gcd\left(27p^2\gamma^2, p\epsilon_0 + 27\lambda\gamma^3 2^{v - 1}t_1^{\ast}\right) = 1.
\end{align*}
Indeed, by $(\ref{t1-t4-equations-in-terms-of-p-and-lambda-gamma-equation})$, we see that
\begin{align*}
2^{v - 1}t_1^{\ast} = \dfrac{t_1}{2} \equiv \dfrac{1}{54\lambda^3\gamma^3} \not\equiv 0 \pmod{p}.
\end{align*}
Since $\gcd(p, \lambda) = \gcd(p, \gamma) = \gcd(p, 3) = 1$, we deduce that
\begin{align*}
p\epsilon_0 + 27\lambda\gamma^3 2^{v - 1}t_1^{\ast} \equiv 27\lambda\gamma^3 2^{v - 1}t_1^{\ast} \not\equiv 0 \pmod{p}.
\end{align*}
Since $\gcd(\lambda, 3\gamma) = 1$, we see that if $l$ is any prime dividing $3\gamma$, then it follows from $(\ref{epsilon0-delta0-equations})$ that
\begin{align*}
p\epsilon_0 \equiv \dfrac{1}{\lambda^2} \not\equiv 0 \pmod{l},
\end{align*}
and hence
\begin{align*}
p\epsilon_0 + 27\lambda\gamma^3 2^{v - 1}t_1^{\ast} \equiv p\epsilon_0 \not\equiv 0 \pmod{l}.
\end{align*}
Thus we deduce that
\begin{align}
\label{E-1-satisfies-conditions-in-Schinzel-Hypothesis-H}
\gcd\left(27p^2\gamma^2, p\epsilon_0 + 27\lambda\gamma^3 2^{v - 1}t_1^{\ast}\right) = 1.
\end{align}

We prove that
\begin{align*}
\gcd\left(3p^2\lambda^2, -\delta_0 + 3\lambda^3\gamma 2^{v - 1}t_1^{\ast} \right) = 1.
\end{align*}
Indeed, by $(\ref{epsilon0-delta0-equations})$ and $(\ref{t1-t4-equations-in-terms-of-p-and-lambda-gamma-equation})$, we see that
\begin{align*}
-\delta_0 + 3\lambda^3\gamma 2^{v - 1}t_1^{\ast} = -\delta_0 + 3\lambda^3\gamma \left(\dfrac{t_1}{2}\right) &\equiv -\dfrac{1}{9\gamma^2} + 3\lambda^3\gamma\left(\dfrac{1}{54\lambda^3\gamma^3}\right) \\
&\equiv -\dfrac{1}{9\gamma^2} + \dfrac{1}{18\gamma^2} \\
&\equiv -\dfrac{1}{18\gamma^2} \not\equiv 0 \pmod{p}.
\end{align*}
By $(\ref{delta0-modulo-3-equal-2-equation})$, we see that
\begin{align*}
-\delta_0 + 3\lambda^3\gamma 2^{v - 1}t_1^{\ast} \equiv -\delta_0 \equiv -2 \equiv 1 \pmod{3}.
\end{align*}
By $(\ref{epsilon0-delta0-equations})$ and since $\gcd(\lambda, 3\gamma) = 1$, we see that if $l$ is any prime dividing $\lambda$, then
\begin{align*}
-\delta_0 + 3\lambda^3\gamma 2^{v - 1}t_1^{\ast} \equiv -\delta_0 \equiv -\dfrac{1}{9\gamma^2} \not\equiv 0 \pmod{l}.
\end{align*}
Therefore we deduce that
\begin{align}
\label{E-2-satisfies-conditions-in-Schinzel-Hypothesis-H}
\gcd\left(3p^2\lambda^2, -\delta_0 + 3\lambda^3\gamma 2^{v - 1}t_1^{\ast} \right) = 1.
\end{align}

By $(\ref{Q-1-ast-satisfies-conditions-in-Schinzel-Hypothesis-H})$, $(\ref{E-1-satisfies-conditions-in-Schinzel-Hypothesis-H})$, $(\ref{E-2-satisfies-conditions-in-Schinzel-Hypothesis-H})$, we see that the polynomial $\Psi$ defined by $(\ref{Equation-Psi-mu-1-satisfies-Schinzel-Hypothesis-H})$ has no fixed prime divisor. On the other hand, $E_1, E_2$ and $Q_1^{\ast}$ have positive leading coefficients and irreducible over $\bQ$. Hence Schinzel's Hypothesis H expects that there should be infinitely many arbitrarily large positive integers $\mu_1$ such that $E_1, E_2$ and $Q_1^{\ast}$ are simultaneously primes. Take such a positive integer $\mu_1$, and define $\mu$ by $(\ref{Definition-mu-in-terms-of-p-and-mu-1-in-Schinzel-Hypothesis-H-producing-septuples-satisfying-Hypothesis-FM})$. We see that the choice of $\mu$ here is \textit{compatible} with that of $\mu$ in \textit{Step 4} in the proof of Lemma \ref{the-infinitude-of-septuples-A-B-C-D-E-F-G-satisfying-A1-A3-A4-A5-A6-A7-lemma}. More precisely, in \textit{Step 4} in the proof of Lemma \ref{the-infinitude-of-septuples-A-B-C-D-E-F-G-satisfying-A1-A3-A4-A5-A6-A7-lemma}, we chose $\mu$ so that $\mu = 3p\mu_1$ and $Q_1^{\ast} = 3pP_2^{\ast}\mu_1 + R_2^{\ast}$ is an odd prime for some integer $\mu_1$, and it is not difficult to realize that the choice of $\mu$ here satisfies these conditions. Repeating the same arguments as in \textit{Steps 5, 6, 7, 8} and \textit{9} in the proof of Lemma \ref{the-infinitude-of-septuples-A-B-C-D-E-F-G-satisfying-A1-A3-A4-A5-A6-A7-lemma}, we deduce that the septuple $(A, B, C, D, E, F, G)$ satisfies $(A1)$ and $(A3)-(A7)$ with respect to the couple $(p, n)$. It remains to prove that $(A, B, C, D, E, F, G)$ satisfies $(A2)$ with respect to the couple $(p, n)$.

By $(\ref{Definition-Equation-E-in-terms-of-E1-ast-and-E2-ast-in-using-Schinzel-Hypothesis-H-to-produce-septuples-satisfying-Hypothesis-FM})$, $(\ref{Equation-E-1-ast-equals-2-E-1})$ and $(\ref{Equation-E-2-ast-equals-2-E-2})$, we see that
\begin{align*}
E = -4F_0^3E_1E_2.
\end{align*}
Let $l$ be any odd prime such that $\gcd(l, 3) = \gcd(l, p) = 1$ and $l$ divides $E$. Then either $l$ divides $F_0$ or $\gcd(l, F_0) = 1$ and $l$ divides $E_1E_2$. If $l$ divides $F_0$, then it follows from $(F1)$ and $(F2)$ in \textit{Step 3} in the proof of Lemma \ref{the-infinitude-of-septuples-A-B-C-D-E-F-G-satisfying-A1-A3-A4-A5-A6-A7-lemma} that $p$ is a square in $\bQ_l^{\times}$. If $l$ does not divide $F_0$ and $l$ divides $E_1E_2$, then since $E_1, E_2$ are odd primes, it follows that
\begin{align*}
v_l(E) = v_l(-4F_0^3E_1E_2) = v_l(E_1E_2) \le 2.
\end{align*}
Thus we deduce that
\begin{align*}
v_l(E) - v_l(G) \le 2 + 0 = 2 < 6n.
\end{align*}
Thus $(A, B, C, D, E, F, G)$ satisfies $(A2)$ with respect to the couple $(p, n)$. Since $(A, B, C, D, E, F, G)$ is defined by $(\ref{the-second-parametrization-of-septuples-A-B-C-D-E-F-G-satisfying-Hypothesis-FM-equations})$, it follows that $(A, B, C, D, E, F, G)$ belongs to $\fC_p$ and satisfies Hypothesis FM with respect to the couple $(p, n)$. Hence our contention follows.

\end{proof}

\section*{Acknowledgements}

I would like to thank Mike Bennett and Bjorn Poonen for their comments. I thank Dinesh Thakur for his interest in this work. I am grateful to Romyar Sharifi for funding me under his NSF Grant DMS-0901526 in the fall of 2011 when part of this work was written. I was supported by a postdoctoral fellowship in the Department of Mathematics at University of British Columbia.


\begin{thebibliography}{179}


\bibitem{bennett-skinner} {\sc M.A. Bennett and C.M. Skinner}, \emph{Ternary Diophantine equations via Galois representations and modular forms}, Canad. J. Math. {\bf56} (2004), no. {\bf1}, pp. 23--54.



\bibitem{cohen} {\sc H. Cohen}, \emph{Number Theory, Volume I: Tools and Diophantine equations}, Graduate Texts in Math. {\bf239}, Springer-Verlag (2007).



\bibitem{coray-manoil} {\sc D. Coray and C. Manoil}, \emph{On large Picard groups and the Hasse principle for curves and $K3$ surfaces}, Acta. Arith. {\bf76} (1996), pp. 165--189.



\bibitem{halberstadt-kraus} {\sc E. Halberstadt and A. Kraus}, \emph{Courbes de Fermat: r\'esultats et probl\`emes}, J. Reine Angew. Math {\bf548} (2002), pp. 167--234.


\bibitem{ivorra-kraus} {\sc W. Ivorra  and A. Kraus}, \emph{Quelques r\'esultats sur les \'equations $ax^p + by^p = cz^2$}, Canad. J. Math. {\bf58} (2006), no. {\bf1}, pp. 115--153.












\bibitem{lind} {\sc C.E. Lind}, \emph{Untersuchungen \"uber die rationalen Punkte der ebenen kubischen Kurven vom Geschlecht Eins}, Thesis, University of Uppasala (1940).







\bibitem{manin} {\sc Yu.I. Manin}, \emph{Le groupe de Brauer-Grothendieck en g\'eom\'etrie Diophantienne}, Actes du Congr\`es International des Math\'ematiciens, Nice (1970), pp. 401--411.










\bibitem{poonen3} {\sc B. Poonen}, \emph{Rational points on varieties}, Available at \url{http://www-math.mit.edu/~poonen/papers/Qpoints.pdf}, (2008).






\bibitem{reichardt} {\sc H. Reichardt}, \emph{Einige im Kleinen \"uberall l\"osbre, im Grossen unl\"osbare diophantische Gleichungen}, J. Reine Angew. Math. {\bf184} (1942), pp. 12--18.





\bibitem{schinzel1} {\sc A. Schinzel}, \emph{Remarks on the paper ``Sur certaines hypoth\`eses concernant les nombres premiers"}, Acta. Arith. {\bf7} (1961/1962), 1--8.


\bibitem{schinzel-sierpinski} {\sc A. Schinzel and W. Sierpi\'nski}, \emph{Sur certaines hypoth\`eses concernant les nombres premiers}, Acta. Arith. {\bf4} (1958), 185--208; corrig\'e ibid. {\bf5} (1958).





\bibitem{selmer} {\sc E.S. Selmer}, \emph{The Diophantine equation $ax^3 + by^3 + cz^3 = 0$}, Acta. Math. {\bf85} (1951), pp. 203--362.














\bibitem{skorobogatov} {\sc A.N. Skorobogatov}, \emph{Torsors and rational points}, CTM {\bf144}. Cambridge Univ. Press (2001).











\bibitem{wiles} {\sc A. Wiles}, \emph{Modular elliptic curves and Fermat's last theorem}, Ann. of Math. (2) {\bf141} (1995), no.3, 443--551.







\end{thebibliography}
\end{document}